\documentclass[aos]{imsart}

\RequirePackage{amsthm,amsmath,amsfonts,amssymb, mathrsfs, physics}
\usepackage{xcolor}
\usepackage{enumitem}
\RequirePackage{natbib}
\RequirePackage[colorlinks,citecolor=blue,urlcolor=blue]{hyperref}
\RequirePackage{graphicx}
\usepackage{subfigure}

\usepackage{multirow,bm}

\makeatletter
\let\oldabs\abs
\def\abs{\@ifstar{\oldabs}{\oldabs*}}
\let\oldnorm\norm
\def\norm{\@ifstar{\oldnorm}{\oldnorm*}}
\makeatother

\newcommand{\HH}{\mathsf{H}}
\newcommand{\bu}{\bm{u}}
\newcommand{\ddd}{\mathrm{dual}}
\newcommand{\buu}{\widehat{\bm{u}}}
\newcommand{\uu}{\widehat{{u}}}

\newcommand{\bv}{\boldsymbol{v}}

\newcommand{\bw}{\boldsymbol{w}}

\newcommand{\bx}{\boldsymbol{x}}

\newcommand{\A}{\mathcal{A}}
\newcommand{\D}{\mathcal{D}}
\newcommand{\W}{\mathcal{W}}
\newcommand{\cP}{\mathcal{P}}
\newcommand{\RE}{\mathrm{RE}}
\newcommand{\h}{\mathcal H}
\newcommand{\sym}{\mathrm{sym}}
\newcommand{\vol}{\mathrm{vol}}
\newcommand{\leave}{\mathrm{leave}}
\newcommand{\full}{\mathrm{full}}
\newcommand{\ptl}{\mathrm{partial}}

\newcommand{\cc}{\mathfrak{r}}
\newcommand{\mix}{\mathfrak{s}}
\newcommand{\FI}{\mathcal G_{n,1}}
\newcommand{\cH}{\mathcal H}
\newcommand{\cJ}{\mathcal H}
\newcommand{\US}{U}
\newcommand{\st}{s}
\newcommand{\E}{E_n}
\newcommand{\FII}{\mathcal G_{n,2}}
\newcommand{\N}{N}
\newcommand{\ndeg}{N}

\newcommand{\Lu}{\mu}

\newcommand{\I}{\mathcal{I}}
\newcommand{\PP}{\mathbb{P}}
\newcommand{\GG}{\mathsf{G}}
\DeclareMathOperator*{\argmax}{arg\,max}
\DeclareMathOperator*{\argmin}{arg\,min}
\startlocaldefs
\numberwithin{equation}{section}
\theoremstyle{plain}
\newtheorem{thm}{Theorem}[section]

\newtheorem{Lemma}{Lemma}[section]
\newtheorem{Proposition}{Proposition}[section]
\newtheorem{claim}{Claim}[section]

\theoremstyle{remark}
\newtheorem{Definition}{Definition}
\newtheorem{condition}{Condition}
\newtheorem{Assumption}{Assumption}

\newtheorem{remark}{Remark}

\renewcommand{\d}[1]{\ensuremath{\operatorname{d}\!{#1}}}
\newcommand{\e}[1]{\textnormal{exp}(#1)}
\newcommand{\mm}{M}
\newcommand{\ms}{m}
\newcommand{\mGamma}{\mathcal M}
\newcommand{\mZ}{\mathcal Z}

\definecolor{brickred}{rgb}{0.58, 0.22, 0.10}

\newcommand{\plainref}[1]{{{\color{blue}\ref*{#1}}}}
\newcommand{\plaineqref}[1]{{{\color{blue}(\ref*{#1})}}}

\begin{document}

\begin{frontmatter}
\title{A Unified Analysis of Likelihood-based Estimators in the Plackett--Luce Model}
\runtitle{Analysis of Likelihood-based Estimators in the PL model}

\begin{aug}

\author[A]{\fnms{Ruijian}~\snm{Han}\ead[label=e1]{ruijian.han@polyu.edu.hk}}\footnote{Authorships are ordered alphabetically\label{myfootnote}.},
\and
\author[B]{\fnms{Yiming}~\snm{Xu}\ead[label=e2]{yiming.xu@uky.edu}}\hyperref[myfootnote]{\footnotemark[\value{footnote}]}

\address[A]{Department of Data Science and Artificial Intelligence, 
	The Hong Kong Polytechnic University\printead[presep={,\ }]{e1}}

\address[B]{Department of Mathematics, University of Kentucky\printead[presep={,\ }]{e2}}
\end{aug}

\begin{abstract}
	The Plackett--Luce model has been extensively used for rank aggregation in social choice theory. A central statistical question in this model concerns estimating the utility vector that governs the model's likelihood. In this paper, we investigate the asymptotic theory of utility vector estimation by maximizing different types of likelihood, such as full, marginal, and quasi-likelihood. Starting from interpreting the estimating equations of these estimators to gain some initial insights, we analyze their asymptotic behavior as the number of compared objects increases. In particular, we establish both uniform consistency and asymptotic normality of these estimators and discuss the trade-off between statistical efficiency and computational complexity. For generality, our results are proven for deterministic graph sequences under appropriate graph topology conditions. These conditions are shown to be informative when applied to common sampling scenarios, such as nonuniform random hypergraph models and hypergraph stochastic block models. Numerical results are provided to support our findings.
\end{abstract}

\begin{keyword}[class=MSC]
	\kwd[Primary]{ 62F07}
		\kwd[; secondary]{ 62F12}
\end{keyword}	

\begin{keyword}	
	\kwd{Asymptotic normality}
	\kwd{Entrywise error}
	\kwd{Hypergraph}
	\kwd{Likelihood-based estimation}
	\kwd{Plackett--Luce model}
\end{keyword}

\end{frontmatter}

\tableofcontents


\section{Introduction}
Combining scores of multiple agents to produce a collective ranking is a frequent task across various disciplines in social science, including econometrics \citep{mcfadden1973conditional}, game analytics \citep{massey1997statistical, glickman1999rating}, psychometrics \citep{thurstone1927method}, and others \citep{hastie1997classification, baltrunas2010group, caron2014bayesian}. One common approach to modeling such situations involves incorporating latent parameters into the objects under comparison. Assuming pairwise comparison outcomes are unaffected by irrelevant alternatives, a unique parametrization of the model can be obtained \citep{MR0108411}. This leads to the so-called Plackett--Luce (PL) model \citep{plackett1975analysis}, which is the focus of this paper. 

As a special instance of random utility models, the PL model is popular as it strikes a balance between model complexity and computational tractability \citep{zhao2018composite}; it has recently been used in modern machine learning tasks such as fine-tuning in large language models \citep{ouyang2022training, rafailov2024direct}. When all comparison data are pairwise, the PL model reduces to the Bradley--Terry (BT) model \citep{MR0070925}, which has been actively investigated over the last 30 years. For a selection of results most relevant to this paper, see \cite{MR1724040, MR2051012, MR2987494, MR3504618, MR3613103, MR3953449, han2020asymptotic, chen2022optimal, chen2022partial, MR4560195, gao2023uncertainty, JMLR:v24:22-0214}. 

In contrast, the PL model with multiple comparisons was less well-understood because of the more involved data structure. Recently, a few studies emerged on the asymptotic analysis in the PL model with equal-sized comparison data \citep{jang2018top, fan2022ranking}. These results, however, do not apply to the PL model with a general comparison graph structure. To fill this gap, this paper provides a unified asymptotic theory for likelihood-based estimators in the PL model, relaxing the assumptions on graph homogeneity and equal comparison size. Before delving into the technical details, we provide a brief overview of related work, with an emphasis on studies directly related to the PL model.

\subsection{Related Work}

A large body of work in the PL model focuses on the computational aspects of parameter estimation. \cite{MR2051012} used the Minorization-Maximization (MM) algorithm to find the maximum likelihood estimation (MLE), and alternative Bayesian approaches have been proposed in \cite{guiver2009bayesian, caron2012efficient}. In \cite{maystre2015fast}, the authors provided a Markov chain interpretation of the MLE in the choice-one setting, leading to a spectral algorithm later improved in \cite{agarwal2018accelerated}. Additionally, \cite{azari2013generalized} introduced a breaking technique in the full-ranking setting by first breaking multiple comparisons into pairwise data and then applying the generalized method of moments. This data-breaking approach is closely related to the quasi-maximum likelihood estimation (QMLE) in Section \ref{3.2}. 

On the theoretical side, most existing studies assume the comparison data in the PL model to be equal-sized. In \cite{jang2018top}, the conditions for achieving exact recovery of top-$K$ ranking were discovered albeit being suboptimal. When only the choice-one comparison data are considered, the PL model falls into a general framework in \cite{MR3504618}, and a weighted $\ell_2$ convergence rate is established for the MLE. More recently, \cite{fan2022ranking} conducted a comprehensive analysis of the choice-one MLE, obtaining both uniform consistency and asymptotic normality. Nevertheless, none of these works considered the PL model in a context where the comparison graph is heterogeneous and the comparison data have varying sizes. This setting is often closer to practical scenarios but presents considerable challenges in building a theoretical statistical foundation. To our knowledge, the only study parallel to ours that accommodates more realistic comparison graph structures is \cite{fan2023spectral}, which focuses on the spectral estimator rather than the MLE. Generally speaking, analyzing network models with nonuniform hypergraph structure and imbalanced edge distribution is a challenging task \citep{zhen2022community}.

\subsection{Contributions}

This paper advances the asymptotic theory for both utility vector estimation and inference in the PL model using likelihood-based methods. The comparison graph considered here may involve edges of varying sizes and a heterogeneous degree distribution. At a high level, our contributions can be described as follows.

\begin{itemize}
\item We formulate various likelihood-based estimators in the PL model, including the full-likelihood, the marginal-likelihood, and the quasi-likelihood. Their estimating equations are interpreted through the lens of rank matching, which provides useful heuristics for why these estimators are expected to work—particularly in the quasi-likelihood setting. The details of the related results are given in Propositions~\ref{lm:1}-\ref{lm:3}. 

\item We develop a unified uniform consistency result for likelihood-based estimators under suitable graph topology conditions in deterministic designs (Theorem~\ref{main:general}). Under additional balancing assumptions, we further establish the asymptotic normality of these estimators (Theorem~\ref{thm:and}). We also discuss the trade-off between statistical efficiency and computational complexity of these estimators, offering practical guidance on their use in uncertainty quantification.

\item To demonstrate the proposed conditions, we show that for several common random sampling scenarios (including the nonuniform random hypergraph model and the hypergraph stochastic block model), the proposed conditions are optimal in terms of the leading-order model sparsity (Lemmas~\ref{main:REs} and \ref{newtech}). Moreover, these conditions are flexible enough to accommodate models with different edge sizes and heterogeneous edge densities. See Theorems~\ref{main_1}-\ref{main_2} for the results on uniform consistency and Theorems~\ref{normal_MLE}-\ref{normal_QMLE} on asymptotic normality in the random design setting.
\end{itemize}

Both the uniform consistency and asymptotic normality results are proven for deterministic graph sequences. For uniform consistency, our analysis extends the chaining technique in \cite{han2022general} to hypergraphs. This method links estimation errors to the asymptotic connectivity of the comparison graph sequence, allowing the former to be bounded using graph-theoretic attributes. Despite its conceptual simplicity, applying this method to hypergraphs is nontrivial. In particular, we need to verify an admissibility property of the chaining sets, which is not straightforward in the PL model due to the interplay between comparison outcomes and the utility vector. Moreover, to instantiate the deterministic results to concrete sampling scenarios, we show that several random graph models of interest can produce hypergraph sequences that satisfy the desired conditions (Lemma~\ref{main:REs}). This step is crucial in illustrating that the proposed conditions for deterministic graph sequences are not merely hypothetical. The proof of this step relies on intricate counting of subgraph structures in a hypergraph.

Our asymptotic normality result follows the standard route of Taylor expansion to compute the asymptotic variance of estimators. To analyze the remainder terms in the expansion, we adopt a truncated error analysis. Since we consider a class of likelihood-based estimators rather than a single one, it is helpful to discuss each estimator separately. For the choice-one MLE and QMLE, which are the most commonly studied in the literature \citep{chen2022optimal, fan2023spectral}, our method exclusively relies on the truncated error analysis. This approach is based on the Neumann series expansion of the normalized Hessians of appropriate log-likelihood functions and differs from the state-of-the-art leave-one-out analysis. Notably, the proposed conditions using truncated error analysis hold in several random sampling scenarios with heterogeneous structures (Lemma~\ref{newtech}), which may not be achievable through leave-one-out analysis alone.

For the marginal MLE, additional complexity arises because the Hessian of the marginal log-likelihood is random. Under such circumstances, we combine the truncated error analysis with a leave-one-out perturbation argument to obtain the desired results. While the latter is similar to \cite{chen2022optimal, fan2023spectral} in spirit, there are notable differences. For instance, our perturbation argument is applied when the hypergraph sequence is deterministic. To formulate the appropriate conditions, we need to bound the spectral gaps of the leave-one-out subgraphs using the spectral gap of the original graph, which requires an additional layer of perturbation analysis on graphs (Lemma~\ref{lastone}). This step involves extra technical machinery and may be of independent interest.

Compared to the analysis in pairwise graphs, our study reveals an interesting edge-sharing phenomenon unique to hypergraphs. For two distinct objects, the edge-sharing ratio is defined as the number of edges shared between them divided by the minimum of their respective degrees. This ratio captures the correlation between the utility estimates of different objects and plays a crucial role in identifying the explicit form of the asymptotic variances of the estimators. While this ratio is asymptotically vanishing in pairwise comparison graphs, it can be of constant order when multiple comparisons are present, introducing additional obstacles in the analysis. 

\subsection{Organization}
The rest of the paper is organized as follows. Section \ref{sec:2} introduces the PL model and the multiple comparison graphs studied throughout the paper. Section \ref{sec:3} reviews several likelihood-based estimators in the PL model, including the full MLE, marginal MLE, and QMLE. Interpretation of their estimating equations is provided from the viewpoint of matching certain rank-related equations. Section \ref{sec:44} contains the main asymptotic results, including both uniform consistency and asymptotic normality for the likelihood-based estimators under random designs, which is the common setting in the literature and thus stated first. These results are proven in full generality under deterministic designs in Section \ref{sec:55}. Section \ref{sec:7} presents numerical simulations to support our theoretical findings. Proofs and other technical details are deferred to the appendices.

\subsection{Notation}

For $n\in\mathbb N$, let $[n] = \{1, \ldots, n\}$. For $T\subseteq [n]$, the power set of $T$ is denoted by $\mathscr P(T)$. A permutation $\pi$ on $T$ is a bijection between $T$ and itself, and the set of permutations on $T$ is denoted by $\mathcal S(T)$. If $|T| = m$, we write $\pi$ as $\pi(1), \ldots, \pi(m)$, where $\pi(i)$ is the element in $T$ that is mapped to the $i$th order under $\pi$. For a multiple comparison model with $n$ objects, we reserve the notation $\bm u^* = (u_1^*, \ldots, u_n^*)^\top$ for the ground truth utility vector, and $\buu$ and $\widetilde{\bu}$ the corresponding (marginal) MLE and QMLE, respectively. 

A comparison graph is represented as a hypergraph $\HH(V, E)$, where $V$ is the vertex set and $E\subseteq \mathscr P(V)$ is the edge set. For $\ms\in [n]$, let $\Omega^{(\ms)} = {V\choose \ms} = \{e\in \mathscr P(V): |e| = \ms\}$ denote all possible edges of size $\ms$. For $U_1, U_2\subseteq V$ with $U_1\cap U_2=\emptyset$, we define the edges between $U_1$ and $U_2$ as $\mathcal E(U_1, U_2)=\{e\in E: e\cap U_i\neq\emptyset, i = 1, 2\}$. The set of boundary edges of $U$ is defined as $\partial U = \mathcal E(U, U^\complement)$. 

For $\bx\in\mathbb R^n$, $B\in\mathbb R^{n\times n}$ and $1\leq p\leq\infty$, we denote the $\ell_p$-norm of $\bx$ as $\|\bx\|_p$, the operator norm of $B$ induced by the $\ell_{p}$-norm as $\|B\|_{p\to p} = \max_{\bx\in\mathbb R^n: \|\bx\|_p = 1}\|B\bx\|_p$. When $p=2$, this is the spectral norm and we often write $\|B\|_{2} = \|B\|_{2\to 2}$. The $\ell_\infty$-norm of vectorized $B$ as $\|B\|_{\max}:= \max_{i,j\in [n]}|B_{ij}|$. When all eigenvalues of $B$ are real, we arrange them in the nondecreasing order as $\lambda_1(B)\leq\cdots\leq \lambda_n(B)$. For a smooth function $f(\bx): \mathbb{R}^n \to \mathbb{R}$ and $k_1, \ldots k_s \in [n]$, we let $\partial_{k_1\ldots k_s} f$ denote the $s$-order partial derivative of $f$ with respect to the $k_1, \ldots k_s$ components. We use $\vee$ and $\wedge$ to denote the \texttt{max} and \texttt{min} operators, respectively.  We use $O(\cdot )/O_p(\cdot )$ and $o(\cdot )/o_p(\cdot )$ to represent the Bachmann--Landau asymptotic notation. We use $\lesssim$ and $\gtrsim$ to denote the asymptotic inequality relations, and $\asymp$ if both $\lesssim$ and $\gtrsim$ hold. We use $\mathbf 1_{A}(x)$ to denote the indicator function on a set $A$, that is, $\mathbf 1_{A}(x)=1$ if $x\in A$ and $\mathbf 1_{A}(x)=0$ otherwise. We let $\langle 1\rangle=(1, \ldots,  1)^\top$ denote the all-ones vector with a compatible dimension that is often clear from the context.

\section{Problem Set-up}\label{sec:2}

We first briefly introduce the PL model and explain the data structure involved. Then we discuss two random comparison hypergraph models that generalize the frequently used random graph models in the pairwise setting.  

\subsection{The PL Model}
Consider a comparison network of $n$ objects. In the PL model, each object $k\in [n]$ is assumed with a latent score $\e{u_k^*}$, where $\bu^* = (u_1^*, \ldots, u_n^*)^\top\in\mathbb R^n$ is known as the utility vector. Given $T\subseteq [n]$ with $|T|=m$, a full observation on $T$ is a totally ordered sequence $\pi(1)\succ \cdots \succ\pi(m)$, where $\pi\in\mathcal S(T)$, and $\pi(j)$ denotes the element with rank $j$. Sometimes it is more convenient to work with the rank of an element, for which we introduce the notation $r(k):= \pi^{-1}(k) = \{j: \pi(j) = k\}$. Identifying a full observation as $\pi$, the PL model assumes that the probability of observing $\pi$ on $T$ is
\begin{align}
\PP_{\bu^*}(\pi \mid T) := \prod_{j=1}^m\frac{\e{u_{\pi(j)}^*}}{\sum_{t = j}^m \e{u_{\pi(t)}^*}}.\label{PL-mass}
\end{align}
The subscript of $\mathbb P$ denotes the dependence on the utility vector and is dropped when clear from the context. 
Note that \eqref{PL-mass} is invariant under the common shift of $\bm u^*$. To ensure the model is identifiable, we impose one commonly used constraint $\langle 1\rangle^\top \bu^* = \sum_{j\in [n]}u_j^* = 0$, where $\langle 1\rangle$ is the all-ones vector with compatible size.

As its name suggests, the PL model is related to Luce's choice axiom \citep{MR0108411}, which states that selecting one object over another is not affected by the presence or absence of other objects. In fact, the PL model can be characterized using the following conditions:
\begin{condition}\label{ass:1}
There exist a family of \emph{choice-one} selection probabilities $\{\mathbb P^T\}_{T\subseteq [n]}$ satisfying Luce's choice axiom. 
\end{condition}
\begin{condition}\label{ass:2}
For $T\subseteq [n]$ with $|T|=m$, $\pi\in\mathcal S(T)$ is a random permutation obtained by sequentially sampling $\pi(j)$ from $\mathbb P^{T_{-j+1}}$, where $T_{-0} = T$ and $T_{-j} = T\setminus\{\pi(1), \ldots, \pi(j)\}$. 
\end{condition}
Under Luce's choice axiom, Condition~\ref{ass:1} is equivalent to 
\begin{align*}
\mathbb P^T(\text{$k$ is selected}) = \frac{ \e{u_k^*} }{\sum_{j\in T}\e{u_j^*}}
\end{align*}
for some $\bu^* \in\mathbb R^n$. Condition~\ref{ass:2} says that $\pi$ is obtained by independently selecting the top element from the remaining comparison set: $\mathbb P(\pi \mid T) = \prod_{j=1}^{m}\mathbb P^{T_{-j+1}}(\text{$\pi(j)$ is selected}).$
With such interpretation, the equivalence between \eqref{PL-mass} and Conditions \ref{ass:1}-\ref{ass:2} is obvious.

An important property of the probability measures defined in \eqref{PL-mass} is that they are \emph{internally consistent} \citep{MR2051012}. 
In particular, for $T' \subseteq T$,
\begin{align}
&\PP(\pi' \mid T') = \sum_{\pi\in \mathcal S(T), \pi^{-1}|_{T'}\sim\pi'^{-1}}\PP(\pi \mid T)&\pi'\in \mathcal S(T'),\label{conS}
\end{align}
where $\pi^{-1}|_{T'}\sim\pi'^{-1}$ means that the relative ranks of $T'$ under $\pi$ and $\pi'$ are the same. 
This implies that the probability of a multiple comparison observation in the PL model, whether treated as a full or partial observation (on a larger comparison set), is well-defined.

\subsection{Comparison Data}\label{sec:2.2}

We now describe the data.  
Let $N$ denote the number of \emph{independent} comparisons, and $N_k := |\{i\in [N]: k\in T_i\}|$ denote the number of those containing the object $k$. For $i\in [N]$, the $i$th data can be represented as a pair $(T_i, \pi_i)$, where $T_i$ (with $|T_i| = m_i$) is where multiple comparisons take place and $\pi_i$ encodes the observed ranks of objects in $T_i$. 
Here we do \emph{not} require $m_i=m_j$ for $i, j\in [N]$, namely, the sizes of the comparison sets are allowed to be different.
We assume that $m_i$'s are independent of the utility vector $\bm u^*$.  
The corresponding full observation $\mathcal{O}^{\full}_i$ is given by
\begin{align*}
\mathcal{O}^{\full}_i  = \pi_i(1)\succ\cdots \succ \pi_i(m_i).
 \end{align*}
The rank of $k$ in $\pi_i$ (assuming $k\in T_i$) is denoted by $r_i(k)$.  

More generally, one may consider the partial observation consisting of the choice-$y_i$ ($y_i\leq m_i$) objects:
\begin{align*}
\mathcal{O}^{\ptl}_i  = \pi_i(1)\succ\cdots \succ \pi_i(y_i) \succ \text{others}.
\end{align*}
When $y_i=1$ for all $i\in [N]$, the partial observations are called the \textit{choice-one} observations. 
The partial observations become full observations if $y_i = m_i$ for all $i\in [N]$. 

It is helpful to recall the sequential sampling view in Condition \ref{ass:2}. 
For every $\pi_i\sim\mathbb P(\cdot \mid T_i)$, it can be viewed as a random process adapted to the natural filtration 
\begin{align}
\mathscr F_{i, j} = \sigma(\pi_i(t), t\leq j).\label{filt}
\end{align} 
This perspective allows for convenient interpretation of certain quantities in Section \ref{sec:3}.

\subsection{Comparison Graphs}\label{sec:cg}

The $n$ objects under comparison together with the comparison data can be represented as a hypergraph $\HH(V, E)$, where $V = [n]$ and $E = \{T_i: i\in [N]\}$. In the following, we consider two types of random hypergraph models whose parameters are independent of the utility vector $\bm u^*$.

\subsubsection{Nonuniform Random Hypergraph Models}\label{am}
We consider the following nonuniform random hypergraph model (NURHM).
Fixing an absolute integer $M\leq n$, we assume that 
\begin{align}\label{NURHM}
&E = \bigcup_{\ms=2}^{M}E^{(\ms)}& E^{(\ms)}\subseteq \Omega^{(m)}={[n]\choose \ms},
\end{align}
where $E^{(\ms)}$ are independent $\ms$-uniform random hypergraphs generated as follows.
Fixing $\ms$, $\mathbf 1_{\{e\in E^{(\ms)}\}}$ are independent Bernoulli random variables with parameter $p_{e, n}^{(\ms)}$ for $e\in\Omega^{(m)}$. 
Denote by 
\begin{align}
&p_n^{(\ms)} = \min_{e\in\Omega^{(m)}}p_{e, n}^{(\ms)}& q_n^{(\ms)} = \max_{e\in\Omega^{(m)}}p_{e, n}^{(\ms)}.\label{pnqn}
\end{align}
When $p_n^{(\ms)} = q_n^{(\ms)} = p_\mm\mathbf 1_{\{\ms=\mm\}}$ for some $p_\mm\in [0, 1]$, $\HH(V, E)$ reduces to the $\mm$-way uniform hypergraph model that generalizes the Erd\H{o}s--R\'enyi model \citep{MR0125031}. 
Nevertheless, NURHM is strictly more general as it can produce varying-size comparison data with different probabilities.  

\subsubsection{Hypergraph Stochastic Block Models}\label{hsbm}

The hypergraph stochastic block model (HSBM) is a special instance of NURHM that generalizes the stochastic block model \citep{Holland_1983} to the hypergraph setting. It was first studied in \cite{ghoshdastidar2014consistency} and has since been analyzed in the context of community detection problems \citep{florescu2016spectral, pal2021community}. Analogous to the stochastic block model \citep{abbe2017community}, an $\mm$-uniform HSBM model with $K<\infty$ clusters partitions $V$ into $K$ subsets, $V = \sqcup_{i\in [K]} V_i$, where edges within and across partitioned sets have different probabilities of occurrence: 
\begin{align}\label{sbm}
&\mathbb P(e\in E, |e| = \mm) = \begin{cases}
\omega_{n, 1}\mathbf 1_{\{e\in\Omega^{(M)}\} }& e\subseteq V_1\\
\vdots\\
\omega_{n, K}\mathbf 1_{\{e\in\Omega^{(M)}\}}& e\subseteq V_K\\
\omega_{n, 0}\mathbf 1_{\{e\in\Omega^{(M)}\}}& \text{otherwise}
\end{cases}.
\end{align}
In \eqref{sbm}, the probabilities of edges across different partitions are assumed to be the same, although they can be further refined depending on which $V_i$ intersects the edge as well as the intersected size. Moreover, the probabilities within the same partition can be made different but of the same order asymptotically. Such generalizations do not change our theoretical results but increase notational complexity. For simplicity, we adopt the setup in \eqref{sbm} when referring to an HSBM. 

An HSBM can produce heterogeneous and clusterable configurations if $\max_{i}\omega_{n, i}/\min_{i}\omega_{n, i}$ diverges as $n\to\infty$. In the literature on classification, data associated with heterogeneous graphs is typically called imbalanced data.

\section{Utility Estimation}\label{sec:3}

Conditioning on the edge set $E = \{T_i: i\in [N]\}$, we consider several commonly used likelihood-based approaches to estimating $\bm u^*$ based on observed data. We provide a unified standpoint to interpret their estimating equations.

\subsection{The Likelihood Approach}\label{3.1}
The log-likelihood function can be explicitly written down in a PL model. Given data $\{(T_i, \pi_i)\}_{i\in [N]}$, the marginal log-likelihood function based on the choice-$y_i$ objects in each $\pi_i$ ($y_i\leq |T_i|$ is given and independent of $\bm u^*$) is 
\begin{align}
l_1(\bm u) = \sum_{i\in [N]}\sum_{j\in [y_i]}\left[u_{\pi_i(j)} - \log\left(\sum_{t=j}^{m_i}\e{u_{\pi_i(t)}}\right)\right].\label{mmle}
\end{align}
Let $\mathbb E_{\bm u}$ denote the expectation operator in a PL model with parameter $\bm u$. 
For convenience, we omit the subscript when $\bm u$ is the ground truth, that is, $\bm u = \bm u^*$.  
The following proposition shows that the (marginal) MLE $\buu$ satisfies a system of estimating equations.
\begin{Proposition}\label{lm:1}
The MLE $\buu$ obtained from maximizing \eqref{mmle} subject to $\langle 1\rangle^\top\bu = 0$, if exists, satisfies
\begin{align*}
&\frac{1}{N_k}\sum_{i: k\in T_i}S_i( {\buu}; k) = 0& k\in [n],
\end{align*} 
where 
\begin{align}\label{estimating_quantity}
S_i(\bu; k) = \Big(\sum_{j\in [m_i]}\mathbb E_{\bu}[\mathbf 1_{\{\text{$r_i(k) = j$ and $r_i(k)\leq y_i$}\}}\mid \mathscr F_{i, j-1}]\Big) - \mathbf 1_{\{r_i(k)\leq y_i\}}
\end{align}
and $\mathscr F_{i, j-1}$ are defined in \eqref{filt}.
Meanwhile, the true parameter $\bu^*$ satisfies $\mathbb E[S_i(\bm u^*; k)] = 0$ for every $i\in [N]$.
\end{Proposition}

The proof follows by checking the first-order optimality condition of the (marginal) MLE and is given in the appendices. 
To better understand the estimating equations, we consider two special cases of $\{y_i\}_{i\in [N]}$ introduced in the previous section.

\subsubsection{Full Observations}

In this case, $y_i = m_i$ for all $i\in [N]$.
The quantity \eqref{estimating_quantity} in the estimating equation for each $i\in [N]$ and $k\in [n]$ reduces to
\begin{align}
\Big(\sum_{j\in [m_i]}\mathbb E_{\bm u}[\mathbf 1_{\text{\{$r_i(k) = j$\}}}\mid \mathscr F_{i, j-1}] \Big)- 1.\label{full1}
\end{align}
 {Note that $\mathbb E_{\bm u}[\mathbf 1_{\text{\{$r_i(k) = j$\}}}\mid \mathscr F_{i, j-1}]$ estimates the probability of the rank of $k$ being $j$ given the information revealed before $\pi_i(j)$. Thus, the first term in \eqref{full1} is an aggregation of probability estimates for the potential ranks of $k$, whose expectation is $1$ under the ground truth model $\bm u^*$ by the tower property. 

\subsubsection{Choice-one Observations}

In this case, $y_i=1$ for all $i\in [N]$, which coincides with the choice-one comparison model in \cite{fan2022ranking}. 
This model is often called Luce's choice model. 
Similar to the full-observation case, the quantity \eqref{estimating_quantity} in the estimating equation for each $i\in [N]$ and $k\in [n]$ reduces to
\begin{align}
\mathbb E_{\bm u}[\mathbf 1_{\text{\{$r_i(k)=1$\}}}]- \mathbf 1_{\text{\{$r_i(k)=1$\}}},\label{choice1}
\end{align}
where $\mathbb E_{\bm u}[\mathbf 1_{\text{\{$r_i(k)=1$\}}}]$ denotes the probability of $k$ being the first element in $\pi_i$ without observing any information.

\begin{remark}\label{kind:remark}
The conditional expectation $\mathbb E_{\bu}[ \cdot \mid \mathscr F_{i, j-1}]$ in \eqref{estimating_quantity} is random and depends on the ranking outcome when $|T_i|>2$. This differs from the BT model ($|T_i|=2$), whose score function can be expressed by two parts: one that only depends on $\bu$ and the other that only depends on the ranking outcome. Such a difference complicates the asymptotic analysis of the marginal MLE in later sections. 
\end{remark}

\subsection{The Quasi-likelihood Approach}\label{3.2}

An alternative approach is to break the observations into pairwise comparisons. According to Condition \ref{ass:2}, multiple comparison data arising from a PL model implies a number of pairwise comparison outcomes. For example, for $T\subseteq [n]$ with $|T| = m$ and $\pi\in\mathcal S(T)$, 
\begin{align}
\{\pi(1)\succ\cdots\succ\pi(m)\} \Longrightarrow \bigcap_{1\leq j<t\leq m}\{\pi(j)\succ\pi(t)\}.\label{BT}
\end{align}
The right-hand side of \eqref{BT} is called the full breaking of $\pi$. Other types of breaking, such as the top/bottom breaking \citep{azari2013generalized}, can be considered similarly. For simplicity, we focus on the full-breaking case. 

The pairwise comparisons in a full breaking are partial observations from different comparison sets. Thanks to the internal consistency of the PL model \eqref{conS}, the pairwise events on the right-hand side of \eqref{BT} could be treated as if they were obtained from the BT model, albeit with a loss of dependence among those originating from the same breaking. As such, we can write down the log-likelihood function in this misspecified model as 
\begin{align}
l_2(\bm u) = \sum_{i\in [N]}\sum_{1\leq j<t \leq m_i}\left[u_{\pi_i(j)}-\log (\e{u_{\pi_i(j)}}+ \e{u_{\pi_i(t)}})\right],\label{pmle}
\end{align}
and seek the corresponding MLE $\widetilde{\bm u}$, which is a QMLE approach per se. 
The classical theory of the QMLE in a fixed dimension suggests that the QMLE will converge to an element in the quasi-likelihood parameter space that is closest to the true likelihood in the Kullback--Leibler (KL) divergence under suitable conditions. If the parameters of the convergent element coincide with the parameters of the true likelihood, then the QMLE is consistent. The next result shows that for every fixed $n$, the closest element to the true PL likelihood in the misspecified model \eqref{pmle} under KL divergence has the same parameters as the true PL likelihood (in the population level). Although this result provides evidence that the QMLE is a promising estimator in practice, it is different from the uniform consistency result in Section \ref{sec:4} where both $n$ and $N$ diverge.

For ease of illustration, we assume the comparison graph $\HH(V, E)$ is fixed. The likelihoods of observing $\bm \pi:=\{\pi_T\}_{T\in E}$, where $\pi_T\in\mathcal S(T)$, under the PL and the BT models (the model for the quasi-likelihood after breaking), are as follows. For the likelihood of the PL model, we have
\begin{align*}
&\text{ $f(\bm\pi; \bm u) = \prod_{T\in E}f_T(\pi_T; \bm u)\quad f_T(\pi_T; \bm u):= \prod_{j\in [|T|]}\frac{\e{u_{\pi_T(j)}}}{\sum_{t\geq j}\e{u_{\pi_T(t)}}}$}
\end{align*} 
For the likelihood of the BT model, we have
\begin{align*}
&\text{ $g(\bm\pi; \bm u) = \prod_{T\in E}g_T(\pi_T; \bm u)\quad g_T(\pi_T; \bm u):=\prod_{1\leq j<t\leq |T|}\frac{\e{u_{\pi_T(j)}}}{\e{u_{\pi_T(j)}}+\e{u_{\pi_T(t)}}}$}.
\end{align*} 
\begin{Proposition}\label{lm:2}
If $\HH(V, E)$ is connected, then for every $\bm u\in\mathbb R^n$ {with $\langle1 \rangle ^\top\bm u = 0$}, 
\begin{align*}
\bu = \argmin_{\bm v: \langle1 \rangle ^\top\bm v = 0}\mathbb E_{\bm\pi\sim f}[\mathsf{KL}(f(\bm\pi; \bm u)\| g(\bm\pi; \bm v))],
\end{align*}
where $\mathsf{KL}(\cdot \| \cdot)$ is the KL-divergence. 
\end{Proposition}

Moreover, an almost identical computation as the proof reproduces an analogous result as Proposition \ref{lm:1}.

\begin{Proposition}\label{lm:3}
The QMLE $\widetilde{\bu}$ obtained from maximizing \eqref{pmle} subject to $\langle 1\rangle^\top\bu = 0$, if exists, satisfies
\begin{align*}
&\frac{1}{N_k}\sum_{i: k\in T_i}\left(r_i(k) - \mathbb E_{\widetilde{\bm u}}[r_i(k)]\right) = 0\label{!11}& k\in [n].
\end{align*} 
Meanwhile, for every $i\in [N]$, the true parameter $\bu^*$ satisfies $\mathbb E[r_i(k)-\mathbb E[r_i(k)]] = 0$.
\end{Proposition} 

Proposition \ref{lm:3} suggests that even though the QMLE does not account for the dependence of outcomes within a multiple edge, it still functions effectively as a moment estimator of $\bu^*$. As a result, the QMLE is expected to exhibit similar desired asymptotic properties as the marginal MLE. We will further explore this in the following sections.

\section{Asymptotic Results for Random Comparison Graphs}\label{sec:44}

We first present the asymptotic results for the likelihood-based estimators when the underlying comparison graph sequence is random, that is, drawn from NURHM or HSBM. Since random designs are a common setting in the related literature, our results offer a convenient comparison with other relevant studies. The results in this section can be seen as specific instances of the more general results for deterministic designs discussed in Section~\ref{sec:55}.

\subsection{Uniform Consistency}\label{sec:4}

For uniform consistency, we require the following assumption on the comparison model.

\begin{Assumption}\label{ap:2}
The utility vector $\bm u^*$ is uniformly bounded, that is, there exists a universal constant  $C_1 \geq 1$ such that
 $\|\bm u^*\|_\infty \leq \log (C_1)$. 
\end{Assumption}

Moreover, we make additional assumptions on the comparison graph, one of which depends on whether the graph is sampled from a NURHM or an HSBM.

\begin{Assumption}\label{ap:1}
The maximum size of comparison edges is asymptotically bounded, that is, $M:=\sup_{n\in\mathbb N}\max_{i\in [N]}|T_i|<\infty$.  
\end{Assumption}
\begin{Assumption}\label{ap:4}
For a NURHM defined in Section \ref{am}, we assume
\begin{equation}
\lim_{n\to\infty}\frac{\xi^2_{n,+}(\log n)^3}{\xi^3_{n,-}} = 0,\label{2022}
\end{equation}
where 
\begin{align}
&\xi_{n,-}:=\sum_{\ms=2}^M n^{\ms-1}p_n^{(\ms)} & \xi_{n,+}:=\sum_{\ms=2}^M n^{\ms-1}q_n^{(\ms)},\label{myxis}
\end{align}
and $p_n^{(\ms)}$, $q_n^{(\ms)}$ are defined in \eqref{pnqn}. 
\end{Assumption}

\begin{Assumption}\label{ap:3}
For an HSBM in Section \ref{hsbm}, we assume
\begin{align*}
\lim_{n\to\infty}\frac{(\log n)^3}{\zeta_{n,-}} = 0,
\end{align*}
where 
\begin{align}
\zeta_{n,-}:=n^{\mm-1}\min_{0\leq i \leq K}\omega_{n, i},\label{myzeta}
\end{align}
and $\omega_{n,i}$ are defined in \eqref{sbm}. 
\end{Assumption}
Assumption \ref{ap:1} assumes that the maximum comparison edge size is uniformly bounded. Assumptions \ref{ap:4}-\ref{ap:3} are concerned with the minimal connectivity of the underlying graph, with the former imposing additional balancing conditions on the degree of heterogeneity of the edge probabilities. For homogeneous $\mm$-uniform hypergraph models, Assumption \ref{ap:4} holds if the comparison rate $p_\mm \gtrsim (\log n)^{3+\kappa}/{{n-1}\choose {\mm-1}}$ for any $\kappa>0$, a near-minimal sparsity condition that matches the result in \cite{fan2022ranking}.

\begin{thm}[Uniform consistency under NURHM]\label{main_1}
	
		Suppose the comparison graph sequence is sampled from a NURHM in Section \ref{am}.
		Under Assumptions \ref{ap:2}-\ref{ap:4}, a.s., {for all sufficiently large $n$}, both the marginal MLE $\buu$ and the QMLE $\widetilde{\bu}$ in Section \ref{sec:3} uniquely exist and are uniformly consistent, that is, 
		\begin{align}
			&\|\bm w -\bu^*\|_\infty\lesssim \sqrt{\frac{\xi^2_{n,+}(\log n)^3}{\xi^3_{n,-}}} \to 0& n\to\infty, \label{k91}
		\end{align} 
		where $\bm w = \buu$ or $\widetilde{\bu}$ and $\xi_{n, \pm}$ are the same as \eqref{myxis}. 

\end{thm}

\begin{thm}[Uniform consistency under HSBM]\label{main_2}
	Suppose the comparison graph sequence is sampled from an HSBM in Section \ref{hsbm}. 
	Under Assumptions \ref{ap:2}, \ref{ap:1}, and \ref{ap:3}, a.s., {for all sufficiently large $n$}, 
	both the marginal MLE $\buu$ and the QMLE $\widetilde{\bu}$ in Section \ref{sec:3} uniquely exist and are uniformly consistent, that is, 
\begin{align}
	&\|\bm w -\bu^*\|_\infty\lesssim\sqrt{\frac{(\log n)^3}{\zeta_{n,-}}}\to 0& n\to\infty,
\end{align}
where $\bm w = \buu$ or $\widetilde{\bu}$ and $\zeta_{n,-}$ is the same as \eqref{myzeta}.
\end{thm}

The random hypergraph models in Theorems~\ref{main_1}-\ref{main_2} cover a range of comparison graph structures of interest.  For instance, the conditions in Theorem~\ref{main_1} admit hyperedges with varying sizes generated with balanced probabilities (allowing for a certain degree of heterogeneity). The balancing condition can be removed if community structure exists. Specifically, in the case of HSBM, according to Theorem~\ref{main_2}, requiring $\min_{0\leq i \leq K} n^{\mm-1}\omega_{n, i}\gtrsim (\log n)^{3+\kappa}$ alone is sufficient for uniform consistency. Under such circumstances, heterogeneity of typical configurations can be severe if $\max_{0\leq i \leq K} n^{\mm-1}\omega_{n, i}\gtrsim n$. Both Theorems~\ref{main_1}-\ref{main_2} generalize the graph conditions in \cite{jang2018top, fan2022ranking}.

\subsection{Asymptotic Normality}\label{sec:6}

To establish asymptotic normality, we need to strengthen the conditions in Assumptions \ref{ap:4}-\ref{ap:3}.

	\begin{Assumption}\label{ap:6}
	The comparison hypergraph is sampled from  a NURHM defined in Section \ref{am}, satisfying
		\begin{equation}
			\lim_{n\to\infty}\max\left\{\frac{\xi_{n, +}^{13}(\log n)^8}{\xi_{n,-}^{14}}, \frac{\xi_{n,+}^5(\log n)^2}{n\xi_{n,-}^5}\right\} = 0, \label{an_nurhm_condition}
		\end{equation}
		where 
		\begin{align*}
			&\xi_{n,-}=\sum_{\ms=2}^M n^{\ms-1}p_n^{(\ms)} & \xi_{n,+}=\sum_{\ms=2}^M n^{\ms-1}q_n^{(\ms)},
		\end{align*}
		and $p_n^{(\ms)}$, $q_n^{(\ms)}$ are defined in \eqref{pnqn}. 
	\end{Assumption}
	
	\begin{Assumption}\label{ap:5}
		The comparison hypergraph is sampled from an HSBM in Section \ref{hsbm}, satisfying
		\begin{align}
	\lim_{n\to\infty}\max\left\{\frac{\zeta_{n, +}^{11}(\log n)^8}{\zeta_{n,-}^{12}}, \frac{\zeta_{n,+}^5(\log n)^2}{n\zeta_{n,-}^5}\right\} =  0, \label{an_hsbm_condition}
		\end{align}
		where 
		\begin{align*}
			&\zeta_{n,-}=n^{\mm-1}\min_{0\leq i \leq K}\omega_{n, i} &\zeta_{n,+}:=n^{\mm-1}\max_{0\leq i \leq K}\omega_{n, i},
		\end{align*}
and $\omega_{n,i}$ are defined in \eqref{sbm}. 
\end{Assumption}

Compared to Assumption \ref{ap:3}, Assumption \ref{ap:5} further imposes an upper bound on $\zeta_{n,+}$ in HSBM, which is necessary for asymptotic normality. Indeed, while uniform consistency requires estimation errors of all parameters to converge to zero, asymptotic normality further demands that each estimated parameter converges at a rate compatible with its respective normalization. For instance, consider an object $k$ involved in $N_k$ comparisons. If asymptotic normality holds for $k$, then its estimated error will converge to zero at a rate of $O_p(1/\sqrt{N_k})$. However, when graph heterogeneity exists, the convergence at object $k$ is influenced by other objects with slower convergence rates. Such influence results in an additional bias term in the analysis. This is similar to the study of semi-parametric models where the nonparametric part slows down the convergence rate of the parametric part \cite[Assumption 5.1(ii)]{MR1303237}. Therefore, although \eqref{an_nurhm_condition}-\eqref{an_hsbm_condition} admit heterogeneous graph configurations, the level of heterogeneity must be controlled to a certain extent. This essentially differs from the uniform consistency result where heterogeneity can be severe; see Section~\ref{sec:add_sim} for the numerical evidence. 

To better describe the asymptotic normality results in detail, we need some additional notations. For the marginal MLE, given an edge $T\subseteq [n]$ with $|T| =m$ and $k\in T$, for $y = 1, \ldots,  m-1$, we define 
\begin{align*}
&\theta_{k, 1}(\bu^*; y, T) \\
&:= \sum_{S \subset {T}\setminus\{k\}: |S|=y-1} \mathbb{P}_{\bu^*}\left(\text{$r(j)<r(k)=y$ for $j\in S$}\mid T\right) (1- \mathbb{P}_{\bu^*}(r(k)=1\mid T\setminus S)),
\end{align*}
where $\mathbb P_{\bu^*}(\cdot \mid T)$ is the selection probability on edge $T$ with the parameter $\bu^*$ defined in \eqref{PL-mass} and $r(k)$ is the rank of $k$ on $T$. {For sake of completeness, the explicit form of $\theta_{k, 1}(\bu^*; y, T)$ is given by
\begin{eqnarray*}
	&&\theta_{k, 1}(\bu^*; y, T) \\
	&&= \sum_{i_{1}  \in {T}\setminus\{k\}}\cdots \sum_{\{i_{y-1}\} \in {T}\setminus\{k,i_1,\ldots,i_{y-2}\}}\Big\{\mathbb{P}_{\bu^*}\left(i_{1} \succ \ldots i_{y-1} \succ k \succ \text{others in }T\right) \\
	&&~~~~~~~~~~~~~~~~~~~~~~~~~~~~~~~~~~~~~~~~~~~~~~~~~~~~~~~~\times  \left(1- \mathbb{P}_{\bu^*}(k \succ \text{others in }T\setminus\{i_{1}, \ldots, i_{y-1}\} )\right)\Big\} \\
	&&= \sum_{i_{1}  \in {T}\setminus\{k\}}\cdots \sum_{\{i_{y-1}\} \in {T}\setminus\{k, i_1,\ldots,i_{y-2}\}}\Bigg\{ \bigg( \prod_{j=1}^{y-1} \frac{\e{u^*_{i_{j}}}}{\sum_{t \in T\setminus\{i_{1}, \ldots, i_{j}\}} \e{u^*_t}}\bigg) \\
	&&~~~~~~~~~~~~~~~~~~~~~~~~~~~~\times \frac{\e{u^*_{k}}}{\sum_{t \in T\setminus\{i_{1},i_{2}, \ldots, i_{y-1}\}} \e{u^*_t}}\times \bigg( 1- \frac{\e{u^*_{k}}}{\sum_{t \in T\setminus\{i_{1}, \ldots, i_{y-1}\}} \e{u^*_t}}\bigg)\Bigg\}.
\end{eqnarray*}}
The inverse asymptotic variance of the marginal MLE is characterized by
\begin{eqnarray}\label{rho_k1}
	\rho_{k,1}^2(\bu^*) := \sum_{i: k\in T_i} \sum_{y \in [y_i]} \theta_{k, 1}(\bu^*; y, T_i),
\end{eqnarray}
where $y_i$ refers to the choice-$y_i$ observations in the $i$th data as assumed in the marginal MLE estimation. 
The asymptotic normality of the marginal MLE is summarized as follows.

\begin{thm}[Asymptotic normality of marginal MLE]\label{normal_MLE}
Under Assumptions \ref{ap:2}, \ref{ap:1} and \ref{ap:6} or Assumptions \ref{ap:2}, \ref{ap:1} and \ref{ap:5}, for any fixed $k \in [n]$, the marginal MLE $\buu$ satisfies 
	\begin{align*}
				&{\rho_{k,1}(\bu^*)}(\uu_k - u_k^*) \to N(0,1)& n\to\infty. 
	\end{align*} 
\end{thm}

In Theorem \ref{normal_MLE}, $1/{\rho_{k,1}(\bu^*)}$ is the asymptotic standard deviation of $(\uu_k - u_k^*)$ for large $n$. Since each $\theta_{k,1}(\bu^*; y, T_i) >0$, \eqref{rho_k1} implies that the standard deviation of $(\uu_k - u_k^*)$ will decrease if $N_k$ increases, which is as anticipated.

On the other hand, \eqref{rho_k1} reveals a trade-off between statistical efficiency and computational complexity between different marginal MLEs. The larger $y_i$, the smaller the standard deviations, but with the increasing cost of computation. The computational cost for $\theta_{k,1}(\bu^*; y, T_i) $ is $|T_i|!/(|T_i|-y)!$. When $|T_i| = M$ and $y_i = y$ for all $i \in[N]$, the total computational cost for \eqref{rho_k1} is $N_k (M!/(M-y)!)$. Specifically, for the full observation, the complexity reaches $N_k(M!)$. Despite being statistically optimal, the asymptotic variance of the full MLE becomes computationally prohibitive even for moderate $M$. For instance, $M=14$ in the horse-racing data analysis in Section \ref{horse}. In this case, we don't use the full MLE due to its heavy computational cost.
\begin{remark}
Theorem~\ref{normal_MLE} extends the results in \cite{fan2022ranking} to the setting of both nonuniform and heterogeneous hypergraphs. 
Among other asymptotic normality results for the BT model \citep{han2020asymptotic, MR4560195, gao2023uncertainty}, Theorem~\ref{normal_MLE} provides the first result that allows for heterogeneous graphs in the sense that the upper and lower bounds of the vertex degree $N_i$ can have different orders as $n\to\infty$ (although in a controlled manner). In the homogeneous case, our assumptions are optimal in terms of the leading-order sparsity, similar to \cite{fan2022ranking}. For example, when $\zeta_{n,-}\asymp \zeta_{n,+}$, \eqref{an_hsbm_condition} becomes $\zeta_{n,-} \gg (\log n)^8$ while \cite{fan2022ranking} requires $\zeta_{n,-} \gtrsim \text{poly}(\log n).$

On the other hand, the optimality of the graph heterogeneity part of the assumptions is less well-understood. Our simulation results in Section \ref{sec:add_sim} suggest that increasing the maximum expected degree alone may undermine asymptotic normality. This indicates that an optimal condition would likely involve a trade-off between the maximum and minimum expected degrees. Nevertheless, identifying an optimal condition is quite challenging and beyond the scope of the present work. Due to the limited existing literature on asymptotic normality in heterogeneous settings, direct comparisons with other results in this respect are difficult.
\end{remark}

\begin{remark}
A heuristic computation to obtain the results in Theorem \ref{normal_MLE} is the following: 
\begin{align}\label{AN_MLE_decomposition}
& \uu_k - u_k^* = \frac{\partial_{k} l_1(\bu^*)}{\partial_{kk} l_1(\bu^*)} + o_p\left(\frac{1}{\sqrt{N_k}}\right) & k \in [n].
\end{align}
Instead of focusing on a single parameter, one can make use of \eqref{AN_MLE_decomposition} to further consider the inference on the space of different parameters, such as the difference between two parameters. Since our main goal is to provide asymptotic properties for the likelihood-based estimators, the general inference on the parameter space is beyond the scope of this work. We leave them for future investigation. 
\end{remark}

For the QMLE, computing the asymptotic variance only uses probabilities of relative ranking of at most triple-wise comparisons. Let 
{ \small   \begin{eqnarray*}
	&&\theta_{k,2}(\bu^*; T) =  \sum_{\{j\} \subset {T}\setminus\{k\}} \frac{\e{u_k^* + u_j^*}}{\{ \e{u_k^*} + \e{u_j^*}\}^2}; \\
	&&\theta_{k,3}(\bu^*; T) = \theta_{k,2}(\bu^*; T)  \\
	 &&   +2\sum_{\{j, t\}  \subseteq {T}\setminus\{k\}} \bigg[\frac{\e{u_k^* }}{\e{u_k^*} + \e{u_j^*}+ \e{u_t^*}} - \frac{\e{2u_k^*}}{\{\e{u_k^*} + \e{u_j^*}\} \{\e{u_k^*} + \e{u_t^*}\} } \bigg], 
\end{eqnarray*}}
we define
\begin{eqnarray}\label{rho_k2}
	\rho_{k,2}^2(\bu^*) =  \frac{\left(\sum_{i: k\in T_i} \theta_{k,2}(\bu^*; T_i) \right)^2}{\sum_{i: k\in T_i} \theta_{k,3}(\bu^*; T_i) }.
\end{eqnarray}
 The asymptotic normality of the QMLE is summarized in the following theorem.
\begin{thm}[Asymptotic normality of QMLE]\label{normal_QMLE}
Under Assumptions \ref{ap:2}, \ref{ap:1} and \ref{ap:6} or Assumptions \ref{ap:2}, \ref{ap:1} and \ref{ap:5}, for any fixed $k \in [n]$, the QMLE $\widetilde{\bu}$ satisfies 
	\begin{align*}
		&{\rho_{k,2}(\bu^*)}(\widetilde{u}_k - u_k^*) \to N(0,1)& n\to\infty. 
	\end{align*} 
\end{thm}

\begin{remark}
	The calculation of QMLE ignores the dependence among pairwise comparisons obtained from edge-wise breaking. Since asymptotic normality is derived by averaging over independent comparisons (edges) rather than within each edge, the dependent structures within an edge have only a local effect that impacts the form of the asymptotic variance of the QMLE but not the asymptotic normality result itself.
\end{remark}

This result also matches Proposition \ref{lm:3}, which states that QMLE is essentially a moment estimator. Note
		$$ \widetilde{u}_k - u_k^* = \frac{\partial_{k} l_2(\bu^*)}{\partial_{kk} l_2(\bu^*)} + o_p\left(\frac{1}{\sqrt{N_k}}\right). $$
Unlike the log-likelihood function $l_1(\bu^*) $, $\mathbb{E}[-\nabla^2 l_2(\bu^*)] \neq \mathbb{E}[\nabla l_2(\bu^*)^\top\nabla l_2(\bu^*)]. $ 

Compared to the full MLE, the QMLE uses all the data information for estimation but requires less computational cost for uncertainty quantification. Consider again the ideal case where $|T_i| = M$ for each $i\in [N]$. The computational cost of $\eqref{rho_k2}$ is of order $N_k M(M-1)$, which is significantly reduced as opposed to $N_k (M!)$ for moderately large $M$ (e.g., when $M = 8$, the latter is $720$ times greater than the former). On the other hand, our numerical studies in Section \ref{sec: an_numerical} show that the standard deviation of the QMLE is just slightly larger than the full MLE. This observation implies that the QMLE is a good alternative to the full MLE for practical uncertainty quantification in the PL model.

\section{Asymptotic Results for Deterministic Comparison Graphs}\label{sec:55} 

In this section, we prove the asymptotic results for the likelihood-based estimators when the underlying comparison graph sequence is deterministic. The only randomness comes from the comparison outcome.

\subsection{Uniform Consistency}\label{sec:5}

Theorems~\ref{main_1} and \ref{main_2} can be deduced from a more general consequence stated for deterministic comparison graph sequences. To explain this result, we generalize the rapid expansion property of pairwise graph sequences introduced in \cite{han2022general} to the setup of hypergraph sequences. Toward this, we introduce the concept of modified Cheeger constant and admissible sequences in the hypergraph setting.

\begin{Definition}[Modified Cheeger constant]\label{def:ch}
Recall that $\partial U$ denotes the boundary edge set of $U$. 
Given a hypergraph $\HH(V, E)$ with $V = [n]$, let  
\begin{align*}
&h_\HH(U) = \frac{|\partial U|}{\min\{|U|, |U^\complement|\}}&U\subset [n].
\end{align*}
The modified Cheeger constant of $\HH$ is defined as $h_\HH = \min_{U \subset [n]}h_\HH(U)$. 
\end{Definition}

By definition, $\HH$ is connected if and only if $h_\HH>0$. Moreover, $h_\HH$ is nondecreasing when adding more edges to $\HH$. This suggests that a larger value of $h_\HH$ often suggests better connectivity of $\HH$.

\begin{Definition}[Admissible sequences]\label{def:ad}
	Given a connected hypergraph $\HH(V, E)$, a strictly increasing sequence of vertices $\{A_j\}_{j\in [J]}$ is called admissible  if at least half of the boundary edges of $A_j$ lie in $A_{j+1}$ for $j<J$, that is, 
	\begin{align}
		&\frac{|\{e\in \partial A_{j}: e\subseteq A_{j+1}\}|}{|\partial A_j|}\geq \frac{1}{2}& 1\leq j <J. \label{re_1}
	\end{align}
\end{Definition}

Admissible sequences arise when applying the chaining method to establish uniform consistency of the MLE/QMLE. A crucial step in the analysis involves constructing a nested sequence of vertex sets that links objects exhibiting the most extreme estimation errors. Such a sequence is random and satisfies the expansion property \eqref{re_1} that defines admissible sequences. To control the estimation errors along all possible admissible sequences, we introduce the following definition concerning the graph sequence topology.

\begin{Definition}[Rapid expansion]\label{def:RE}
	Let $\HH_n(V_n, E_n)$ be a sequence of connected hypergraphs with $V_n = [n]$. 
	Let $\mathscr A_{n}$ denote the set of admissible sequences in $\HH_n$. 
	$\{\HH_n\}_{n\in\mathbb N}$ is said to be rapidly expanding (RE) if
	\begin{align}
		&\Gamma^{\RE}_{n} := \max_{\{A_j\}_{j=1}^{J}\in\mathscr A_{n}}\sum_{j=1}^{J-1}\sqrt{\frac{\log n}{h_{\HH_n}(A_j)}}\to 0& n\to\infty,\label{myRE}
	\end{align}
	where $h_{\HH_n}(A_j)$ is defined in Definition \ref{def:ch}.
\end{Definition}

As we will see in Theorem~\ref{main:general} shortly, $\Gamma^{\RE}_{n}$ governs the uniform convergence rate of the MLE/QMLE. Moreover, it can be further bounded by $(J_n-1)\sqrt{\log n/h_{\HH_n}}$, where $h_{\HH_n}$ is the modified Cheeger constant of $\HH_n$ and $J_n$ is the maximum length of all admissible sequences in $\HH_n$. The intuition of $J_n$ is best understood when $\HH_n$ is a pairwise graph. In this setting, if we further restrict the set of admissible sequences by replacing the constant $1/2$ in \eqref{re_1} with $1$, then it can be verified by definition that $(J_n-1)$ coincides with the graph diameter. Hence, $(J_n-1)$ defined for admissible sequences serves as a generalization of the notion of graph diameter. The RE property thus provides a quantitative characterization of the asymptotic connectivity of a graph sequence $\{\HH_n\}_{n\in\mathbb N}$. We formulate it as the following assumption. 

\begin{Assumption}\label{ap:11}
The hypergraph sequence $\{\HH_n\}_{n\in\mathbb N}$ is RE. 
\end{Assumption}

We are ready to state the main result. 

\begin{thm}\label{main:general}
Under Assumptions \ref{ap:2}-\ref{ap:1} and \ref{ap:11}, for all sufficiently large $n$, with probability at least $1-n^{-3}$, both the marginal MLE $\buu$ and QMLE $\widetilde{\bu}$ in Section \ref{sec:3} uniquely exist and satisfy $\|\bm w -\bu^*\|_{\infty}\lesssim \Gamma_n^{\RE}\to 0$, where $\bm w = \buu$ or $\widetilde{\bu}$. In particular, both $\buu$ and $\widetilde{\bu}$ are uniformly consistent for $\bm u^*$. 
\end{thm}

The main idea of the proof is to create a chain between vertices with large positive and negative estimation errors while ensuring a smooth transition of errors along the chain. This technique works for both the marginal MLE and QMLE since it relies solely on certain estimating equations for constructing the chain. It differs from the popular approach based on regularization plus leave-one-out analysis \citep{MR3953449, fan2022ranking, chen2022optimal, chen2022partial}, which can be used to obtain sharp phase transitions in homogeneous settings \citep{chen2022partial} but is also more restrictive in scope. (Leave-one-out analysis is inherently a perturbation method and requires sufficient symmetry or homogeneity in the underlying graph structure to be effective.) 

To see that Theorem~\ref{main:general} is nonvacuous, one needs to verify that Assumption~\ref{ap:11} is not purely conceptual but can be realized under suitable graph sampling scenarios. The following lemma provides an affirmative answer to address this issue.

\begin{Lemma}\label{main:REs}
Let $\HH_n(V_n, E_n)$ be a hypergraph sequence with $V_n = [n]$. Under Assumptions~\ref{ap:2}-\ref{ap:1}, the following statements hold: 
\begin{enumerate}
\item [(1)] If $\HH_n$ is sampled from a NURHM in Section \ref{am}, then under additional Assumption \ref{ap:4}, $\HH_n$ is RE a.s., with the corresponding $\Gamma_n^{\RE}$ satisfying
\begin{align*}
&\Gamma_n^{\RE}\lesssim \sqrt{\frac{\xi^2_{n,+}(\log n)^3}{\xi^3_{n,-}}},
\end{align*}
where $\xi_{n, \pm}$ are the same as \eqref{myxis}.
\item [(2)] If $\HH_n$ is sampled from an HSBM in Section \ref{hsbm}, then under additional Assumption \ref{ap:3}, $\HH_n$ is RE a.s., with the corresponding $\Gamma_n^{\RE}$ satisfying
\begin{align*}
&\Gamma_n^{\RE}\lesssim \sqrt{\frac{(\log n)^3}{\zeta_{n,-}}},
\end{align*} 
where $\zeta_{n,-}$ is the same as \eqref{myzeta}. 
\end{enumerate}
\end{Lemma}
Lemma~\ref{main:REs} generalizes \cite[Proposition 2]{han2022general} to the setting of random hypergraphs. 
To sketch the proof, note that $\Gamma_n^{\RE}\leq (J_n-1)\sqrt{\log n/h_{\HH_n}}$ (as per definition), where $h_{\HH_n}$ is the modified Cheeger constant of $\HH_n$ and $J_n$ is the maximum length of all admissible sequences in $\HH_n$. Therefore, it suffices to obtain an upper bound on $J_n$ and a lower bound on $\HH_n$, respectively. The upper bound is standard and follows from a degree concentration argument that applies simultaneously for both NURHM and HSBM. The lower bound, however, is model-specific and requires different machinery, especially for HSBM. Indeed, for HSBM, we observe that for any admissible sequence $\{A_j\}_{j\in [J]}$, except for at most a constant number of times (which depends only on the number of communities), there exists at least one community $V_k$ such that either $|A_j\cap V_k|$ grows exponentially or $|A^\complement_j\cap V_k|$ decreases exponentially, both at a constant rate independent of $n$. Establishing this observation goes beyond the existing proof in \cite{han2022general} and requires an elaborate combinatorial analysis of the expansion rate of admissible sequences. 

The combination of Lemma \ref{main:REs} and Theorem \ref{main:general} together with the Borel--Cantelli lemma leads to Theorems \ref{main_1} and \ref{main_2}. 

\begin{remark}\label{new_adding}
NURHM and HSBM are two examples of random hypergraph models for which the RE parameter $\Gamma_n^{\RE}$ can be explicitly bounded in terms of graph sparsity parameters. Slight variations of these models lead to other interesting random graph constructions where similar RE bounds apply. For example, by altering the comparison probabilities within each community of the HSBM while maintaining their asymptotic order, one can obtain a degree-corrected HSBM without affecting the RE property of the graph sequence.
\end{remark}

\subsection{Asymptotic Normality}\label{sec:6+}

We provide a general asymptotic normality result for deterministic graph sequences. For clarity, we introduce some additional definitions concerning the comparison graph structure.

Let $\N_{n,+} $ and $\N_{n,-}$ be the maximum and minimum vertex degree in $\HH_n$ respectively:
	\begin{align}
		&\N_{n,+} = \max_{k\in [n]} \ndeg_k& \N_{n,-} = \min_{k\in [n]} \ndeg_k.\label{mydn}
	\end{align}
For $j\neq k\in [n]$, denote the number of shared edges between $j$ and $k$ as 
\begin{align}
\N_{jk} = \left|\{i: \{j, k\}\subseteq T_i\}\right|.\label{mydij}
\end{align} 
The quantity $\N_{jk}$ is either 1 or 0 in the pairwise comparison setting (assuming $\HH_n$ is simple) but may diverge when multiple comparisons exist. To better describe the edge-sharing phenomenon, we define 
\begin{align}
\cc = \max_{j\neq k\in [n]}\frac{\N_{jk}}{\ndeg_j}.\label{mycor}
\end{align} 
The value of $\cc$ manifests the strength of the correlation of estimation across the objects. When dealing with pairwise comparisons, $\cc\leq 1/\N_{n, -}=o(1)$, so that estimation between different objects is asymptotically independent. In a hypergraph, however, $\cc$ can arbitrarily approach one when two objects share a large number of edges asymptotically. 
	 
To establish asymptotic normality, we need to compute the Hessians of the marginal log-likelihood and quasi-log-likelihood functions.  
Denote by $l(\cdot)$ the marginal log-likelihood or the quasi-log-likelihood and $\bm w$ the corresponding marginal MLE or QMLE estimator. Let $\cH(\bu) = \nabla^2 l(\bu)$. It follows from the direct computation that for the marginal MLE,
\begin{align}
	&\{\cH(\bm u)\}_{kk'} = 
	\begin{cases}
		\sum_{i: \{k, k'\}\subseteq T_i}\sum_{j\in [r_i(k)\wedge r_i(k')\wedge y_i]}\frac{\e{u_k}\e{u_{k'}}}{\{\sum_{t \geq j}\e{u_{\pi_i(t)}}\}^2} & k\neq k'\\
		\\
		-\sum_{i: k\in T_i}\sum_{j\in [r_i(k)\wedge y_i ]}\frac{\e{u_k}(\sum_{t \geq j, t\neq r_i(k)}\e{u_{\pi_i(t)}})}{(\sum_{t \geq j}\e{u_{\pi_i(t)}})^2} & k = k',
	\end{cases}
	\label{wearediff}
\end{align}
while for the QMLE,
\begin{align}	
		&\{\cH(\bm u)\}_{kk'} = 
	\begin{cases} 
		\sum_{i: \{k, k'\}\subseteq T_i} \frac{\e{u_k}\e{u_{k'}}}{(\e{u_k} + \e{u_k'})^2} & k\neq k'\\
		\\
		- 	\sum_{j \in [n]} \sum_{i: \{k, j\}\subseteq T_i}\frac{\e{u_k}\e{u_{j}}}{(\e{u_k} + \e{u_j})^2} &  k = k'.
	\end{cases}
	\label{binyan}
\end{align}

Although \eqref{wearediff} and \eqref{binyan} look different, a crucial shared property we will use later is that both are the negatives of some weighted graph Laplacian matrices. Moreover, we should point out that $\cH(\bu)$ is random even when the comparison graph sequence is deterministic unless $\bm w$ is the choice-one MLE ($y_i=1$ for all $i\in [N]$) or QMLE. This presents additional challenges in the analysis of the PL model. 
 
To get a deterministic quantity, define $\cH^*(\bu^*):=\mathbb E[\cH(\bu^*)]$. It is easy to see that $-\cH^*(\bu^*)$ is also a weighted graph Laplacian matrix with weight matrix $\W$, where $\W_{ij} = [\cH^*(\bu^*)]_{ij}$ for $i\neq j\in [n]$ and $\W_{ii} = 0$. The associated degree matrix $\D$ is a diagonal matrix with $\D_{ii} = -[\cH^*(\bu^*)]_{ii}$. The (symmetric) normalized version of $-\cH^*(\bu^*)$ is defined as 
\begin{equation*}
\mathcal{L}_{\sym} = -\D^{-1/2}\cH^*(\bu^*)\D^{-1/2} = \I - \A, 
\end{equation*}
where $\A = \D^{-1/2}\W\D^{-1/2}$ is the normalized weight matrix. The eigenvalues of $\mathcal{L}_{\sym}$ are denoted by $0 = \lambda_1(\mathcal{L}_{\sym})\leq\cdots\leq\lambda_n(\mathcal{L}_{\sym})\leq 2$. Define
\begin{align}
\mix = \min \{\lambda_2(\mathcal{L}_{\sym}),  2-\lambda_n(\mathcal{L}_{\sym})\}.\label{mymix}
\end{align}
The quantity $\mix$ is related to the mixing of random walk on the graph associated with $-\cH^*(\bu^*)$ \citep{MR1421568}. In our setting, it serves as a measure of the heterogeneity of the graph and will appear as a convergence rate factor in the operator expansion analysis of $\mathcal{L}_{\sym}^\dagger$. More precisely, $\lambda_2(\mathcal{L}_{\sym})$ measures the conductance of the graph and thus the heterogeneity, while $2-\lambda_n(\mathcal{L}_{\sym})$ measures the periodicity of the graph, that is, how close the graph is being bipartite. It is $\lambda_2(\mathcal{L}_{\sym})$ that often matters.

The assumption required to establish asymptotic normality for the choice-one MLE and the QMLE can be stated as follows:
\begin{Assumption}\label{ap:deterministic}
	The hypergraph sequence $\{\HH_n\}_{n\in\mathbb N}$ satisfies
	\begin{align}
			\lim_{n \to \infty} \frac{\log n}{\mix}\times\max\left\{\frac{(\Gamma^{\RE}_n)^2 \N_{n,+}}{\sqrt{\N_{n,-}}}, \ \sqrt{\cc}\right\}  = 0,\label{coupledq}
	\end{align}
	where $\Gamma_n^{\RE}$, $\cc$, and $\mix$ are defined in \eqref{myRE}, \eqref{mycor}, and \eqref{mymix}, respectively.  
\end{Assumption}
Assumption~\ref{ap:deterministic} captures the joint influence of several graph-theoretic attributes. In order for the left-hand side of \eqref{coupledq} to vanish in the limit as $n\to\infty$, both $\Gamma_n^\RE$ and $\cc$ must decay to zero at a sufficiently fast rate. The parameter $\Gamma_n^\RE$ quantifies the asymptotic connectivity of $\{\HH_n\}_{n\in\mathbb N}$, with faster decay indicating better connectivity, while $\cc$ measures the correlation of estimates associated with distinct objects, where smaller values correspond to weaker correlations. Consequently, Assumption~\ref{ap:deterministic} may be interpreted as imposing that $\{\HH_n\}_{n\in\mathbb N}$ remains asymptotically well-connected and that inter-object estimation dependencies diminish appropriately. Moreover, the parameter $\mix$ should remain bounded away from zero (though it may get closer asymptotically at a controlled speed) to prevent severe bottleneck structures, thereby ensuring a controlled level of graph heterogeneity across $\{\HH_n\}_{n\in\mathbb N}$.

For the marginal MLE in general, however, due to the randomness of $\cH(\bm u)$, our analysis requires bounding an extra remainder term for which we apply a leave-one-out perturbation analysis \citep{gao2023uncertainty}. Since we do not assume a model for comparison graphs, we will need an additional assumption on the Hessians of the leave-one-out log-likelihood functions. For $k\in [n]$, define the leave-one-out log-likelihood function by 
\begin{align}\label{leave_one_out_likelihood}
	&l_1^{(-k)}(\bm u) = \sum_{i: k \notin T_i} \sum_{j\in [y_i]}\left[u_{\pi_i(j)} - \log\left(\sum_{t=j}^{m_i}\e{u_{\pi_i(t)}}\right)\right]& k \in [n].
\end{align}
Denote by $\cH^{(-k)}(\bu):= \nabla^2l_1^{(-k)}(\bm u)$ and ${\cH^*}^{(-k)}(\bu) = \mathbb E[\cH^{(-k)}(\bu)]$. Note ${\cH^*}^{(-k)}(\bu)$ is nonrandom and one can check that $-{\cH^*}^{(-k)}(\bu)$ is a weighted graph Laplacian matrix on $[n]\setminus \{k\}$. Then we let
\begin{align}\label{myleave}
\lambda^{\leave}_2 = \min_{k\in [n]}\lambda_2(-{\cH^*}^{(-k)}(\bu^*))\geq 0.
\end{align}
The additional assumption can be stated as follows:
\begin{Assumption}\label{ap:deterministic+}
	The hypergraph sequence $\{\HH_n\}_{n\in\mathbb N}$ satisfies
	\begin{align*}
\lim_{n \to \infty} \frac{\log n}{\mix}\times\max\left\{\frac{\Gamma_n^\RE \N_{n,+}^{3/2}\sqrt{\log n}}{\sqrt{\N_{n,-}}\lambda_2^\leave}, \ \frac{\N_{n,+}\log n}{\sqrt{\N_{n,-}}\lambda_2^\leave}\right\}  = 0,
	\end{align*}
	where $\Gamma^{\RE}_n$, $\mix$, and $\lambda_2^\leave$ are defined in \eqref{myRE}, \eqref{mymix}, and \eqref{myleave}, respectively. 
\end{Assumption}

Aside from the conditions imposed on $\Gamma_n^\RE$, $\cc$, and $\mix$, Assumption~\ref{ap:deterministic+} additionally requires that the leave-one-out spectral parameter $\lambda_2^\leave$ diverge at a sufficiently rapid rate. Since a large value of $\lambda_2^\leave$ indicates strong graph conductance even in the worst-case scenario of removing a single vertex, Assumption~\ref{ap:deterministic+} ensures that the graph sequence $\{\HH_n\}_{n\in\mathbb N}$ behave similarly to a homogeneous graph sequence, a common scenario where the leave-one-out technique becomes effective.

Since both $\lambda_2(\mathcal L_\sym)$ and $\lambda_2^\leave$ are concerned with graph heterogeneity, one may wonder if $\lambda_2^\leave$ can be further lower bounded using $\lambda_2(\mathcal L_\sym)$ so that Assumption~\ref{ap:deterministic+} can be stated without $\lambda_2^\leave$. The answer to this question is in general negative. For instance, consider a hypergraph where all edges involving $k'$ also involve $k$. When all edges containing $k$ are removed, $k'$ is disconnected from the rest of the graph regardless of the connectivity of the original graph. Fortunately, such cases can be excluded by limiting the total number of edges shared by any two vertices, which is precisely what $\cc$ stands for.

\begin{Lemma}\label{lastone}
If $\cc/\mathcal \lambda_2(\mathcal L_\sym)\to 0$, then $\lambda_2^\leave\gtrsim \N_{n,-}\lambda^2_2(\mathcal L_\sym)$. 
\end{Lemma}

Consequently, under Assumption~\ref{ap:deterministic}, a sufficient assumption for Assumption~\ref{ap:deterministic+} without introducing $\lambda_2^\leave$ is the following:
 
\begin{Assumption}\label{ap:deterministic++}
	The hypergraph sequence $\{\HH_n\}_{n\in\mathbb N}$ satisfies
	\begin{align*}
\lim_{n \to \infty} \frac{\log n}{\mix^3}\times\max\left\{\frac{\Gamma_n^\RE \N_{n,+}^{3/2}\sqrt{\log n}}{\N^{3/2}_{n,-}}, \ \frac{\N_{n,+}\log n}{\N^{3/2}_{n,-}}\right\}  = 0,
	\end{align*}
	where $\Gamma^{\RE}_n$ and $\mix$ are defined in \eqref{myRE} and \eqref{mymix}. 
\end{Assumption}
We are now ready to give the main result for asymptotic normality. 
\begin{thm}\label{thm:and}
Under Assumptions \ref{ap:2}, \ref{ap:1}, and \ref{ap:deterministic}, for any fixed $k \in [n]$, $\bm w$ satisfy the following:
\begin{enumerate}
\item [(i)] If $\bm w$ is the choice-one MLE, then ${\rho_{k,1}(\bu^*)}(w_k - u_k^*) \to N(0,1)$ 
\item [(ii)] If $\bm w$ is the QMLE, then ${\rho_{k,2}(\bu^*)}(w_k - u_k^*) \to N(0,1)$  
\end{enumerate} 
where $\rho_{k, 1}$ and $\rho_{k,2}$ are defined in \eqref{rho_k1} and \eqref{rho_k2}, respectively. 
Furthermore, if either Assumption~\ref{ap:deterministic+} or \ref{ap:deterministic++} holds, then statement (i) also holds for the general marginal MLE. 
\end{thm}
The proof of Theorem \ref{thm:and} is based on a truncated error analysis of the normalized Hessian matrices using the Neumann series expansion. This approach contains new ingredients such as obtaining entrywise estimates on the higher-order moments of the Hessian matrices. For the choice-one MLE or QMLE, truncated error analysis alone is sufficient to yield an asymptotic normality result that applies to heterogeneous graphs, generalizing the existing results obtained using the state-of-the-art leave-one-out analysis \citep{gao2023uncertainty, fan2022ranking}. For other marginal MLEs, additional complexity arises from the randomness in the Hessian matrices. To address this, we apply both truncated error analysis (Assumption \ref{ap:deterministic}) and leave-one-out perturbation analysis (Assumption~\ref{ap:deterministic+} or \ref{ap:deterministic++}) to obtain the desired result.

Assumptions~\ref{ap:deterministic}-\ref{ap:deterministic++} may look rather complex at first glance. Nevertheless, for both NURHM and HSBM, explicit bounds on each graph parameter can be obtained using $\xi_{n, \pm}$ and $\zeta_{n, \pm}$. Since our analysis overall does not exploit specific properties of HSBM, we present the result for NURHM only and identify HSBM as a special instance of NURHM with $\xi_{n, +}=\zeta_{n, +} $ and $\xi_{n, -} = \zeta_{n, -}$. We include a remark to indicate where improvements can be made by leveraging the community structure of HSBM. These results are stated in the following lemma. 

\begin{Lemma}\label{newtech}
Let $\HH_n(V_n, E_n)$ be a hypergraph sequence with $V_n = [n]$. Suppose that $\HH_n$ is sampled from NURHM in Section~\ref{am}. If Assumptions \ref{ap:2}-\ref{ap:1} hold and $\xi^7_{n,-} \gtrsim(\xi^6_{n,+}\log n)$, then a.s., for all sufficiently large $n$, the following estimates hold:
\begin{align}
\xi_{n, -}&\lesssim \N_{n, -}\leq \N_{n, +}\lesssim \xi_{n, +};\label{2024111}\\
\mix&\gtrsim\left(\frac{\xi_{n, -}}{\xi_{n, +}}\right)^2, \ \lambda_2^\leave\gtrsim \frac{\xi_{n,-}^3}{\xi_{n,+}^2}\label{2024222};\\
\cc & \lesssim   \max \left\{\frac{\xi_{n,+}}{n\xi_{n,-}}, \frac{\log n}{\xi_{n,-}} \right\}
\label{2024333}.
\end{align}
\end{Lemma}

Estimate \eqref{2024111} follows from a standard degree concentration argument and holds under the much weaker assumption $\xi_{n,-}\gtrsim\log n$, while \eqref{2024222} is more technical and involves applying Cheeger's inequalities (from both ends of the spectrum) to control the spectral gap \citep{MR1421568, bauer_jost_2013}. Due to the nested model structure in NURHM, the lower bound in \eqref{2024222} is better than directly applying Lemma~\ref{lastone} to the estimates of $\mix$ in \eqref{2024222}, which would yield $\lambda_2^\leave\gtrsim \xi^5_{n,-}/\xi^4_{n,+}$. Estimate \eqref{2024333} can be further refined in HSBM when the sizes of communities are of the same order.

The asymptotic normality results in Theorems~\ref{normal_MLE} and \ref{normal_QMLE} follow by combining Lemmas~\ref{main:REs} and \ref{newtech} and Theorem~\ref{thm:and}. 
\begin{remark}\label{ksjsjf}
The estimate on $\cc$ can be improved in HSBM using the community structure. Let $V_1, \ldots, V_K$ be the $K$ communities in an HSBM. It can be shown that 
\begin{align*}
\cc & \lesssim \min\left\{\frac{\max\{n^{-1}\zeta_{n,+}, \log n\}}{\zeta_{n,-}}, \frac{1}{\min_{i\in [K]}|V_i|}\right\}.
\end{align*}
When the sizes of communities have the same order, that is, $\min_{i\in [K]}|V_i|\asymp n$, $\cc\lesssim\max\{n^{-1}, (\log n)/\zeta_{n,-}\}$, which improves the general result \eqref{2024333} in NURHM. The details of the proof can be found in the appendices. 
\end{remark}

 \section{Numerical Results}\label{sec:7}
In this section, we conduct numerical experiments to support the theoretical findings in the previous sections. This includes verifying both uniform consistency and asymptotic normality of the likelihood-based estimators using synthetic data. We also demonstrate that appropriate control of heterogeneity is necessary to ensure asymptotic normality; see the discussion after Assumption~\ref{ap:5}. Moreover, we apply the QMLE to analyze a real horse-racing dataset; the full/marginal MLE is not applied due to computational reasons, as clarified in Section \ref{horse}. All likelihood-based estimates are numerically computed using the MM algorithm in \cite{MR2051012}. The estimated standard deviations are computed using the plug-in method by replacing $\bm u^*$ with the estimated scores $\widehat{\bm u}$ and $\widetilde{\bm u}$ in \eqref{rho_k1} and \eqref{rho_k2}, respectively. All numerical experiments are conducted on a laptop with an Intel i7-4790 CPU.

\subsection{Uniform Consistency}\label{sec: un_numerical}
We verify the uniform consistency of both the (marginal) MLE and QMLE in the PL model using two different sampling models: NURHM and HSBM. The edge sizes of NURHM are chosen from $\{3,4,5,6,7\}$, while the edge size $\mm$ in HSBM is 5. The number of edges for each edge size in NURHM and the size of each community in HSBM are generated in a balanced fashion. In particular, for NURHM and each edge size $\ms$, we uniformly sample $0.02(\log n)^3$ edges of that size. Such choices lead to $N=0.1n(\log n)^3$ edges in total. For HSBM, we partition the objects into two communities: community one with size $0.4n$ and community two with size $0.6n$. There are three types of edge probabilities: the internal probabilities for edges within Community One and Community Two, and the cross-community probability for edges striding between the communities. We set their ratio as $5:3:2$. In total, we sample around $0.1n(\log n)^3$ comparisons. The utility vector $\bm u^*$ is selected uniformly between the range $[-0.5, 0.5]$.  

\begin{table}[b]
	\begin{center}
    \begin{tabular}{rrrrr}
      Number of objects $n$ & Number of edges $N$ in Section \ref{sec: un_numerical} & Number of edges $N$ in Section \ref{sec: an_numerical}\\\hline
      $200$ & 2,800  & 5,800 \\
      $400$ & 8,400  & 13,300\\
      $600$ & 15,600 & 21,600\\
      $800$ & 23,200 & 30,500\\
      1,000 & 32,000 & 39,800
    \end{tabular}
\end{center}
 \caption{Number of edges $N$ used in the experiments. The second and third columns are the $N$ used in the study of uniform consistency and asymptotic normality, respectively. $N$ is the same for both NURHM and HSBM in the same testing procedure. }\label{tab: size}
\end{table}

The total number of objects $n$ in the PL model ranges from 200 to 1000 with increments of 200. For each value of $n$, we repeat the experiment $300$ times and compute the average estimation errors in the $\ell_\infty$ norm. To get a thorough comparison, the likelihood-based estimators considered in this experiment include the choice-one MLE, the choice-2 MLE, the full MLE, and the QMLE. In our setup, the choice-one MLE does not exist in 3 out of 300 experiments while other estimators always exist. We drop the invalid experiments when computing summary statistics reported in Figure \ref{fig:1}.
 
 \begin{figure}[htbp]
  \centering 
   \subfigure[]{\includegraphics[width=0.45\linewidth, trim={0.2cm 0.7cm 2.3cm 2.5cm},clip]{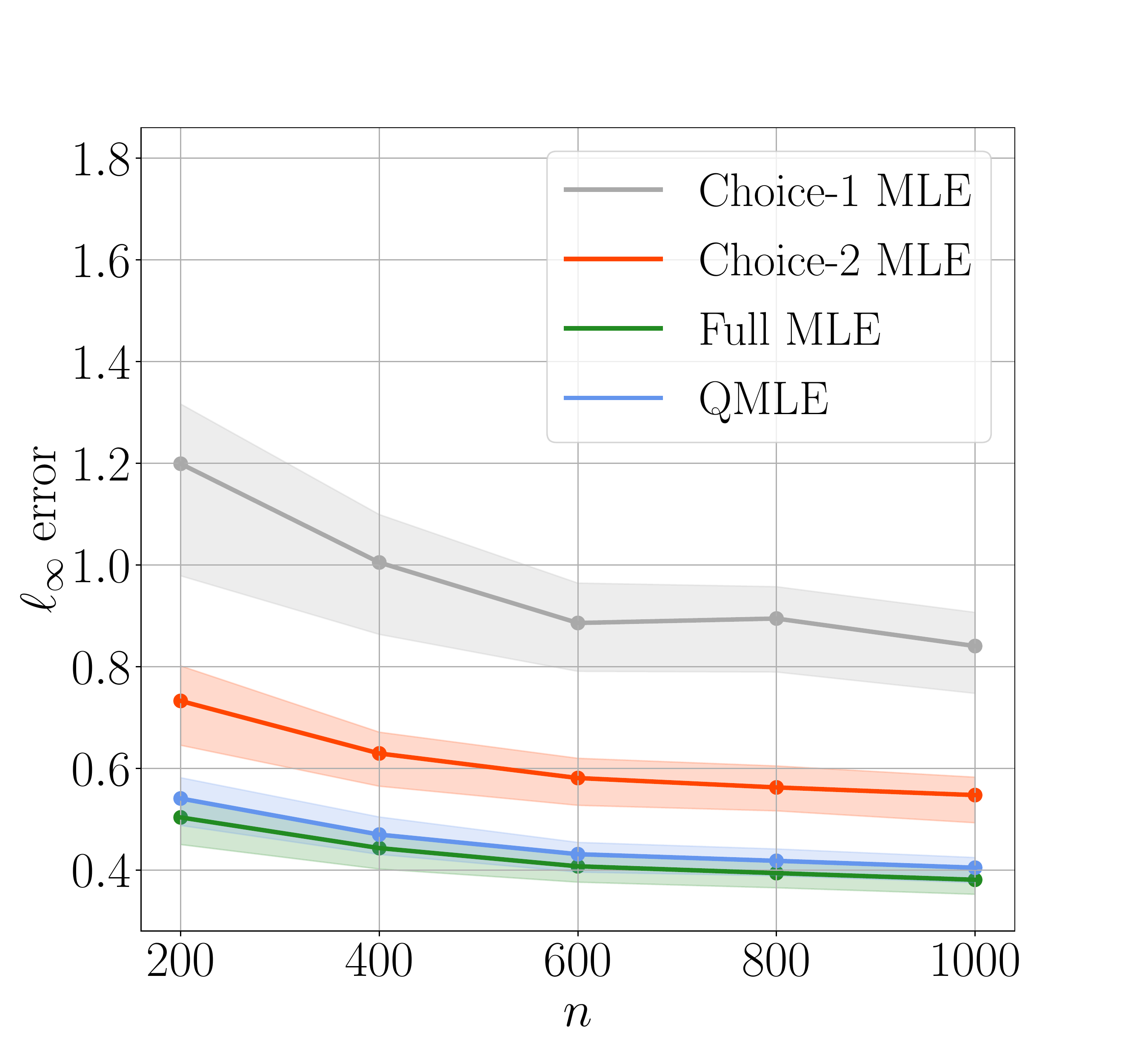}}\hfill
   \subfigure[]{\includegraphics[width=0.45\linewidth, trim={0.2cm 0.7cm 2.3cm 2.5cm},clip]{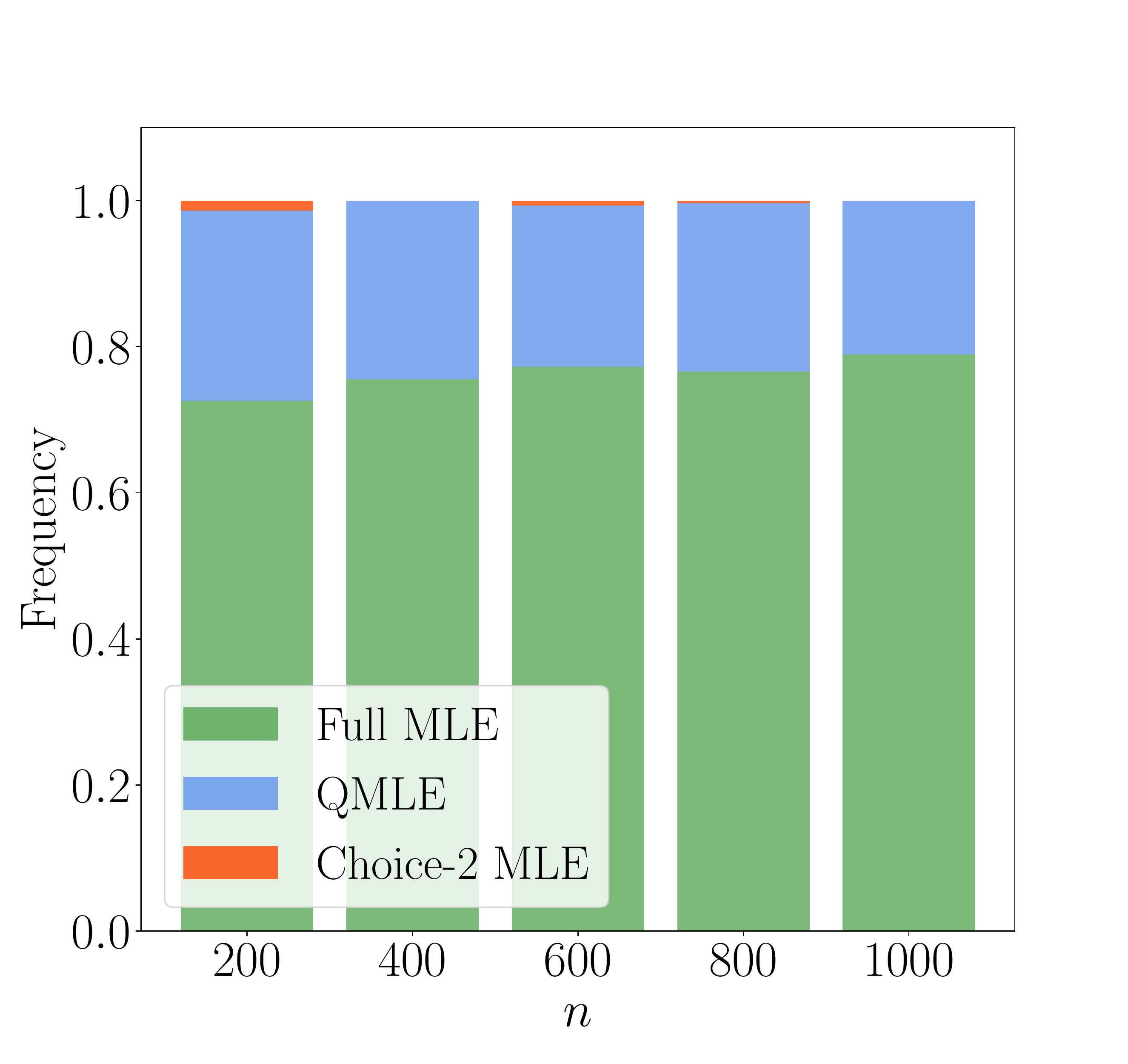}}\hfill
    \subfigure[]{\includegraphics[width=0.45\linewidth, trim={0.2cm 0.7cm 2.3cm 2.5cm},clip]{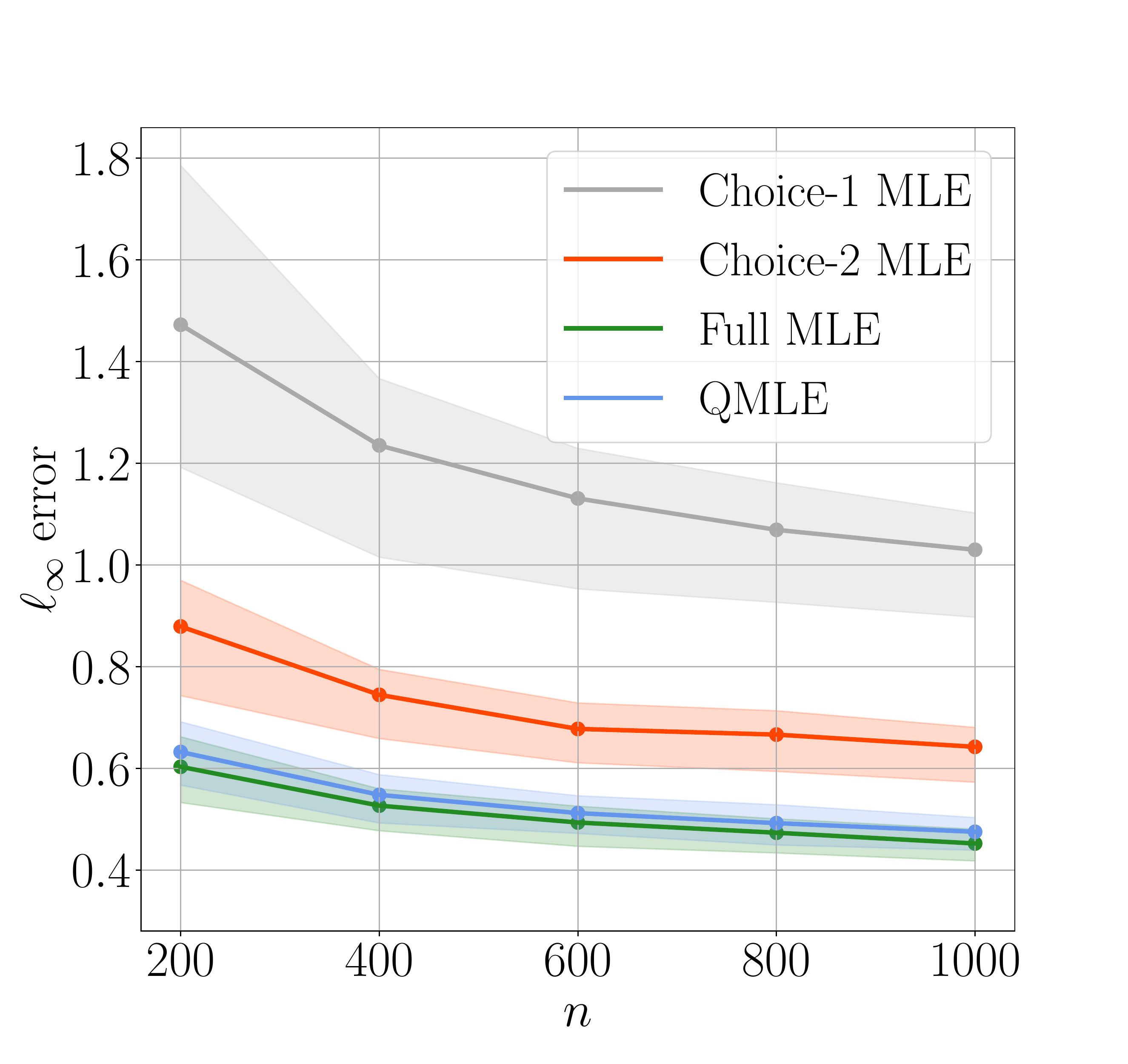}}\hfill
  \subfigure[]{\includegraphics[width=0.45\linewidth, trim={0.2cm 0.7cm 2.3cm 2.5cm},clip]{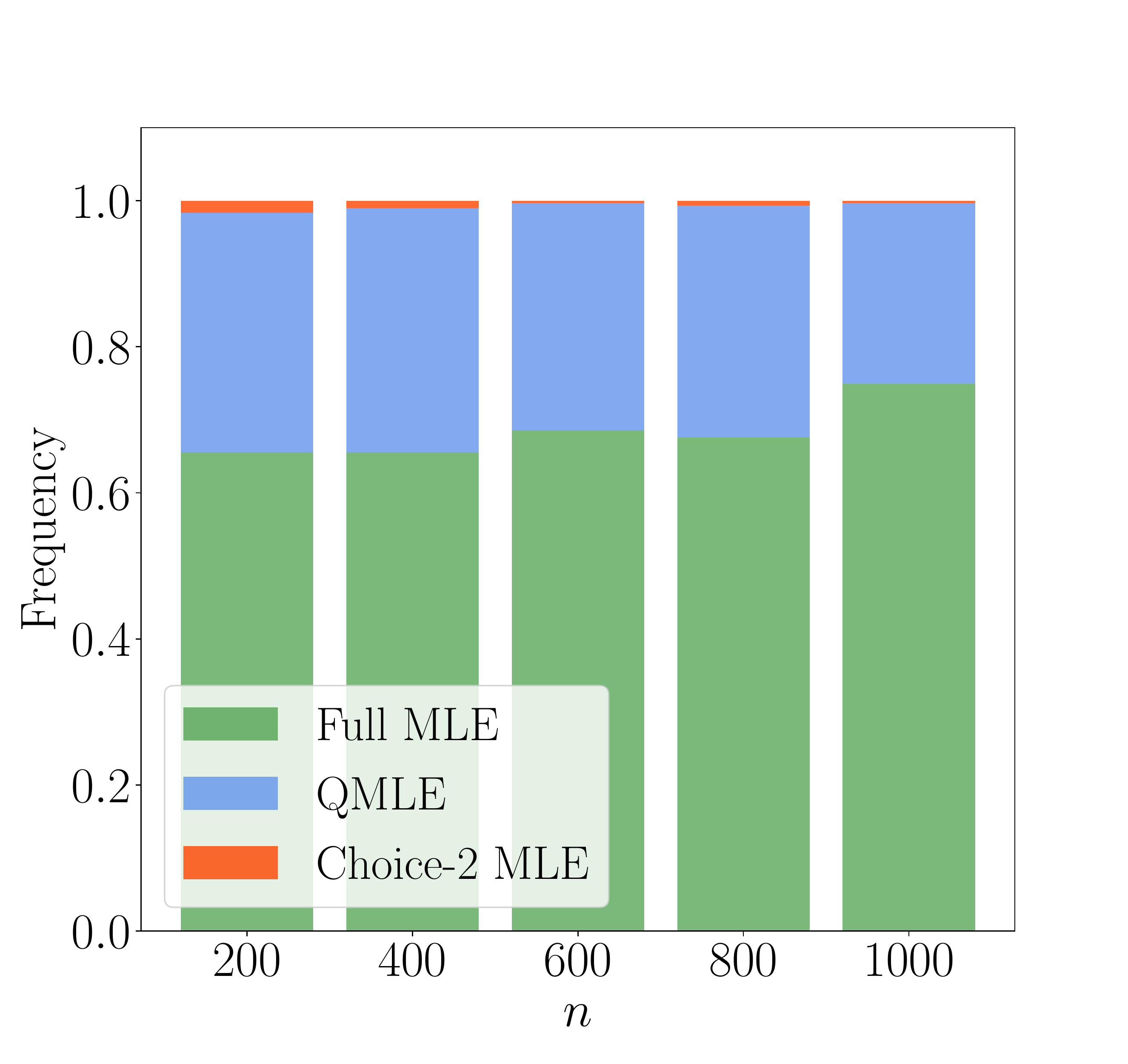}}\hfill
    \caption{Average $\ell_\infty$ error with corresponding quartiles for the marginal MLE (the choice-one MLE and the choice-2 MLE), the full MLE, and the QMLE and the frequencies of each estimator yielding the smallest $\ell_\infty$ error among others in $300$ experiments as the number of objects $n$ increases from $200$ to $1000$. (a)-(b): NURHM; (c)-(d): HSBM.}
   \label{fig:1}
\end{figure}

Figure \ref{fig:1} shows the average $\ell_\infty$ error of the choice-one MLE, the choice-2 MLE, the full MLE, and the QMLE in the PL model as the total number of objects $n$ increases. For each $n$, we also provide the frequency of each estimator obtaining the smallest $\ell_\infty$ error among others in the $300$ simulations. As expected, in both sampling scenarios, the $\ell_\infty$ errors of the estimators decrease to zero as $n$ increases, verifying the uniform consistency results in Section \ref{sec:4}. For fixed $n$, the full MLE has the smallest $\ell_\infty$ error. The QMLE achieves a competitive result as the full MLE despite using misspecified likelihood, and in around one in five experiments on average, the QMLE has a smaller $\ell_\infty$ error than the full MLE. Meanwhile, both choice-one and choice-two estimators in this example have a slower convergence rate due to using only a fraction of the data.

\subsection{Asymptotic Normality}\label{sec: an_numerical}

We demonstrate the asymptotic normality result for the (marginal) MLE and the QMLE in the PL model using two different sampling graphs: NURHM and HSBM. The experiment setup is almost identical to Section \ref{sec: un_numerical} except that we choose the edge sizes of NURHM as $\{3,4,5,6\}$ to accelerate computation and a slightly larger $N$. Specifically, for NURHM and each edge size, we sample $2.5n^{1.2}$ edges of that size uniformly. Such choices ensure that there are about $10n^{1.2}$ comparisons in total. (One can also consider $n(\log n)^8$, but it is even larger than the $10n^{1.2}$ in the finite sample case.) For HSBM, we also sample around $10n^{1.2}$ comparisons. The detailed numbers can be found in the right-most column in Table \ref{tab: size}. The calculation for the standard deviation of the full MLE is extremely slow if taking $\{3,4,5,6,7\}$ in NURHM as Section \ref{sec: un_numerical}. The results for NURHM and HSBM are reported in Tables \ref{An_result}-\ref{An_result_2}, respectively.

Tables \ref{An_result}-\ref{An_result_2} show the average estimated standard deviation, coverage probability of 95\% confidence interval, and the computational time of the estimated standard deviation for the choice-one MLE, the choice-2 MLE, the full MLE, and the QMLE in the PL model under different values of $n$. For each method, the coverage probability is close to 95\% in both NURHM and HSBM, supporting our result of asymptotic normality. Meanwhile, there appears to be a clear trade-off between statistical efficiency and computational cost. 
The full MLE has the smallest standard deviation but is significantly more expensive than the other methods. This is because the computation of the variance of the full MLE requires averaging over all possible permutations on each edge. The choice-one MLE is the cheapest but yields the worst efficiency. The QMLE strikes a good balance between accuracy and cost. In particular, the QMLE has a computational time of only one in 20-40 of the full MLE but only sacrifices a small amount of statistical efficiency.

\begin{table}[h]
\centering
\begin{tabular}{l l  lll l lll}
& 	&Method  & &      \multicolumn{5}{c}{Sample size $n$}  \\ 
 & &	 &  & 200  &400   &600 &  800 & 1000  \\
	\hline
	\multirow{4}{*}{Standard Deviation}
& 	&QMLE &
	&0.131 &0.122 &0.118& 0.114& 0.112\\
	
& 	&Full MLE  &
	&0.124 &0.116 &0.111& 0.108& 0.105\\
	
& 	&choice-one MLE &
	&0.220 &0.205 &0.197& 0.191& 0.187\\
	
& 	&choice-2 MLE  &
	&0.159 &0.148 &0.141& 0.138& 0.135\\
	\hline
	
	\multirow{4}{*}{Coverage Probability}
& 		&QMLE &
	&0.950 &0.950 &0.950& 0.950& 0.950\\
	
&	&Full MLE &  
	&0.950 &0.950 &0.950& 0.949& 0.950\\
	
& 	&choice-one MLE &
	&0.951 &0.951 &0.950& 0.949 & 0.950\\
	
& 	&choice-2 MLE  &
	&0.950 &0.950 &0.950&0.950 & 0.950\\
	\hline
	
	\multirow{4}{*}{Computation Time(s)}
& 		&QMLE &
	&0.933 &2.149 &3.676& 4.755& 5.485\\
	
& 	&Full MLE  &
	&47.74 &111.7 &193.6& 247.0& 284.1\\
	
& 	&choice-one MLE &
	&0.644 &1.519 &2.599& 3.346& 3.820\\
	
& 	&choice-2 MLE  &
	&1.844 &4.335 &7.438& 9.588& 10.99\\
	\hline
\end{tabular}
\caption{NURHM: Simulation results are summarized over 300 replications.}\label{An_result}
\end{table}

\begin{table}[h]
\centering
\begin{tabular}{l l  lll l lll}
	& 	&Method  & &      \multicolumn{5}{c}{Sample size $n$}  \\ 
	& &	 &  & 200  &400   &600 &  800 & 1000  \\
	\hline
	\multirow{4}{*}{Standard Deviation}
	& 	&QMLE &
	&0.153 &0.143 &0.137&0.133 & 0.130\\
	
	& 	&Full MLE  &
	&0.146 &0.136 &0.131&0.127 & 0.124\\
	
	& 	&choice-one MLE &
	&0.279 &0.258 &0.247&0.240 & 0.234 \\
	
	& 	&choice-2 MLE  &
	&0.196 &0.182 &0.175&0.170 & 0.166\\
	\hline
	
	\multirow{4}{*}{Coverage Probability}
	& 		&QMLE &
	&0.946 &0.949 &0.948& 0.949& 0.949\\
	
	&	&Full MLE &  
	&0.945 &0.947 &0.948& 0.948& 0.948\\
	
	& 	&choice-one MLE &
	&0.946 &0.948 &0.949& 0.946& 0.943\\
	
	& 	&choice-2 MLE  &
	&0.947 &0.948 &0.949& 0.949& 0.949\\
	\hline
	
	\multirow{4}{*}{Computation Time(s)}
	& 		&QMLE &
	&1.208 &2.607 &4.202& 5.607& 6.833\\
	
	& 	&Full MLE  &
	&23.54 &54.08 &86.71& 117.9& 142.6\\
	
	& 	&choice-one MLE &
	&0.750 &1.721 &2.772& 3.767& 4.562\\
	
	& 	&choice-2 MLE  &
	&2.231 &5.144 &8.283& 11.26& 13.59\\
	\hline
\end{tabular}
\caption{ HSBM: Simulation results are summarized over 300 replications.}\label{An_result_2}

\end{table}

\subsection{Influence of Heterogeneity}\label{sec:add_sim}

To understand whether the additional balancing condition in the asymptotic normality result for HSBM in Theorems~\ref{normal_MLE}-\ref{normal_QMLE} is necessary (compared to its uniform consistency result in Theorem~\ref{main_2}), we conduct an additional numerical experiment to investigate the influence of graph heterogeneity on asymptotic normality. We consider an HSBM with $n = 300$ and $M=5$. There are two communities: 30 objects belong to Community One, and the remaining objects belong to Community Two. There are three types of edges: internal edges within Community One, internal edges within Community Two, and cross-community edges between the two communities. We sample these three types of edges with equal probability and then uniformly assign the comparison objects.  In total, we sample around $5n^{1.2}$ comparisons. The utility vector $\bm u^*$ is selected uniformly between the range $[-0.5, 0.5]$. We then increase the heterogeneity of the comparison graph by randomly adding hyperedges within Community One. The average coverage probabilities of the 95\% confidence interval for objects within Community One are reported in Figure \ref{fig:syn}. 

Figure~\ref{fig:syn} shows that when no additional data is added, the coverage probabilities derived from both the full MLE and QMLE are close to 0.95. However, as more data is added and the graph becomes increasingly heterogeneous, the coverage probabilities decrease significantly. This phenomenon aligns with our findings: severe heterogeneity undermines the normality of likelihood-based estimators.

 \begin{figure}[htbp]
  \centering 
   \subfigure[]{\includegraphics[width=0.45\linewidth, trim={0.3cm 0.5cm 2cm 0.5cm}, clip]{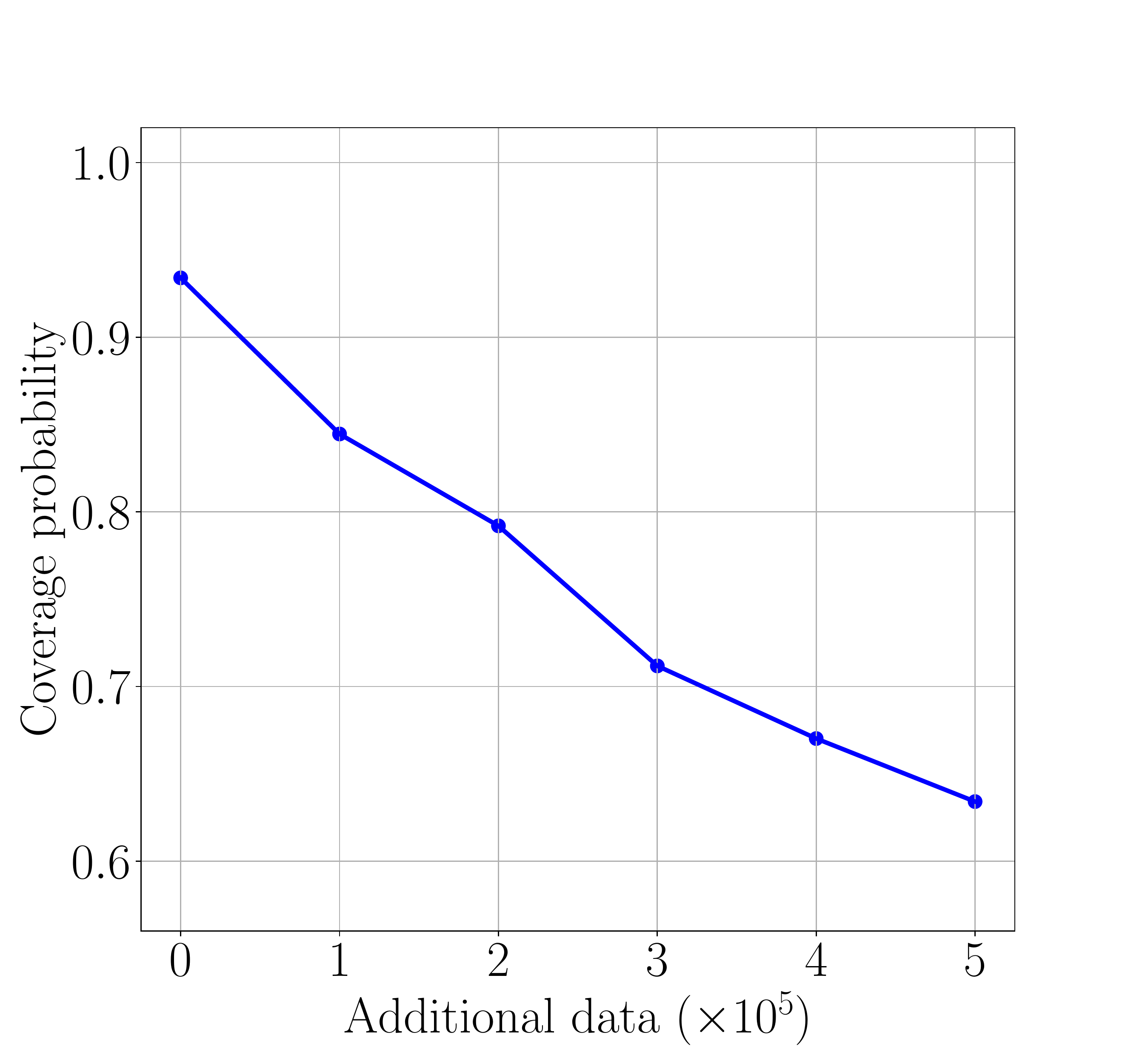}}\hfill
   \subfigure[]{\includegraphics[width=0.45\linewidth, trim={0.3cm 0.5cm 2cm 0.5cm}, clip]{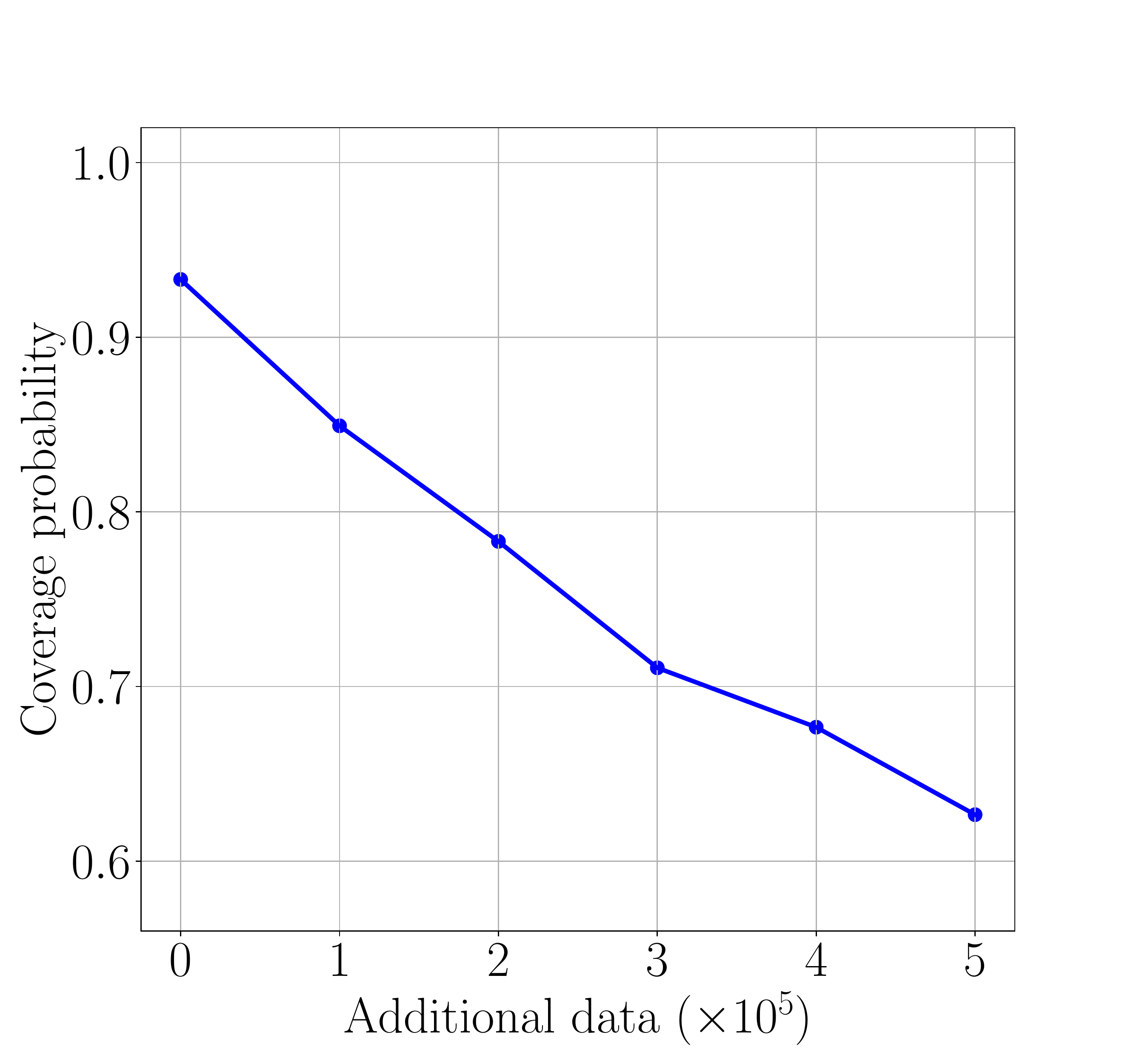}}\hfill
    \caption{Coverage probabilities when adding additional data to increase the heterogeneity of an HSBM. (a): Full MLE; (b): QMLE.}
	\label{fig:syn}
\end{figure}

\subsection{Horse-racing Data}\label{horse}
	We apply the proposed likelihood-based estimators to the Hong Kong horse-racing data\footnote{\href{https://www.kaggle.com/datasets/gdaley/hkracing}{https://www.kaggle.com/datasets/gdaley/hkracing}}. This dataset contains horse-racing competitions from the year 1999 to the year 2005. We cleaned the dataset by removing horses that either participated in too few competitions (fewer than 10) or had won/lost in all the competitions they participated in. After preprocessing, there are 2,814 horses and 6,328 races. The sizes of comparison in this dataset are nonuniform and take values between 4 and 14, with the most common in 12 or 14, making estimating the standard deviation in the full MLE computationally prohibitive. In contrast, the QMLE requires much less computational cost and achieves a comparable efficiency with the full MLE. For this reason, we only apply the QMLE to this dataset. The results are reported in Table \ref{horse_data}. We present the information on the top-10 horses based on the estimation from the QMLE, including the number of races, the average place in the race, the estimated utility score, and the estimated $95\%$ confidence interval. As can be seen in Table~\ref{horse_data}, the identified horses by the QMLE performed consistently well with their historical record. 	
\begin{table}[t]
	\centering
	\begin{tabular}{cccccc}
		Horse id & Race & Average place & Estimation & 95\% confidence interval        & Rank \\ \hline
		1033     & 21   & 1.810      & 5.656      & (4.808, 6.504) & 1   \\
		564      & 11   & 1.364      & 5.424      & (3.324, 7.524) & 2   \\
		2402     & 14   & 1.857      & 5.291      & (4.328, 6.254) & 3   \\
		1588     & 12   & 2.500      & 4.968      & (4.053, 5.883) & 4   \\
		2248     & 24   & 2.875      & 4.543      & (3.917, 5.170) & 5   \\
		2558     & 28   & 2.464      & 4.540      & (3.945, 5.136) & 6   \\
		160      & 11   & 3.091      & 4.452      & (3.562, 5.342) & 7   \\
		218      & 16   & 2.375      & 4.387      & (3.616, 5.159) & 8   \\
		1044     & 19   & 2.684      & 3.877      & (3.182, 4.572) & 9   \\
		2577     & 15   & 3.160      & 3.865      & (3.255, 4.475) & 10   
	\end{tabular}
	\caption{ Horse-racing data: the QMLE for the top ten horses.  }\label{horse_data}
\end{table}
}

\section{Summary}
In this paper, we studied several likelihood-based estimators for utility vector estimation in the PL model. We showed that these estimators can be interpreted from a unified perspective through their estimating equations. Based on this, we established both uniform consistency and asymptotic normality of the estimators under appropriate conditions characterized by the underlying comparison graph sequence; we also discussed the trade-off between statistical efficiency and computational complexity for practical uncertainty quantification. For several common random sampling scenarios, such as NURHM and HSBM, the proposed conditions hold optimally in terms of the leading-order graph sparsity. To the best of our knowledge, the asymptotic results in this paper are amongst the first in the PL model that apply to nonuniform edge sizes and heterogeneous comparison probabilities while covering different likelihood-based estimators.

There are several directions worth further study. First, although the proposed conditions are near-optimal in terms of graph sparsity for homogeneous random graphs models, it remains unclear whether they are also optimal in terms of graph heterogeneity, which is empirically suggested to be appropriately controlled in order for the asymptotic normality result to hold (see Section~\ref{sec: an_numerical}). In addition, our uniform consistency results provide a simple condition purely characterized by the underlying comparison graph sequence. It would be interesting to investigate if such conditions hold for other random graph models with prescribed degree structures. Furthermore, since we have provided the asymptotic normality of the likelihood-based estimator of the PL model, a natural step forward is to consider constructing efficient hypothesis testing procedures on the parameter space.

\section*{Acknowledgment}

We express our sincere gratitude to the Editors, the Associate Editors, and the anonymous referees for their constructive comments, which significantly improved the manuscript's presentation and led to the inclusion of the discussion on asymptotic normality for deterministic graph sequences in Section~\ref{sec:6+}. R. Han was supported by the Hong Kong Research Grants Council (No. 14301821) and the Hong Kong Polytechnic University (P0044617, P0045351). Y. Xu was supported by start-up funding from the University of Kentucky and the AMS-Simons Travel Grant (No. 3048116562).


\section*{Appendices}
In the appendices, we provide rigorous justifications for the technical results in the manuscript that are stated without proof. We list the notations frequently used in the subsequent analysis in Section \ref{sec:s1} for the reader's convenience. The proofs of Propositions \plainref{lm:1}-\plainref{lm:3} are provided in Section \ref{sec:s2}. Section \ref{sec:8} contains the proofs of Theorems~\plainref{main:general}-\plainref{thm:and}. The proofs of the lemmas related to uniform consistency and asymptotic normality are included in Section~\ref{sec:s4} and Section~\ref{sec:s5}, respectively. 

\appendix

\section{Notation List}\label{sec:s1}
In this part, we summarize the important notation used throughout both the main article and the appendices.

\begin{itemize}
		\item $[n]: =\{1,\ldots,n\}$, where $ n\in\mathbb N $ is the number of objects.
	\item $ \bu = (u_1,\ldots,u_n)^\top$: true score vector of the model, where $ u_k$ is the latent score of object $k$ for $ k \in [n]$. $\mathbb E_{\bm u}$ stands for the expectation operator in a PL model with parameter $\bm u$. 
	\item $N$: the number of comparison data; $N_k$: the number of comparisons involving object $k$; $N_{kk'}$: the number of the comparisons involving both objects $k$ and $k'$.
	\item $N_{n,-}:=\min_{k \in [n]} N_k$; $N_{n,+}:=\max_{k \in [n]} N_k$.
	\item $T_i, i\in[N]$: the set of objects participating in the $i$th comparison.
	\item $m_i$: the number of objects participating in the $i$th comparison, namely $m_i = |T_i|$. 
		The upper bound on $m_i$ is $M$. $y_i$ denotes the number of objects observed from the top.
	\item $\pi_i(t)$: the object with rank $t$ in $T_i$; $r_i(k)$: the rank of the object $k$ in $T_i$. 
	\item $l_1(\bu)$: the log-likelihood function for the (marginal) MLE; $l_2(\bu)$: the log-likelihood function for the QMLE.
	\item $p_n^{(\ms)}$: the lower bound on the probability of an $\ms$-size edge in NURHM; $q_n^{(\ms)}$: the upper bound on the probability of an $\ms$-size edge in NURHM.
	\item In NURHM, we define $\xi_{n,-}:=\sum_{\ms=2}^M n^{\ms-1}p_n^{(\ms)}$ and $ \xi_{n,+}:=\sum_{\ms=2}^M n^{\ms-1}q_n^{(\ms)}$. These two numbers represent the minimum and maximum order of the expected number of edges in NURHM.
	\item $\omega_{n, i}$, $i\in[K]$: the probability of an edge in the $i$th community in HSBM.
	\item In HSBM, we define $\zeta_{n,-}=n^{\mm-1}\min_{0\leq i \leq K}\omega_{n, i}$ and $\zeta_{n,+}:=n^{\mm-1}\max_{0\leq i \leq K}\omega_{n, i}.$ These two numbers represent the minimum and maximum order of the expected number of edges in HSBM.
	\item Hypergraph  $\HH(V, E)$, where $V$ is the vertex set and $E\subseteq \mathscr P(V)$ is the edge set. 
	\item For $U_1, U_2\subseteq V$ with $U_1\cap U_2=\emptyset$, we define the edges between $U_1$ and $U_2$ as $\mathcal E(U_1, U_2)=\{e\in E: e\cap U_i\neq\emptyset, i = 1, 2\}$. The set of boundary edges of $U$ is defined as $\partial U = \mathcal E(U, U^\complement)$. 
	\item Given a hypergraph $\HH(V, E)$ with $V = [n]$, for $U\subset [n]$, let $h_\HH(U) ={|\partial U|}/{\min\{|U|, |U^\complement|\}}.$ 	The modified Cheeger constant of $\HH$ is defined as $h_\HH = \min_{U \subset [n]}h_\HH(U)$. 
		\item ${\cH}(\bm u)$: the Hessian matrix of log-likelihood function at $\bm u$. $\cH^*(\bu^*):$ the expectation of ${\cH}(\bm u)$, namely, $\cH^*(\bu^*)=\mathbb{E}[\cH({\bu^*})]$. (With some abuse of notation, we do not distinguish between $l_1$ and $l_2$ since their analyses are analogous under our approach.) 

	\item Since $-\cH^*(\bm u^*)$ is an unnormalized weighted graph Laplacian, we further define its degree matrix and weight matrix to be $\D$ and $\W$ respectively, and normalized graph Laplacian to be $ {\mathcal{L}_{\sym}} =  (I -\A ) $ with $ \A = \D^{-1/2}\W\D^{-1/2}$. 

\item $\lambda_1(B)\leq\cdots\leq \lambda_n(B)$: the eigenvalue of the matrix $B$ in the increasing order (assuming all real).
\end{itemize}

\section{Interpretation of Estimating Equations}\label{sec:s2} 
In this section, we provide the proofs of Propositions \plainref{lm:1}-\plainref{lm:3}.
\subsection{Proof of Proposition \plainref{lm:1}}\label{app:11}
By the first-order optimality condition, $\partial_k l_1(\buu) = 0$ for $k\in [n]$.
This can be further expanded as
\begin{align*}
\partial_kl_1(\buu) &= \sum_{i: k\in T_i}\left(\mathbf 1_{\{r_i(k)\leq y_i\}}- \sum_{j=1}^{r_i(k)\wedge y_i}\frac{\e{\widehat{u}_k}}{\sum_{t = j}^{m_i}\e{\widehat{u}_{\pi_i(t)}}}\right) = 0,
\end{align*}
where $r_i(k)$ is the rank of $k$ in $\pi_i$. 
Dividing both sides by $N_k$ and rewriting, 
\begin{align*}
\frac{1}{N_k}\sum_{i: k\in T_i}\left(\mathbf 1_{\{r_i(k)\leq y_i\}} -\sum_{j\in [m_i]}\frac{\mathbf 1_{\{j\leq r_i(k)\wedge y_i\}}\e{\widehat{u}_k}}{\sum_{t = j}^{m_i}\e{\widehat{u}_{\pi_i(t)}}}\right) =  0.
\end{align*}
To understand the second term in the parenthesis, note that assuming $\pi_i$ (or $r_i$) is observed in a PL model with parameters $\bu$, 
\begin{align*}
\sum_{j\in [m_i]}\frac{\mathbf 1_{\{j\leq r_i(k)\wedge y_i\}}\e{u_k}}{\sum_{t = j}^{m_i}\e{u_{\pi_i(t)}}} = \sum_{j\in [m_i]}\mathbb E_{\bu}[\mathbf 1_{\{\text{$r_i(k) = j$ and $r_i(k)\leq y_i$}\}}\mid\mathscr F_{i, j-1}].
\end{align*}
The first statement in the lemma follows by setting $\bu = \buu$, while the second statement follows by setting $\bu = \bm u^*$ and taking expectation with respect to $\pi_i$:
\begin{align*}
&\mathbb E\left[\sum_{j\in [m_i]}\mathbb E_{\bm u^*}[\mathbf 1_{\{\text{$r_i(k) = j$ and $r_i(k)\leq y_i$}\}}\mid\mathscr F_{i, j-1}] - \mathbf 1_{\{r_i(k)\leq y_i\}}\right]\\
=&\  \mathbb E\left[\sum_{j\in [m_i]}\mathbf 1_{\{\text{$r_i(k) = j$ and $r_i(k)\leq y_i$}\}}-\mathbf 1_{\{r_i(k)\leq y_i\}}\right]=  \ 0.
\end{align*}
Since $\pi_i$'s are sampled from the PL model with the true utility vector $\bm u^*$, the first equality holds as a result of the tower property. 

\subsection{Proof of Proposition \plainref{lm:2}}\label{app:12}
By definition, 
\begin{align*}
\argmin_{\bm v: \langle1 \rangle ^\top\bm v = 0}\mathbb E_{\bm\pi\sim f}[\mathsf{KL}(f(\bm\pi; \bm u)\| g(\bm\pi; \bm v))] &= \argmin_{\bm v: \langle1 \rangle ^\top\bm v = 0}\mathbb E_{\bm\pi\sim f}\left[\log\frac{f(\bm\pi; \bu)}{g(\bm\pi; \bm v)}\right]\\
& = \argmax_{\bm v: \langle1 \rangle ^\top\bm v = 0}\mathbb E_{\bm\pi\sim f}\left[\log g(\bm\pi; \bm v)\right].
\end{align*}
Denoting $G(\bm v) = \mathbb E_{\bm\pi\sim f}\left[\log g(\bm\pi; \bm v)\right]$ and expanding $G(\bm v)$,  
\begin{align*}
G(\bm v)=&\sum_{T\in E}\sum_{\pi_T\in \mathcal S(T)}f_T(\pi_T; \bm u)\times \\
&~ \left[\sum_{k\in T}
\bigg\{(|T|-r_{\pi_T}(k))v_k-\sum_{k'\in T: r_{\pi_T}(k')>r_{\pi_T}(k)}\log(\e{v_k} + \e{v_{k'}}) \bigg\}\right],
\end{align*}
where $r_{\pi_T}(k)$ is the rank of $k$ in $\pi_T$. 
To show $\bm u$ is the unique maximizer of $G(\bm v)$, it suffices to verify that $\bm u$ satisfies the first-order optimality condition and $G(\bm v)$ is strictly concave. 
The former can be verified by a direct calculation as follows. 
\begin{align}
\partial_k G(\bm v) &= \sum_{T\in E}\sum_{\pi_T\in \mathcal S(T)}f_T(\pi_T; \bm u)\left(|T|-r_{\pi_T}(k)-\sum_{k'\in T: k'\neq k}\frac{\e{v_k}}{\e{v_k} + \e{v_{k'}}}\right)\nonumber\\
& = \sum_{T\in E}\left(|T|-\sum_{k'\in T: k'\neq k}\frac{\e{v_k}}{\e{v_k} + \e{v_{k'}}}-\mathbb E_{\pi_T\sim f_T(\cdot; \bu)}[r_{\pi_T}(k)]\right)\nonumber\\
& = \sum_{T\in E}\left[\mathbb E_{\pi_T\sim f_T(\cdot; \bm v)} \left(|T|-\sum_{k'\in T: k'\neq k}\mathbf 1_{\{r_T(k)<r_T(k')\}}\right)-\mathbb E_{\pi_T\sim f_T(\cdot; \bu)}[r_{\pi_T}(k)]\right]\nonumber\\
&=\sum_{T\in E}\left(\mathbb E_{\pi_T\sim f_T(\cdot; \bv)}[r_{\pi_T}(k)] - \mathbb E_{\pi_T\sim f_T(\cdot; \bu)}[r_{\pi_T}(k)]\right),\label{guaguagua}
\end{align}
where the third equality uses the internal consistency of the PL model. Taking $\bm v = \bm u$ yields $\partial_k G(\bm u) = 0$ for all $k$. 

To prove the strict concavity, we compute the Hessian of $G(\bm v)$:
\begin{align*}
\partial_{kk'} G(\bm v)
= 
\begin{cases}
\sum_{T\in E: k\in T}\sum_{j\in T: j\neq k}\frac{-\e{v_k + v_{j}}}{(\e{v_k} + \e{v_{j}})^2}& k' = k\\
\\
\sum_{T\in E: \{k, k'\}\subseteq T}\frac{\e{v_k + v_{k'}}}{(\e{v_k} + \e{v_{k'}})^2}& k' \neq k
\end{cases}
.
\end{align*}
This expression shows that $-\nabla^2 G(\bm v)$ is the unnormalized graph Laplacian of the weighted graph on $[n]$ with weights $\{\partial_{kk'} G(\bm v)\}$ for $k\neq k'$.
As a result, $-\nabla^2 G(\bm v)$ is nonnegative definite.  
Meanwhile, a moment's thought reveals that the connectivity of $\HH$ implies that the weighted graph associated with $-\nabla^2 G(\bm v)$ is connected. Hence, the eigenspace of $\nabla^2 G(\bm v)$ associated with eigenvalue zero is one-dimensional and spanned by the all-ones vector $\langle 1\rangle$. Since the constraint requires $\langle 1\rangle^\top\bm v = 0$, $G(\bm v)$ is strictly concave on the feasible set of $\bm v$. This verifies the uniqueness of the minimizer.  

\subsection{Proof of Proposition \plainref{lm:3}}

By the first-order optimality, 
\begin{align*}
\partial_kl_2(\widetilde{\bu}) &= \sum_{i: k\in T_i}\left(|T_i|-r_{i}(k)-\sum_{k'\in T_i: k'\neq k}\frac{\e{\widetilde{u}_k}}{\e{\widetilde{u}_k} + \e{\widetilde{u}_{k'}}}\right)\\
& = -\sum_{i: k\in T_i}\left[r_{i}(k)-\left(|T_i|-\sum_{k'\in T_i: k'\neq k}\frac{\e{\widetilde{u}_k}}{\e{\widetilde{u}_k} + \e{\widetilde{u}_{k'}}}\right)\right]\\
& = -\sum_{i: k\in T_i}(r_{i}(k)-\mathbb E_{\widetilde{\bu}}[r_{i}(k)])\\
& = 0,
\end{align*}
where the penultimate step follows from a same computation as \eqref{guaguagua}. The proof is finished by dividing both sides by $-N_k$.

\section{Proofs of Theorems \ref{main:general}-\ref{thm:and}}\label{sec:8}
In this section, we provide the proofs of Theorems \plainref{main:general}-\plainref{thm:and}.

\subsection{Proof of Theorem \plainref{main:general}}\label{app:main}

\subsubsection{Preparation}

To establish both unique existence and uniform consistency, we introduce some additional notations. For the former, we define
\begin{align}
\FI := \left\{\text{$\forall U\subset [n]$, $\exists (k_1, k_2)\in U\otimes U^\complement$ and $T_i\in \E$, such that $k_1, k_2\in T_i$ and $k_1\succ k_2$}\right\}.\label{goodin}
\end{align}
The event $\FI$ ensures that the negative log-likelihood or the negative quasi-likelihood functions are coercive, which is sufficient and necessary for the unique existence of the marginal MLE and QMLE \citep[Lemma 1]{MR2051012}. 

Moreover, for both the marginal MLE and QMLE, it follows from the direct computation that, for $k\in [n]$ and $\bm u\in\mathbb R^n$, 
\begin{align}
\partial_{k}l_1(\bu) = \sum_{i: k\in T_i}\psi(k; T_i, \pi_i, \bu)\quad\quad \partial_{k}l_2(\bu) = \sum_{i: k\in T_i}\varphi(k; T_i, \pi_i, \bu)\label{eq:1++},
\end{align}
where 
\begin{align}
\psi(k; T_i, \pi_i, \bm u) &= \mathbf 1_{\{r_i(k)\leq y_i\}} - \sum_{j\in [r_i(k)\wedge y_i]}\frac{\e{u_k}}{\sum_{t=j}^{m_i} \e{u_{\pi_i(t)}}}\label{yourpsi},\\
\varphi(k; T_i, \pi_i, \bm u) &= |T_i|-r_i(k)-\sum_{k'\in T_i, k'\neq k}\frac{\e{u_k}}{\e{u_k} + \e{u_{k'}}}\label{yourphi},
\end{align}
with $\pi_i$ and $r_i(k)$ denoting the permutation associated with $T_i$ and the rank of $k$ in $\pi_i$, respectively. By Propositions~\plainref{lm:1} and \plainref{lm:3}, $\mathbb E[\partial_{k}l_1(\bu^*)]=\mathbb E[\partial_{k}l_2(\bu^*)] = 0$. 
A similar event to \eqref{goodin} concerning $\partial_k l_1$ and $\partial_kl_2$ we will need is the following:
{\small
\begin{align}
\FII:= \left\{\text{$\bigg|\sum_{k\in \US}\partial_{k}l_{s}(\bu^*)\bigg| \leq {4}M^2\sqrt{|\US||\partial \US|\log n}, \text{for any } \US\subset [n]$ with $|\US|\leq n/2$, and $s= 1, 2$}\right\}. \label{goodin1}
\end{align}}

The following lemma states that under Assumptions~\plainref{ap:2}-\plainref{ap:1} and \plainref{ap:11}, both $\FI$ and $\FII$ hold with high probability. 

\begin{Lemma}\label{yadongwang}
Under Assumptions~\plainref{ap:2}-\plainref{ap:1} and \plainref{ap:11}, for all sufficiently large $n$, 
\begin{align*}
\mathbb P(\FI\cap\FII)\geq 1-n^{-3},
\end{align*}
where $\FI$ and $\FII$ are defined in \eqref{goodin} and \eqref{goodin1}, respectively. 
\end{Lemma}

The proof of Lemma~\ref{yadongwang} is deferred to the end of the section. In the subsequent sections, we will condition on the event $\FI\cap\FII$.

\subsubsection{Chaining}
Since $\FI$ ensures that both the marginal MLE and QMLE exist, to finish the proof, it remains to establish the uniform consistency of $\bm w$. 

Since $\langle 1\rangle^\top\bm u^* = \langle 1\rangle^\top\bm w = 0$, where $\langle 1\rangle$ is the all-ones vector of compatible size, $\langle 1\rangle^\top(\bm w-\bm u^*) = 0$. 
Thus, the largest and smallest components of $\bm w-\bm u^*$ have opposite signs, whence $\|\bm w-\bu^*\|_\infty\leq \max_{k\in [n]}(w_k-u_k^*)-\min_{k\in [n]}(w_k-u_k^*)$.
To prove the theorem, it suffices to show that conditional on $\FI$ and $\FII$, 
\begin{align*}
\max_{k\in [n]}(w_k-u_k^*)-\min_{k\in [n]}(w_k-u_k^*) \lesssim \Gamma_n^{\RE}.
\end{align*} 
Specifically, define $\alpha\in \arg\max_{k\in [n]}(w_k  - u_k^*)$ and $\beta \in \arg\min_{k\in [n]}(w_k  - u_k^*)$. We will show that for all sufficiently large $n$,  $(w_\alpha-u_\alpha^*)-(w_\beta-u_\beta^*)\lesssim\Gamma_n^{\RE}$. To this end, we first construct an increasing sequence of neighbors based on RE-induced estimation errors to chain them together. Then, we show that a chain constructed in such a way is admissible in the sense of Definition \plainref{def:ad}. Under Assumption~\plainref{ap:11}, the desired bound follows.  

To fill in the details, let $c>0$ be  an absolute constant, and $ \{\Delta b_z\}_{z=0}^{\infty} $, $ \{\Delta d_z\}_{z=0}^{\infty}  $ be two increasing sequences that are specified shortly. 
Consider the two sequences of neighbors started at $\alpha$ and $\beta$ respectively defined recursively as follows. 
Let $\Delta b_0 = \Delta d_0 = 0$, and for $z\geq 1$,  
\begin{equation}\label{olkui}
\begin{aligned}
    &B_z = \left\{j: (w_\alpha - u_\alpha^*) - (w_j - u_j^*) \leq \sum_{t=0}^{z-1} \Delta b_t \right\}, & \Delta b_z = c \sqrt{\frac{\log n}{h_{\HH_n}(B_z)}} \\
    &D_z = \left\{j: (w_\beta - u_\beta^*) - (w_j - u_j^*) \geq -\sum_{t=0}^{z-1} \Delta d_t \right\}, & \Delta d_z = c \sqrt{\frac{\log n}{h_{\HH_n}(D_z)}}
\end{aligned}
\end{equation}
where $h_{\HH_n}$ is defined in Definition \plainref{def:ch}. 
Let $Z_{n,1}$ and $Z_{n,2}$ be the stopping times
\begin{align*}
	&Z_{n,1} = \min\left\{z: |B_z|>\frac{n}{2}\right\}& Z_{n,2} = \min\left\{z: |D_z|>\frac{n}{2}\right\}.
\end{align*} 
Note that $\Delta b_z$ and $\Delta d_z$ are defined similarly as the summand of $\Gamma^{\RE}_n$ in Definition~\plainref{def:RE} (up to a multiplicative constant) but for random sequences $\{B_z\}_{z\in [Z_{n,1}] }$ and $\{D_z\}_{Z\in [Z_{n,2}]}$.  
We make the following claim:

\begin{claim}\label{claim:1}
Under Assumptions~\plainref{ap:2}-\plainref{ap:1}, \plainref{ap:11} and conditional on $\FI \cap \FII$, there exists an absolute constant $c>0$ such that for all sufficiently large $n$, $\{B_z\}_{z\in [Z_{n,1}] }$, $\{D_z\}_{z\in [Z_{n,2}]}$ are admissible sequences. 
\end{claim}

By the construction of $Z_{n,1}$ and $Z_{n,2}$,  $B_{Z_{n,1}}\cap D_{Z_{n,2}}\neq \emptyset$, so that $\alpha$ and $\beta$ can be chained together using $\{B_z\}_{z\in [Z_{n,1}]}\cup \{D_z\}_{z\in [Z_{n,2}]}$. As a result, 
{\begin{align*}
(w_\alpha-u_\alpha^*)-(w_\beta-u_\beta^*)\leq c\sum_{z\in [Z_{n,1}-1]}\Delta b_z + c\sum_{z\in [Z_{n,2}-1]}\Delta d_z\leq 2c\Gamma_n^\RE,
\end{align*}
where the last step follows from the definition of $\Gamma_n^\RE$. }
Since $\Gamma^{\RE}_n\to 0$ under Assumption~\plainref{ap:11}, the proof is complete. Therefore, it remains to verify the claim.

The proof sketch described above is similar to \cite[Theorem 4]{han2022general}; however, verification of the claim requires a notable amount of technical details which are laid out in Section~\ref{chedashuai}. For ease of illustration, we only prove for $\{B_z\}_{z\in [Z_{n,1}]}$ as the case for $\{D_z\}_{z\in [Z_{n,2}]}$ can be treated similarly. 

\subsubsection{Admissible Sequence Verification}

We first prove the claim for the marginal MLE and then show that a similar argument works for the QMLE. 

\smallskip

\textbf{\underline{Case I}: The marginal MLE}\label{chedashuai}

\smallskip

We begin with the marginal MLE. 
Recall that the marginal log-likelihood function based on the choice-$y_i$ observations in the $i$th data is given by 
\begin{align*}
&l_1(\bm u) = \sum_{i\in [N]} \left\{ \sum_{j\in [y_i]} u_{\pi_i(j)} - \log\left( \sum_{t=j}^{m_i}\e{u_{\pi_i(t)}}\right)\right\}& |T_i| = m_i.
\end{align*}
Conditional on $\FI$, $\buu$ uniquely exists, so we can apply the first-order optimality condition and \eqref{eq:1++} to obtain 
\begin{align}
&\partial_{k}l_1(\buu) = \sum_{i: k\in T_i}\psi(k; T_i, \pi_i, \buu) = 0& k\in [n],\label{eq:1}
\end{align}
where $\psi$'s are defined in \eqref{yourpsi}. 

Fixing $T_i$, $\pi_i$, and $\bm u$, a crucial property of $\psi(k; T_i, \pi_i, \bm u)$ is that
\begin{align}
\sum_{k\in T_i}\psi(k; T_i, \pi_i, \bm u) &= y_i - \sum_{k\in T_i}\sum_{j\in [r_i(k)\wedge y_i]}\frac{\e{u_k}}{\sum_{t=j}^{m_i} \e{u_{\pi_i(t)}}}\label{sm}\\
&= y_i - \sum_{k\in T_i}\sum_{j\in [r_i(k)\wedge y_i]}\frac{\e{u_{\pi_i(r_i(k))}}}{\sum_{t=j}^{m_i} \e{u_{\pi_i(t)}}}\nonumber\\
& = y_i - \sum_{s\in [y_i]}\sum_{k: r_i(k)\geq s}\frac{\e{u_{\pi_i(k)}}}{\sum_{t=s}^{m_i}\e{u_{\pi_i(t)}}}\nonumber\\
& = 0,\nonumber
\end{align}
where the third identity follows from changing the order of summation. 
By Proposition \plainref{lm:1}, the true parameter $\bm u^*$ satisfies 
\begin{align}
&\mathbb E[\psi(k; T_i, \pi_i, \bm u^*)] = 0\label{eq:2}
\end{align}
for every $k, T_i$.
Combining \eqref{eq:1}-\eqref{eq:2} yields 
\begin{align}
\sum_{k\in B_z}\partial_{k}l_1(\bu^*)-\sum_{k\in B_z}\partial_{k}l_1(\buu)&\stackrel{\eqref{eq:1}}{=} \sum_{k\in B_z }\partial_{k}l_1(\bu^*)\nonumber\\
& = \sum_{k\in B_z}\sum_{i: k\in T_i}\psi(k; T_i, \pi_i, \bm u^*)\nonumber\\
&\stackrel{\eqref{sm}}{=}\sum_{i: T_i\in\partial B_z}\sum_{k\in B_z\cap T_i}\psi(k; T_i, \pi_i, \bm u^*)\nonumber\\
&\stackrel{\eqref{eq:2}}{\leq} {4}M^2\sqrt{|B_z||\partial B_z|\log n},\label{upper}
\end{align}
where the last inequality holds on $\FII$ since $|B_z|\leq n/2$ for $z< Z_{n,1}$. 

To obtain the desired claim, we now derive a lower bound on the left-hand side of \eqref{upper}. This step requires appropriate monotonicity conditions of the score function to hold in the PL model, which is more complicated than the setting in the BT model due to the additional dependence on $\pi_i$'s. We begin by refining the calculations in \eqref{upper}:
\begin{align}
&\sum_{k\in B_z}\partial_{k}l_1(\bu^*)-\sum_{k\in B_z}\partial_{k}l_1(\buu)\label{guagua}\\
\stackrel{\eqref{sm}}{=}&\ \sum_{i: T_i\in\partial B_z}\sum_{k\in B_z\cap T_i}\sum_{j\in [r_i(k)\wedge y_i]}\left(\frac{\e{\widehat{u}_k}}{\sum_{t=j}^{m_i} \e{\widehat{u}_{\pi_i(t)}}}-\frac{\e{u_k^*}}{\sum_{t=j}^{m_i} \e{u_{\pi_i(t)}^*}}\right)\nonumber\\
=&\  \sum_{i: T_i\in\partial B_z}\sum_{j = 1}^{\max_{k\in B_z\cap T_i}( r_i(k)\wedge y_i)}\left(\frac{\sum_{  t\geq j: \pi_i(t)\in B_z\cap T_i}\e{\widehat{u}_{\pi_i(t)}}}{\sum_{t=j}^{m_i} \e{\widehat{u}_{\pi_i(t)}}}-\frac{\sum_{  t\geq j: \pi_i(t)\in B_z\cap T_i}\e{u^*_{\pi_i(t)}}}{\sum_{t=j}^{m_i} \e{u^*_{\pi_i(t)}}}\right).\nonumber
\end{align}

Note that each summand in the last equation is nonnegative. To see this, we take the quotient of the first and second terms in the parentheses to obtain
\begin{align}
&\frac{\sum_{t\geq j: \pi_i(t)\in B_z\cap T_i}\e{\widehat{u}_{\pi_i(t)}}}{\sum_{t=j}^{m_i} \e{\widehat{u}_{\pi_i(t)}}}\bigg/\frac{\sum_{t\geq j, \pi_i(t)\in B_z\cap T_i}\e{u^*_{\pi_i(t)}}}{\sum_{t=j}^{m_i} \e{u^*_{\pi_i(t)}}}\label{zhaohan}\\
=&\ \frac{\sum_{  t\geq j: \pi_i(t)\in B_z\cap T_i}\e{\widehat{u}_{\pi_i(t)}}}{\sum_{ t\geq j: \pi_i(t)\in B_z\cap T_i}\e{u^*_{\pi_i(t)}}}\times\frac{\sum_{  t\geq j: \pi_i(t)\in B_z\cap T_i} \e{u^*_{\pi_i(t)}} + \sum_{  t\geq j: \pi_i(t)\in B^\complement_z\cap T_i} \e{u^*_{\pi_i(t)}}}{\sum_{  t\geq j: \pi_i(t)\in B_z\cap T_i} \e{\widehat{u}_{\pi_i(t)}} + \sum_{  t\geq j: \pi_i(t)\in B^\complement_z\cap T_i} \e{\widehat{u}_{\pi_i(t)}}}\nonumber\\
\geq&\ \frac{\sum_{  t\geq j: \pi_i(t)\in B_z\cap T_i}\e{\widehat{u}_{\pi_i(t)}}}{\sum_{ t\geq j: \pi_i(t)\in B_z\cap T_i}\e{u^*_{\pi_i(t)}}}\times\min\left\{\frac{\sum_{  t\geq j: \pi_i(t)\in B_z\cap T_i} \e{u^*_{\pi_i(t)}}}{\sum_{  t\geq j: \pi_i(t)\in B_z\cap T_i} \e{\widehat{u}_{\pi_i(t)}}  }, \frac{\sum_{  t\geq j: \pi_i(t)\in B^\complement_z\cap T_i} \e{u^*_{\pi_i(t)}}}{\sum_{  t\geq j: \pi_i(t)\in B^\complement_z\cap T_i} \e{\widehat{u}_{\pi_i(t)}}}\right\}\nonumber\\
\geq&\ \frac{\sum_{  t\geq j: \pi_i(t)\in B_z\cap T_i}\e{\widehat{u}_{\pi_i(t)}}}{\sum_{ t\geq j: \pi_i(t)\in B_z\cap T_i}\e{u^*_{\pi_i(t)}}}\times\frac{\sum_{  t\geq j: \pi_i(t)\in B_z\cap T_i} \e{u^*_{\pi_i(t)}}}{\sum_{  t\geq j: \pi_i(t)\in B_z\cap T_i} \e{\widehat{u}_{\pi_i(t)}}  } = 1\nonumber.
\end{align}

To obtain the last inequality, observe for $t, t'\geq j$ and $\pi_i(t)\in B_z\cap T_i$, $\pi_i(t')\in B^\complement_z\cap T_i$, $\widehat{u}_{\pi_i(t)}-u^*_{\pi_i(t)}\geq \widehat{u}_{\pi_i(t')}-u^*_{\pi_i(t')}$ by the definition of $B_z$. As a result,  
\begin{align*}
&\frac{\sum_{  t\geq j: \pi_i(t)\in B_z\cap T_i} \e{u^*_{\pi_i(t)}}}{\sum_{  t\geq j: \pi_i(t)\in B_z\cap T_i} \e{\widehat{u}_{\pi_i(t)}}}\leq\max_{t\geq j: \pi_i(t)\in B_z\cap T_i}\e{u^*_{\pi_i(t)}-\widehat{u}_{\pi_i(t)}}\\
&\leq \min_{t\geq j: \pi_i(t)\in B^\complement_z\cap T_i}\e{u^*_{\pi_i(t)}-\widehat{u}_{\pi_i(t)}}\leq \frac{\sum_{  t\geq j: \pi_i(t)\in B^\complement_z\cap T_i} \e{u^*_{\pi_i(t)}}}{\sum_{  t\geq j: \pi_i(t)\in B^\complement_z\cap T_i} \e{\widehat{u}_{\pi_i(t)}}}.
\end{align*}
Consequently, we can obtain a lower bound on \eqref{guagua} by only summing over the edges in $\partial B_z\cap\partial B_{z+1}$ which will have its summand lower bounded by some positive constant. Without loss of generality, we may assume $\partial B_z\cap\partial B_{z+1}\neq\emptyset$, since otherwise $\partial B_z\subseteq B_{z+1}$, which immediately implies the desired admissibility condition. Then, 

{\footnotesize\begin{align}
&\sum_{k\in B_z}\partial_{k}l_1(\bu^*)-\sum_{k\in B_z}\partial_{k}l_1(\buu) \nonumber\\
\geq& \sum_{i: T_i\in\partial B_z\cap\partial B_{z+1}}\sum_{j = 1}^{\max_{k\in B_z\cap T_i} (r_i(k)\wedge y_i)}\Bigg(\frac{\sum_{  t\geq j: \pi_i(t)\in B_z\cap T_i}\e{\widehat{u}_{\pi_i(t)}}}{\sum_{t=j}^{m_i} \e{\widehat{u}_{\pi_i(t)}}}-\frac{\sum_{ t\geq j, \pi_i(t)\in B_z\cap T_i}\e{u^*_{\pi_i(t)}}}{\sum_{t=j}^{m_i} \e{u^*_{\pi_i(t)}}}\Bigg)\nonumber\\
\geq&\ \sum_{i: T_i\in\partial B_z\cap\partial B_{z+1}}\Bigg(\frac{\sum_{  t\geq 1: \pi_i(t)\in B_z\cap T_i}\e{\widehat{u}_{\pi_i(t)}}}{\sum_{t=1}^{m_i} \e{\widehat{u}_{\pi_i(t)}}}-\frac{\sum_{ t\geq 1, \pi_i(t)\in B_z\cap T_i}\e{u^*_{\pi_i(t)}}}{\sum_{t=1}^{m_i} \e{u^*_{\pi_i(t)}}}\Bigg).\label{xiyue}
\end{align}}

By definition, for any $T_i\in\partial B_{z+1}\cap \partial B_z$, $T_i\cap B_{z+1}^\complement  \neq\emptyset$, so that
\begin{align*}
\max_{j\in T_i}\left\{(\widehat{u}_\alpha - u^*_\alpha)-(\widehat{u}_j - u^*_j)\right\}\geq \sum_{t = 0}^z\Delta b_t. 
\end{align*}
This allows us to obtain a refined lower bound on \eqref{zhaohan} based on the following elementary inequality. For $\delta_1, \delta_2, \delta_3, \delta_4>0$, if ${\delta_2}{\delta_3}\geq (1+\epsilon_1) {\delta_1}{\delta_4}$ and $\delta_2\geq \epsilon_2 \delta_1$ for some $0<\epsilon_1, \epsilon_2<1/2$, then
\begin{align}
\frac{\delta_1+\delta_2}{\delta_3+\delta_4}\geq \left(1+\frac{\epsilon_1\epsilon_2}{2}\right)\cdot\frac{\delta_1}{\delta_3}.\label{need?}
\end{align}
This inequality can be verified as  
\begin{align*}
	\frac{\delta_1+\delta_2}{\delta_3+\delta_4}\geq\frac{\delta_1+\epsilon_2\delta_1}{\delta_3+ (\epsilon_2 \delta_1\delta_4/\delta_2)}\geq\frac{ (1+\epsilon_1)(1+ \epsilon_2)  \delta_1}{(1+\epsilon_1)\delta_1+\epsilon_2 \delta_3}\geq \left(1+\frac{\epsilon_1\epsilon_2}{2}\right)\cdot\frac{\delta_1}{\delta_3}. 
\end{align*}
Under Assumption \plainref{ap:2}, $C_1^{-2}\leq \e{u^*_i-u^*_j}\leq C^2_1$ for all $i, j\in [n]$. Since $T_i\cap B_{z+1}^\complement  \neq\emptyset$, 
{\small\begin{align*}
&\frac{\sum_{t\geq 1: \pi_i(t)\in B_z\cap T_i}\e{\widehat{u}_{\pi_i(t)}}}{\sum_{t=1}^{m_i} \e{\widehat{u}_{\pi_i(t)}}}\bigg/\frac{\sum_{t\geq 1, \pi_i(t)\in B_z\cap T_i}\e{u^*_{\pi_i(t)}}}{\sum_{t=1}^{m_i} \e{u^*_{\pi_i(t)}}}\\
=&\ \frac{\sum_{t\geq 1: \pi_i(t)\in B_z\cap T_i}\e{\widehat{u}_{\pi_i(t)}}}{\sum_{t\geq 1, \pi_i(t)\in B_z\cap T_i}\e{u^*_{\pi_i(t)}}}\\
& \times\underbrace{\frac{\sum_{t\geq 1, \pi_i(t)\in B_{z}\cap T_i}\e{u^*_{\pi_i(t)}} + \sum_{t\geq 1, \pi_i(t)\in B_z^\complement\cap B_{z+1}\cap T_i}\e{u^*_{\pi_i(t)}}+\sum_{t\geq 1, \pi_i(t)\in B^\complement_{z+1}\cap T_i}\e{u^*_{\pi_i(t)}}}{\sum_{t\geq 1: \pi_i(t)\in B_{z}\cap T_i}\e{\widehat{u}_{\pi_i(t)}}+ \sum_{t\geq 1: \pi_i(t)\in B^\complement_{z}\cap B_{z+1}\cap T_i}\e{\widehat{u}_{\pi_i(t)}}+\sum_{t\geq 1: \pi_i(t)\in B^\complement_{z+1}\cap T_i}\e{\widehat{u}_{\pi_i(t)}}}}_{(*)}\\
\geq&\ \frac{\sum_{t\geq 1: \pi_i(t)\in B_z\cap T_i}\e{\widehat{u}_{\pi_i(t)}}}{\sum_{t\geq 1, \pi_i(t)\in B_z\cap T_i}\e{u^*_{\pi_i(t)}}}\\
&\times \underbrace{\frac{C^2_1(m_i-1)\sum_{t\geq 1, \pi_i(t)\in B_{z}\cap T_i}\e{u^*_{\pi_i(t)}} + \sum_{t\geq 1, \pi_i(t)\in B^\complement_{z+1}\cap T_i}\e{u^*_{\pi_i(t)}}}{C^2_1(m_i-1)\sum_{t\geq 1: \pi_i(t)\in B_{z}\cap T_i}\e{\widehat{u}_{\pi_i(t)}}+ \sum_{t\geq 1: \pi_i(t)\in B^\complement_{z+1}\cap T_i}\e{\widehat{u}_{\pi_i(t)}}}}_{(**)}.
\end{align*}}
To see why the last step holds, note for $x_1, \ldots, x_6>0$ satisfying $x_1/x_4\leq x_2/x_5\leq x_3/x_6$ and $x_2/x_1\leq \epsilon$, 
\begin{align*}
\frac{x_1+x_2+x_3}{x_4+x_5+x_6} &\geq \frac{x_1 + x_2 + x_3}{x_4 + (x_2x_4/x_1) + x_6} = \frac{x_1(1+x_2/x_1) + x_3}{x_4(1+x_2/x_1) + x_6}\geq \frac{x_1(1+\epsilon)+x_3}{x_4(1+\epsilon) + x_6}.
\end{align*}
Identifying ($*$) in the form of $(x_1+x_2+x_3)/(x_4+x_5+x_6)$ and noting $x_2/x_1\leq C_1^2(m_i-1)$ yields the desired result. Furthermore, identifying ($**$) in the form of $(\delta_1+\delta_2)/(\delta_3+\delta_4)$ in \eqref{need?}, we can check using the definition of $B_z$ that
\begin{align*}
\frac{\delta_2\delta_3}{\delta_1\delta_4}&\geq \frac{\e{\widehat{u}_\alpha - u^*_\alpha - \sum_{j\in [z-1]}\Delta b_j}}{\e{\widehat{u}_\alpha - u^*_\alpha - \sum_{j\in [z]}\Delta b_j}} = \e{\Delta b_z}\geq 1+\Delta b_z,  \ \delta_2 \geq\frac{\delta_1}{C^4_1M^2}. 
\end{align*}
Under Assumption~\plainref{ap:11}, $\Delta b_z\leq 1/2$ for all $z<Z_{n,1}$ and all sufficiently large $n$. 
Applying \eqref{need?}, we obtain for $T_i\in\partial B_{z+1}\cap \partial B_z$,  
\begin{align*}
\eqref{zhaohan}&\geq 1+\frac{\Delta b_z}{2C_1^4M^2}.
\end{align*}
As a result, 
\begin{align*}
\frac{\sum_{t\geq 1: \pi_i(t)\in B_z\cap T_i}\e{\widehat{u}_{\pi_i(t)}}}{\sum_{t=1}^{m_i} \e{\widehat{u}_{\pi_i(t)}}}-\frac{\sum_{t\geq 1, \pi_i(t)\in B_z\cap T_i}\e{u^*_{\pi_i(t)}}}{\sum_{t=1}^{m_i} \e{u^*_{\pi_i(t)}}}\geq\frac{\Delta b_z}{M^3C_1^6},
\end{align*}
which is substituted into \eqref{xiyue} to yield
\begin{align}
\sum_{k\in B_z}\partial_{k}l_1(\bu^*)-\sum_{k\in B_z}\partial_{k}l_1(\buu)\geq\frac{\Delta b_z}{M^3C_1^6}|\partial B_z\cap \partial B_{z+1}|. \label{lower}
\end{align}
Combining \eqref{upper} and \eqref{lower}, 
\begin{align*}
\frac{1}{M^3C_1^6}|\partial B_z\cap \partial B_{z+1}|\Delta b_z\leq {4}M^2\sqrt{|B_z||\partial B_z|\log n}. 
\end{align*}
Taking $c = {8}M^5C_1^6$ in \eqref{olkui} yields
\begin{align*}
|\partial B_z\cap\partial B_{z+1}|\leq \frac{1}{2}|\partial B_z|\Longrightarrow |\{e\in\partial B_z: e\subseteq B_{z+1}\}|\geq\frac{1}{2}|\partial B_z|.
\end{align*}
This finishes the proof of $\{B_z\}_{z\in [Z_{n,1}]}$ being admissible.

\smallskip

\textbf{\underline{Case II}: The QMLE}\label{chedashuai1}

\smallskip

 Although the quasi-likelihood is misspecified (as it ignores the dependence among pairwise comparisons obtained from breaking the same edge), its derivatives provide a set of unbiased estimating equations so it is more appropriate to view QMLE as a moment estimator. Writing down the log-quasi-likelihood, we have
\begin{align*}
	l_2(\bm u) = \sum_{i\in [N]}\sum_{1\leq j<t \leq m_i}\left[u_{\pi_i(j)}-\log (\e{u_{\pi_i(j)}}+ \e{u_{\pi_i(t)}})\right].
\end{align*}
By the first-order optimality condition, 
\begin{align}
&\partial_{k}l_2(\widetilde{\bu}) = \sum_{i: k\in T_i}\varphi(k; T_i, \pi_i, \widetilde{\bu}) = 0& k\in [n],\label{eq:231}
\end{align}
where $\varphi$ is defined in \eqref{yourphi}. 
Analogously, when fixing $T_i$, $\pi_i$, and $\bm u$, we can check that 
\begin{align}
\sum_{k\in T_i}\varphi(k; T_i, \pi_i, \bm u) &= \sum_{k\in T_i}\left(|T_i|-r_i(k)-\sum_{k'\in T_i, k'\neq k}\frac{\e{u_k}}{\e{u_k} + \e{u_{k'}}}\right)\nonumber\\
& = \sum_{k\in T_i}\left(|T_i|-\sum_{k'\in T_i, k'\neq k}\mathbb E[\mathbf 1_{\{r_i(k)<r_i(k')\}}]-r_i(k)\right)\nonumber\\
& =  \sum_{k\in T_i}\left(\mathbb E[r_i(k)]-r_i(k)\right) = 0,\label{youwang}
\end{align}
where we have used the internal consistency of the PL model.
Meanwhile, $\mathbb E[\varphi(k; T_i, \pi_i, \bm u^*)] = 0$ due to the tower property. This observation is critical to obtain appropriate concentration estimates as in the marginal MLE case. 

With the same idea in the proof of the marginal MLE, $\FII$ gives the following upper bound: 
\begin{align}
&\sum_{k\in B_z}\partial_{k}l_2(\bu^*)-\sum_{k\in B_z}\partial_{k}l_2(\widetilde{\bu})\leq {4}M^2\sqrt{|B_z||\partial B_z|\log n}& z\leq Z_{n,1}.\label{upper_}
\end{align}
 We further obtain a lower bound via the following calculation:
	\begin{align*}
		&\sum_{k\in B_z}\partial_{k}l_2(\bu^*)-\sum_{k\in B_z}\partial_{k}l_2(\widetilde{\bu})\\
		=&\  \sum_{i: T_i\in\partial B_z}\sum_{j\in T_i\cap B_z}\sum_{t\in T_i\cap B^\complement_z}\left(\frac{\e{\widetilde{u}_{j}}}{\e{\widetilde{u}_{j}}+ \e{\widetilde{u}_t}}-\frac{\e{u^*_{j}}}{\e{u^*_{j}}+ \e{u^*_t}}\right)\\
		\geq&\ \sum_{i: T_i\in\partial B_z\cap\partial B_{z+1}}\sum_{j\in T_i\cap B_z}\sum_{t\in T_i\cap B^\complement_{z+1}}\left(\frac{\e{\widetilde{u}_{j}}}{\e{\widetilde{u}_{j}}+ \e{\widetilde{u}_t}}-\frac{\e{u^*_{j}}}{\e{u^*_{j}}+ \e{u^*_t}}\right)\\
		\geq&\ \ C''|\partial B_z\cap \partial B_{z+1}|\Delta b_z,
	\end{align*}
	where $C''$ is an absolute constant depending on $C_1$ and $M$ only. 
	Here the first inequality follows from the observation that for $j\in B_z$ and $t\in B_z^\complement$, 
	\begin{align*}
		&\frac{\e{\widetilde{u}_{j}}}{\e{\widetilde{u}_{j}}+ \e{\widetilde{u}_t}}\bigg/\frac{\e{u^*_{j}}}{\e{u^*_{j}}+ \e{u^*_t}} = \e{\widetilde{u}_j-u^*_j}\cdot\frac{\e{u^*_{j}}+ \e{u^*_t}}{\e{\widetilde{u}_{j}}+ \e{\widetilde{u}_t}}\\
		&\geq \e{\widetilde{u}_j-u^*_j}\cdot\min\left\{\e{u^*_{j}-\widetilde{u}_j}, \e{u^*_{t}-\widetilde{u}_t}\right\} = \e{\widetilde{u}_j-u^*_j}\cdot \e{u^*_{j}-\widetilde{u}_j} = 1.
	\end{align*}
	For the last inequality, we observe that for $j\in T_i\cap B_z$ and $t\in T_i\cap B^\complement_{z+1}$, by the mean-value theorem, 
	\begin{align*}
	&\frac{\e{\widetilde{u}_{j}}}{\e{\widetilde{u}_{j}}+ \e{\widetilde{u}_t}}-\frac{\e{u^*_{j}}}{\e{u^*_{j}}+ \e{u^*_t}} = \frac{1}{1+ \e{-(\widetilde{u}_j-\widetilde{u}_t)}}-\frac{1}{1+ \e{-(u^*_j-u^*_t)}}\\
	&\geq \frac{1}{1+ \e{-(u^*_j-u^*_t + \Delta b_z)}}-\frac{1}{1+ \e{-(u^*_j-u^*_t)}}\quad\quad \text{($\widehat{u}_j-\widehat{u}_t\geq u_j^*-u_t^*+\Delta b_z$)}\\
	& = \frac{e^{-\varrho}\Delta b_z}{(1+ \e{-\varrho})^2}\quad\quad \text{(where $\varrho\in [u^*_j-u^*_t, u^*_j-u^*_t + \Delta b_z]\subseteq [-2C_1, 2C_1+1]$)},
	\end{align*}
	so taking $C'' = \min_{x\in [-2C_1, 2C_1 + 1]}\exp(-x)/((1+\exp(-x))^2)$ suffices, where we have used again that $\Delta b_z\to 0$ for all $z<Z_{n,1}$ and all sufficiently large $n$ under Assumption~\plainref{ap:11}.  
	Setting $c = {4}M^2C''$ and combining the upper and lower bounds as before yields the admissibility of $\{B_z\}_{z\in [Z_{n,1}]}$.

\subsubsection{Proof of Lemma~\ref{yadongwang}}

We show that $\FI$ and $\FII$ hold with probability at least $1-n^{-4}$, respectively. The desired result follows by applying a union bound. 
 
We begin with $\FI$. For any $U\subset [n]$, the sequence defined by $A_1 = U$ and $A_2 = [n]$ is an admissible sequence. Therefore, Assumption~\plainref{ap:11} implies the following lower bound on its modified Cheeger constant:
\begin{align}
\frac{|\partial U|}{\min\{|U|, |U^\complement|\}\log n}\to\infty\quad\quad \text{for~} U\subset [n].\label{coco1}
\end{align}
We claim that the event in \eqref{coco1} is sufficient to ensure $\mathbb P(\FI)\geq 1-n^{-4}$ for all sufficiently large $n$, as desired in Assumption~\plainref{ap:11}. Indeed, for any $T\subset [n]$ with $|T| \leq M$ and $i, j\in T$, 
\begin{align}
\mathbb P(\text{$i\succ j$ on $T$}\mid T\in \E)&\geq \mathbb P(r_T(i) = 1\mid T\in \E)\nonumber\\
&\geq \frac{\e{u_i^*}}{\sum_{t\in T}\e{u^*_t}}\geq\frac{1}{MC_1^2}>0, &\text{(Assumptions~\plainref{ap:2} and \plainref{ap:1})}\label{pjl}
\end{align} 
where $r_T$ is the ranking outcome on $T$.  Therefore, for any partition $U$ and  $U^\complement$, define 
\begin{align*}
\mZ(U) = \sum_{T\in\partial U}\mathbf 1_{\{\text{there exist $i\in T\cap U$, $j\in T\cap U^\complement$ with $i\succ j$ on $T$}\}}.
\end{align*}
By definition, $\mZ(U)$ is a sum of independent Bernoulli random variables with
\begin{align}
\mathbb E[\mZ(U)]\stackrel{\eqref{pjl}}{\geq}\frac{|\partial U|}{MC_1^2}\stackrel{\eqref{coco1}}{\geq}48\min\{|U|, n-|U|\}\log n.\label{xisi}
\end{align} 
Note that for all $i\in U, j\in U^\complement$ and $\{i, j\}\subset T\in \E$, $i\prec j$ if and only if $\mZ(U) = 0$. 
By the Chernoff bound, 
\begin{align}
\mathbb P(\mZ(U)= 0)&\leq\mathbb P(\mZ(U)\leq \mathbb E[\mZ(U)]/2)\nonumber\\
&\leq \e{-\mathbb E[\mZ(U)]/8}\stackrel{\eqref{xisi}}{\leq}\e{-6\min\{|U|, n-|U|\}\log n}.
\end{align}
Taking a union bound over all $\emptyset\neq U\subset [n]$ yields
\begin{align*}
	\mathbb P\left(\FI\right) &= 1 - \mathbb P\left(\FI^\complement\right) {\geq } 1-\sum_{s=1}^{n-1}\sum_{U\subset [n]: |U|=s}\mathbb P(\mZ(U)= 0)\\
	&\geq 1 - \sum_{s=1}^{n-1}{n\choose s}\e{-6\min\{s, n-s\}\log n}\\
	&\geq 1 - \sum_{s=1}^{n-1}\e{-5\min\{s, n-s\}\log n}\\
	&\geq 1- n^{-4}. 
\end{align*}

For $\FII$, note that for each $U\subset [n]$, 
\begin{align*}
\sum_{k\in \US}\partial_{k}l_1(\bu^*)\stackrel{\eqref{eq:1++}, \eqref{sm}}{=}\sum_{i: T_i\in\partial \US}\left(\sum_{k\in \US\cap T_i}\psi(k; T_i, \pi_i, \bm u^*)\right)\\
\sum_{k\in \US}\partial_{k}l_2(\bu^*)\stackrel{\eqref{eq:1++}, \eqref{youwang}}{=}\sum_{i: T_i\in\partial \US}\left(\sum_{k\in \US\cap T_i}\varphi(k; T_i, \pi_i, \bm u^*)\right).
\end{align*}
Under Assumptions \plainref{ap:2}-\plainref{ap:1}, the summation (over $i$) in both settings above is over mean-zero independent random variables bounded by $M^2$. By Hoeffding's inequality, with probability at least $1-n^{-6|\US|}$, 
\begin{align*}
\bigg|\sum_{k\in \US}\partial_{k}l_{s}(\bu^*)\bigg| \leq {4}M^2\sqrt{|\US||\partial \US|\log n}\quad\quad s = 1, 2. 
\end{align*}
Taking a union bound over $U\subset [n]$ yields $\mathbb P(\FII)\geq 1-n^{-4}$.  We finish the proof of Lemma~\ref{yadongwang}.

\subsection{Proof of Theorem \plainref{thm:and}}\label{tn-t}
Denote by $l$ the log-likelihood or quasi-log-likelihood function, and $\bm w$ the corresponding marginal MLE or QMLE. Recall  we have demonstrated $\mathbb E[\nabla l(\bu^*)] = 0$ and $\nabla l(\bm w) = 0$ in the proof of Theorem \plainref{main:general}.
By the Taylor expansion of $\nabla l$ at $\bm u$, 
\begin{align}
- \nabla l(\bm u^*) = \nabla l(\bm w) - \nabla l(\bm u^*) = \cJ(\bm w, \bm u^*)(\bm w - \bm u^*),\label{zihao}
\end{align}
where 
\begin{align}
\cJ(\bm w, \bm u^*) = \int_0^1 \cH(t\bm w + (1-t)\bm u^*) \d t\label{myJ}
\end{align}
with ${\cH}(\bm u) = \nabla^2 l(\bm u)$ defined in Section~\plainref{sec:6+}. 
Note that $\cH({\bu^*})$ is a random matrix due to the multiple comparison outcomes unless $\bm w$ is the choice-one MLE or QMLE. This additional complexity makes the subsequent analysis more challenging as opposed to Luce's choice model. 

Recall $\cH^*(\bu^*) =\mathbb{E}[\cH({\bu^*})]$, the matrix $\D$ is a diagonal matrix with $\D_{ii} = -[\cH^*(\bu^*)]_{ii}$, and 
$\mathcal{L}_{\sym} = -\D^{-1/2}\cH^*(\bu^*)\D^{-1/2}$.
 Rewriting $\cJ(\bm w, \bm u^*)$ as a summation of three parts:
\begin{equation*}
	\cJ(\bm w, \bm u^*)=  \{\cJ(\bm w, \bm u^*) - \cH({\bu^*})\} + \{\cH({\bu^*}) -\cH^*(\bu^*) \} + \cH^*(\bu^*).
\end{equation*}
Substituting this into \eqref{zihao} and using $\mathcal{L}_{\sym} = -\D^{-1/2}\cH^*(\bu^*)\D^{-1/2}$ yields
\begin{align}
	&\D^{1/2}\mathcal L_\sym\D^{1/2}(\bm w - \bm u^*) = -\cH^*(\bu^*)(\bm w - \bm u^*) \label{AN_decomp} \\
	&=  \nabla l(\bm u^*) + \{\cJ(\bm w, \bm u^*) - \cH({\bu^*})\}(\bm w - \bm u^*) 
	 + \{\cH({\bu^*}) -\cH^*(\bu^*)\}(\bm w - \bm u^*).\nonumber
\end{align}
The weighted graph associated with $\mathcal L_\sym$ is connected whenever $\bw$ exists. In this case, the zero-eigenspace of $-\cH^*(\bm u^*)$ is one-dimensional and spanned by the all-ones vector $\langle 1\rangle$. 
Moreover, it can be verified using \plaineqref{wearediff} and \plaineqref{binyan} that $-\cJ(\bm w, \bm u^*)$ and $-\cH({\bu^*})$ are also weighted graph Laplacians with the same zero-eigenspace with $-\cH^*(\bm u^*)$. Since $\bw-\bu^*\in \text{span}(\langle 1\rangle)^\perp$, we can rearrange \eqref{AN_decomp} to obtain
\begin{eqnarray}\label{dpskl}
	\D^{1/2}(\bm w - \bm u^*) &=& \underbrace{{\mathcal{L}^\dagger_{\sym}} \D^{-1/2} \nabla l(\bm u^*)}_{({\rm i})} + \underbrace{{\mathcal{L}^\dagger_{\sym}} \D^{-1/2} \{\cJ(\bm w, \bm u^*) - \cH({\bu^*})\}(\bm w - \bm u^*)}_{({\rm ii})} \nonumber\\
	&& +\ \underbrace{{\mathcal{L}^\dagger_{\sym}} \D^{-1/2} \{\cH({\bu^*}) -\cH^*(\bu^*)\}(\bm w - \bm u^*)}_{({\rm iii})},
	\end{eqnarray}
where ${\mathcal{L}^\dagger_{\sym}}$ is the pseudoinverse of ${\mathcal{L}_{\sym}}$.

To analyze \eqref{dpskl}, we will use that each component of $\nabla l(\bu^*)$ can be written as a sum of mean-zero independent random variables by arranging summands from the same edge together (for example, using the $\psi$ and $\varphi$ notation in \eqref{eq:1} and \eqref{eq:231}). Therefore, each component of $\D^{-1/2} \nabla l(\bm u^*)$ is expected to converge to a normal distribution of constant order by the central limit theorem (CLT). Meanwhile, under Assumptions~\plainref{ap:deterministic}-\plainref{ap:deterministic+}, both $({\rm ii})$ and $({\rm iii})$ are higher-order terms. These heuristics are made precise by the following lemmas.

\begin{Lemma}\label{CLT:1}
Let $l$ denote the log-likelihood of the marginal MLE or the QMLE. 
Under Assumptions \plainref{ap:2}, \plainref{ap:1} and \plainref{ap:deterministic}, for any fix $k \in [n]$, 
	\begin{equation}\label{rightone}
		({\mathcal{L}^\dagger_{\sym}} \D^{-1/2} \nabla l(\bm u^*))_k \to N(0, \Sigma_{kk} ) 
	\end{equation}
in distribution as $n \to \infty$, where $ \Sigma = \D^{-1/2} \mathbb{E}[\nabla l(\bm u^*) \nabla l(\bm u^*)^\top] \D^{-1/2}.$
 \end{Lemma}
 \begin{Lemma} \label{CLT:2}
 Under Assumptions \plainref{ap:2}, \plainref{ap:1} and \plainref{ap:deterministic}, 
 	\begin{equation}\label{residual_2}
\lVert{\mathcal{L}^\dagger_{\sym}} \D^{-1/2} \{\cJ(\bm w, \bu^*) - \cH({\bu^*})\}(\bm w - \bm u^*)\lVert_\infty = o_p(1), 
 	\end{equation}
and $\lVert{\mathcal{L}^\dagger_{\sym}} \D^{-1/2} \{\cH({\bu^*}) -\cH^*(\bu^*)\}(\bm w - \bm u^*)\lVert_\infty = 0$ when $\bm w$ is the choice-one MLE or the QMLE. For the general marginal MLE, if we further assume Assumption~\plainref{ap:deterministic+} holds, then 
 	\begin{equation}\label{residual_3}
	 \lVert{\mathcal{L}^\dagger_{\sym}} \D^{-1/2} \{\cH({\bu^*}) -\cH^*(\bu^*)\}(\bm w - \bm u^*)\lVert_\infty = o_p(1).
\end{equation}
 \end{Lemma}
Combining \eqref{rightone}-\eqref{residual_3} and applying Slutsky's theorem yields the desired CLT result. For the reader's convenience, we sketch the proofs of Lemmas \ref{CLT:1}-\ref{CLT:2}; the technical details are given in Section \ref{sec:s5}. Lemma \ref{CLT:1} follows from an application of Chebyshev's inequality that involves entrywise estimates of the moments of $\mathcal A$. This step requires the upper bound on $\cc$ in Assumption~\plainref{ap:deterministic}. Lemma \ref{CLT:2} is more technical and based on a truncated error analysis associated with the Neumann series expansion of $\mathcal{L}^\dagger_{\sym}$. Specifically, for \eqref{residual_3}, to deal with the extra randomness of the comparison outcomes, we further apply a perturbation analysis similar to the leave-out methods developed in \cite{gao2023uncertainty,fan2022ranking}. Nevertheless, this additional tool will not be needed for Luce's choice model (which is the model considered in \cite{gao2023uncertainty,fan2022ranking}). This sets our method apart from theirs. 

To further illustrate, we sketch the proof of Lemma \ref{CLT:2}, for which we need the following estimates:

\begin{Lemma}\label{CLT:3}
	Under Assumptions \plainref{ap:2}, \plainref{ap:1} and \plainref{ap:deterministic}, 
	\begin{align}\label{residual_21}
		\|\D^{-1/2} \{\cJ(\bw, \bu^*) - \cH({\bu^*})\}(\bm w - \bm u^*)\|_\infty = O_p\left( (\Gamma^{\RE}_n)^2  \sqrt{\N_{n,+}}   \right).
	\end{align}
	For the general marginal MLE, if we further assume Assumption~\plainref{ap:deterministic+} holds, then 
	\begin{align}\label{residual_31}
		\|\D^{-1/2} \{\cH({\bu^*}) -\cH^*(\bu^*)\}(\bm w - \bm u^*)\|_\infty = O_p\Big( \frac{\Gamma_n^{\RE} \N_{n, +}\sqrt{\log n} + \sqrt{\N_{n, +}}\log n}{\lambda_2^\leave}\Big).
	\end{align}
\end{Lemma}
Assuming Lemma~\ref{CLT:3} holds, we demonstrate how to prove Lemma~\ref{CLT:2}.
Note that a crude $\ell_2$-$\ell_\infty$ estimate would fail badly since
 \begin{align*}
&\|\mathcal{L}^\dagger_{\sym} \D^{-1/2} \{\cJ(\bw, \bu^*) - \cH({\bu^*})\}(\bm w - \bm u^*)\|_\infty\\
\leq&\ \|\mathcal{L}^\dagger_{\sym} \D^{-1/2} \{\cJ(\bw, \bu^*) - \cH({\bu^*})\}(\bm w - \bm u^*)\|_2\\
\leq& \ \|\mathcal{L}^\dagger_{\sym}\|_2 \|\D^{-1/2} \{\cJ(\bw, \bu^*) - \cH({\bu^*})\}(\bm w - \bm u^*)\|_2\\
\leq&\  \sqrt{n}\|\mathcal{L}^\dagger_{\sym}\|_2\|\D^{-1/2} \{\cJ(\bw, \bu^*) - \cH({\bu^*})\}(\bm w - \bm u^*)\|_\infty. 
\end{align*}
Since $\|\mathcal{L}^\dagger_{\sym}\|_2\gtrsim 1$, the last term may not converge if loosely using the bound in Lemma \ref{CLT:3}.

To address this issue, we require a sharper estimate. Write ${\mathcal{L}^\dagger_{\sym}} = \sum_{t=0}^{\infty} (\A - \cP_1)^t - \cP_1$, where $\cP_1$ is the projection operator onto the zero-eigenspace of $\mathcal{L}^\dagger_{\sym}$. 
 One key observation is that 
 \begin{equation}\label{tobeu}
 	\cP_1 \D^{-1/2} \{\cJ(\bw, \bu^*) - \cH({\bu^*})\}(\bm w - \bm u^*) = \cP_1 \D^{-1/2} \{\cH({\bu^*}) -\cH^*(\bu^*)\}(\bm w - \bm u^*) = \bm 0,
 \end{equation}
which can be seen by noting that both $\D^{-1/2} \{\cJ(\bw, \bu^*) - \cH({\bu^*})\}\D^{-1/2}$ and $ \D^{-1/2} \{\cH({\bu^*}) -\cH^*(\bu^*)\} \D^{-1/2}$ are symmetric with their zero-eigenspace containing the range of $\cP_1$.  
Consequently,
\begin{eqnarray*}
&&{\mathcal{L}^\dagger_{\sym}} \D^{-1/2} \{\cJ(\bw, \bu^*) - \cH({\bu^*})\}(\bm w - \bm u^*)\\
&=& \sum_{t=0}^{\infty} (\A - \cP_1)^t  \D^{-1/2} \{\cJ(\bw, \bu^*) - \cH({\bu^*})\}(\bm w - \bm u^*).
\end{eqnarray*}

We conduct a truncated error analysis to control the series on the right-hand side. Under Assumption~\plainref{ap:deterministic}, $\|\A - \cP_1\|_2 = \max\{1-\lambda_2(\mathcal L_\sym), \lambda_n(\mathcal L_\sym)-1\} = 1-\mix<1$. Define 
\begin{align}
t_n = \frac{4M\log n}{1-\|\A - \cP_1\|_2} = \frac{4M\log n }{\mix}.\label{mytn}
\end{align}
As a result, 
\begin{align*}
\left\|\sum_{t = t_n}^{\infty} (\A- \cP_1)^t\right\|_2\leq\frac{\lVert \A- \cP_1\lVert_2^{t_n}}{1-\|\A-\cP_1\|_2}\leq \frac{n^{-4M}}{\mix}\leq n^{-2M},
\end{align*}
where the last step follows from Assumption~\plainref{ap:deterministic} for all sufficiently large $n$:
\begin{align*}
\frac{\log n\times (\Gamma_n^\RE)^2\N_{n,+}}{\mix\sqrt{\N_{n,-}}}\to 0&\Longrightarrow\frac{1}{\mix}\leq \frac{1}{(\Gamma_n^\RE)^2}\leq n^{2M}.
\end{align*}
For convenience, let $\Delta \cH \in\left\{\cJ(\bw, \bu^*) - \cH({\bu^*}), \cH({\bu^*}) - \cH^*(\bu^*)\right\}$.
Writing
\begin{align*}
&{\mathcal{L}^\dagger_{\sym}} \D^{-1/2} \Delta \cH(\bm w - \bm u^*)\\
 =&\ \underbrace{\sum_{t=0}^{t_n-1 } (\A - \cP_1)^t  \D^{-1/2} \Delta \cH (\bm w - \bm u^*)}_{({\rm iv})} + \underbrace{\sum_{t=t_n}^{\infty} (\A - \cP_1)^t  \D^{-1/2} \Delta \cH (\bm w - \bm u^*)}_{({\rm v})},
\end{align*}
we bound $\lVert (\rm v) \lVert_\infty$ and $\lVert (\rm iv) \lVert_\infty$ separately. For $\lVert (\rm v) \lVert_\infty$, we have
\begin{align*}
	&\lVert ({\rm v}) \lVert_\infty\leq\sqrt{n} \left\lVert \sum_{t = t_n}^{\infty} (\A- \cP_1)^t \right\lVert_2 \lVert \D^{-1/2} \Delta \cH(\bm w - \bm u^*) \lVert_\infty\\
	&\stackrel{\eqref{residual_21}, \eqref{residual_31}}{\leq} \begin{cases}
		n^{-M}(\Gamma^{\RE}_n)^2  \sqrt{\N_{n,+}},&\text{$\bm w\in\{\text{choice-one MLE, QMLE}\}$}\\
		\\
		n^{-M} \times\frac{\Gamma_n^{\RE} \N_{n, +}\sqrt{\log n} + \sqrt{\N_{n, +}}\log n}{\lambda_2^\leave},&\text{others}. 
	\end{cases}
\end{align*}
Meanwhile, 
\begin{align}
	&\lVert ({\rm iv}) \lVert_\infty \leq \sum_{t=0}^{t_n-1} \lVert(\A - \cP_1)^t  \D^{-1/2} \Delta \cH(\bm w - \bm u^*) \lVert_\infty \nonumber \\
	&\lesssim {t_n}  \sqrt{\frac{\N_{n,+}}{\N_{n,-}}} \lVert \D^{-1/2} \Delta \cH(\bm w - \bm u^*) \lVert_\infty \label{tobejustified}\\
	&\stackrel{\eqref{residual_21}, \eqref{residual_31}}{\lesssim} {t_n}  \sqrt{\frac{\N_{n,+}}{\N_{n,-}}} \times\begin{cases}
		(\Gamma^{\RE}_n)^2  \sqrt{\N_{n,+}},&\text{$\bm w\in\{\text{choice-one MLE, QMLE}\}$}\\
		\\
		\frac{\Gamma_n^{\RE} \N_{n, +}\sqrt{\log n} + \sqrt{\N_{n, +}}\log n}{\lambda_2^\leave},&\text{others}.	\end{cases}\nonumber
\end{align}
The second inequality \eqref{tobejustified} is far from obvious and requires computing $(\A - \cP_1)^t$, which is proved in Section~\ref{laile}. Based on Assumptions~\plainref{ap:deterministic}-\plainref{ap:deterministic+}, we conclude that both $\lVert ({\rm iv}) \lVert_\infty$ and  $\lVert ({\rm v}) \lVert_\infty$ are $o_p(1).$ We finish the proof of Theorem \plainref{thm:and}.

\section{Proofs of Lemmas Related to Uniform Consistency}\label{sec:s4} 
In this section, we provide the proof of Lemma \plainref{main:REs}, which verifies that both NURHM and HSBM are RE under appropriate conditions.

For any admissible sequence $\{A_j\}_{j\in [J]}$, by the definition of modified Cheeger constants in Definition~\plainref{def:ch}, 
\begin{align}
\sum_{j=1}^{J-1}\sqrt{\frac{\log n}{h_{\HH_n}(A_j)}}&\leq (J-1)\sqrt{\frac{\log n}{h_{\HH_n}}}.\label{kpl}
\end{align}
To establish the RE property for the hypergraph sequences of interest, it suffices to show the right-hand side of \eqref{kpl} converges to zero uniformly regardless of the choice of $\{A_j\}_{j\in [J]}$. We approach this by first obtaining a lower bound on $h_{\HH_n}$ and then a uniform upper bound on $J$, and we proceed to show this case by case. For convenience, the following notation will be used throughout the proof:
\begin{align*}
a_j &:= |A_j|& \Delta a_j := a_{j+1}-a_j\\
\partial A_{j\subseteq j+1} &:=\{e\in \partial A_j: e\subseteq A_{j+1}\}\subseteq \partial A_j.
\end{align*}
Without loss of generality, we assume $\Delta a_j\leq n-\Delta a_j$, as the opposite case can occur at most once in any admissible sequence, having an asymptotically negligible impact on the estimate on $J$. 

\subsection{NURHM}\label{huapangzi}

The proof consists of two steps. 

\smallskip

\textbf{\underline{Step I}: Lower bound on $h_{\HH_n}$.}

\smallskip
 
For $2\leq \ms\leq M$ and $U\subset [n]$ with $|U| = \st$, the number of edges in $\partial U$ with size $\ms$ is a sum of $\Lu(n, \st, \ms)$ independent Bernoulli random variables with probabilities bounded between $p_n^{(\ms)}$ and $q_n^{(\ms)}$, where $\Lu(n, \st, \ms) = {n\choose \ms} - {\st\choose \ms} - {{n-\st}\choose \ms}.$
Thus, 
\begin{align}
\sum_{\ms=2}^M \Lu(n, \st, \ms)p_n^{(\ms)}=:\mu_{n,-}(\st)\leq\mathbb E[|\partial U|]\leq \mu_{n,+}(\st) := \sum_{\ms=2}^M \Lu(n, \st, \ms)q_n^{(\ms)}.\label{pangermu}
\end{align} 
For $U_2\subset U_1\subseteq [n]$ with $|U_1| = \st_1$ and $|U_2|= \st_2$, the edges in $\partial U_2$ contained in $U_1$ with size $\ms$ are a subset of edges in $\partial (U_1\setminus U_2)$. Hence, 
\begin{align}
\mathbb E[|\{e\in\partial U_2: e\subset U_1\}|]\leq\mathbb E[|\partial (U_1\setminus U_2)|]\leq\mu_{n,+}(\st_1-\st_2).\label{married}
\end{align}
Both $\mu_{n,-}(\st)$ and $\mu_{n,+}(\st)$ will appear in the subsequent analysis. The following estimate on $\Lu(n, \st, \ms)$ will be useful and its proof is deferred to the end of the section. 
\begin{Lemma}\label{lemma:ln}
For $\st\geq \ms$, the following inequalities hold:
\begin{align}
\frac{\ms\min\{\st, n-\st\}}{n}{n\choose \ms}\lesssim \Lu(n, \st, \ms)\leq \frac{\ms\min\{\st, n-\st\}}{n-\ms+1}{n\choose \ms}.\label{lu1}
\end{align}
\end{Lemma}
Consequently, combining \eqref{lu1}, \eqref{pangermu}, and \plaineqref{myxis} yields
\begin{align}
\min\{\st, n-\st\}\xi_{n, -}\lesssim\mu_{n,-}(\st)\leq \mu_{n,+}(\st)\lesssim \min\{\st, n-\st\}\xi_{n, +},\label{jingyue}
\end{align}
where the implicit constant depends on $M$ only. 

By the Chernoff bound \citep{MR0057518}, for all $U\subset [n]$ with $|U| = \st\leq n/2$ and all sufficiently large $n$, the following inequalities hold with probability at least $1-n^{-5\st}$, 
\begin{align}
&\frac{1}{2}\mu_{n,-}(\st)\leq(1-\delta)\mu_{n,-}(\st)\leq(1-\delta)\mathbb E[|\partial U|]\leq |\partial U|\nonumber\\
&|\partial U| \leq (1+\delta)\mathbb E[|\partial U|]\leq (1+\delta)\mu_{n,+}(\st)\leq 2\mu_{n,+}(\st)& \delta = \sqrt{\frac{10\st\log n}{\mathbb E[|\partial U|]}},\label{lps}
\end{align}
where we used that
\begin{align*}
\delta = \sqrt{\frac{10\st\log n}{\mathbb E[|\partial U|]}}\stackrel{\eqref{pangermu}}{\leq} \sqrt{\frac{10\st\log n}{\mu_{n,-}(\st)}}\stackrel{\eqref{jingyue}}{\lesssim} \sqrt{\frac{\log n}{\xi_{n,-}}}\xrightarrow{\plaineqref{2022}} 0.
\end{align*}
Taking a union bound over $U$ in \eqref{lps} yields a uniform bound for all $U$ with $|U|\leq n/2$ (there are ${n\choose \st}\leq n^{\st}$ different subsets of $[n]$ with size $\st$) that holds with probability at least {$1-n^{-4}$}. 
In what follows we condition on this event.
On this event, the modified Cheeger constant $h_{\HH_n}$ can be bounded from below as 
\begin{align}
h_{\HH_n}\geq\min_{1\leq \st\leq n/2}\frac{\mu_{n,-}(\st)}{2\st}\stackrel{\eqref{jingyue}}{\gtrsim}\xi_{n,-}.\label{wenlu}
\end{align}
In particular, for all $U\subset [n]$, 
\begin{align}
|\partial U|\gtrsim \min\{|U|, n-|U|\}\xi_{n,-}. \label{coco}
\end{align}
\smallskip

\textbf{\underline{Step II}: Upper bound on $J$.}

\smallskip

Setting $U = A_i$ in \eqref{coco} yields the following lower bound on $|\partial A_j|$:
\begin{align}
|\partial A_j|&\gtrsim\min\{a_j, n-a_j\}\xi_{n, -}\label{2023}.
\end{align}
Meanwhile, we also have an upper bound on $|\partial A_{j\subseteq j+1}|$ as follows: 
\begin{align}
|\partial A_{j\subseteq j+1}|\stackrel{\eqref{married}}{\leq}|\partial (A_{j+1}\setminus A_j)|\stackrel{\eqref{lps}}{\leq} 2\mu_{n,+}(\Delta a_j)\stackrel{\eqref{jingyue}}{\lesssim}\Delta a_j\xi_{n,+},\label{gamma2024}
\end{align}
where the last step holds under the assumption $\Delta a_j\leq n-\Delta a_j$. 
Since $\{A_j\}_{j\in [J]}$ is an admissible sequence, $|\partial A_{j\subseteq j+1}|\geq |\partial A_j|/2$. This combined with \eqref{2023} and \eqref{gamma2024} implies that 
\begin{align*}
\frac{(n-a_{j})-(n-a_{j+1})}{\min\{a_j, n-a_j\}}=\frac{a_{j+1}-a_j}{\min\{a_j, n-a_j\}}=\frac{\Delta a_j}{\min\{a_j, n-a_j\}}\gtrsim \frac{\xi_{n,-}}{\xi_{n,+}}
\end{align*}
Consequently, we conclude that when $a_j\leq n/2$, $a_j$ grows exponentially with rate of at least $1+(C\xi_{n,-}/\xi_{n,+})$; when $a_j> n/2$, $n-a_j$ shrinks exponentially with rate of at least $1-(C\xi_{n,-}/\xi_{n,+})$, where $C>0$ is some absolute constant depending neither on $\{A_j\}_{j\in [J]}$ nor $n$. Since $a_J\leq n$, it takes $a_j$ at most $O(\xi_{n,+}(\log n)/\xi_{n,-})$ steps to reach $a_J$. Substituting this and \eqref{wenlu} into \eqref{kpl} and applying the Borel--Cantelli lemma over $n$ yields the desired result. We finish the proof of Lemma \plainref{main:REs} under the NURHM.

\subsection{HSBM}

The proof for the HSBM requires more careful counting. Specifically, we will provide a sharper upper bound on $J$ by exploiting the community structure in HSBM. In an HSBM introduced in Section~\plainref{hsbm}, the comparison hypergraph is $\mm$-uniform and has $K$ blocks $V_1, \ldots, V_K$ with $|V_i| = n_i$ and $\sum_{i\in [K]}n_i = n$. The probabilities of a given $\mm$-edge lying within and across $V_i, i \in [K]$ are $\omega_{n, i}$ and $\omega_{n,0}$, respectively.
In this setup, we let
\begin{align*}
	&a_{j, i} = |A_j\cap V_i| \quad\quad \Delta a_{j, i} = |(A_{j+1}\setminus A_j) \cap V_i|&i\in [K].
\end{align*}
 By definition,
\begin{align*}
	&\Delta a_{j, i} = a_{j+1, i}-a_{j,i} &\Delta a_{j} = a_{j+1}-a_{j}.
\end{align*}
Following the same reason of the assumption $\Delta a_j\leq n-a_j$, we further assume 
\begin{align}
	\Delta a_{j, i}\leq \frac{n_i}{2}\quad\quad i\in [K].\label{uselot}
\end{align}
The proof proceeds by lower bounding $h_{\HH_n}$ and upper bounding $J$.

\smallskip

\textbf{\underline{Step I}: Lower bound on $h_{\HH_n}$.}

\smallskip

This step is almost identical to the NURHM setting. Applying the same concentration argument as \eqref{lps} and a union bound yields that, with probability at least $1-n^{-3}$, the following analogs of \eqref{2023} and \eqref{gamma2024} hold for all admissible sequences:
\begin{align}
|\partial A_j|\gtrsim \mathbb E[|\partial A_j|]\to\infty\quad\quad |\partial A_{j\subseteq j+1}|\lesssim \mathbb E[|\partial (A_{j+1}\setminus A_j)|]\label{qiangbu},
\end{align}
where we kept the expectation rather than using the lower and upper bounds $\mu_{n,-}(t)$ and $\mu_{n,+}(t)$ for further processing. 
Moreover, the modified Cheeger constant can be lower bounded as
\begin{align}
h_{\HH_n}\gtrsim \zeta_{n,-}.\label{kaizi}
\end{align} 

\smallskip

\textbf{\underline{Step II}: Upper bound on $J$.}

\smallskip

Obtaining a sharp bound on $J$ in HSBM requires more work than in NURHM. In particular, we will utilize the homogeneity within  communities to refine estimates on $\mathbb E[|\partial A_j|]$ and $\mathbb E[|\partial (A_{j+1}\setminus A_j)|]$. The resulting bound we are targeting is $J = O(\log n)$. We will prove this by showing that, apart from at most a fixed number of steps that depend only on $K$, at each step $j$ there exists at least one community index $i\in [K]$ such that either $a_{j,i}$ grows exponentially with absolute constant rate larger than one from $j$ to $j+1$, or $(n_i-a_{j,i})$ shrinks exponentially with absolute constant rate less than one from $j$ to $j+1$. Thus, it takes at most $O(\log n)$ steps for each $a_{j,i}$ to grow to full. Since $K$ is finite, $J\lesssim \log n$. 

Recall $\Lu(n, \st, \ms) = {n\choose \ms} - {\st\choose \ms} - {{n-\st}\choose \ms}$ in Section \ref{huapangzi} that counts the number of $\ms$-boundary edges of a size-$\st$ subset of an size-$n$ set. 
Grouping edges in $\partial A_j$ and $\partial (A_{j+1}\setminus A_j)$ based on whether they are within or across $V_i$, 
\begin{align}
\mathbb E[|\partial A_j|] &= \sum_{i\in [K]}\underbrace{\omega_{n,i}\Lu(n_i, a_{j,i}, \mm)}_{\vartheta_i} + \underbrace{\omega_{n,0}\left(\Lu(n, a_j, \mm)-\sum_{i \in [K]}\Lu(n_i, a_{j,i}, \mm)\right)}_{\vartheta_{K+1}}\nonumber
\end{align}
and
{\small
\begin{align}
\mathbb E[|\partial (A_{j+1}\setminus A_j)|]&=\sum_{i\in [K]}\underbrace{\omega_{n,i}\Lu(n_i, \Delta a_{j,i}, \mm)}_{\tau_i} + \underbrace{\omega_{n,0}\left(\Lu(n, \Delta a_j, \mm)-\sum_{i \in [K]}\Lu(n_i, \Delta a_{j,i}, \mm)\right)}_{\tau_{K+1}}. \nonumber
\end{align}}
Thanks to $\eqref{qiangbu}$, for all sufficiently large $n$, $ \mathbb E[|\partial A_j|] \geq K+1$ for all nontrivial $A_j$. Therefore,
the admissibility condition $|\partial A_{j\subseteq j+1}|\geq |\partial A_j|/2$ yields
\begin{align}
\frac{2\sum_{i\in [K+1]}\tau_i}{\sum_{i\in [K+1]}(\vartheta_i +1)} \stackrel{}{\geq} \frac{\sum_{i\in [K+1]}\tau_i}{\sum_{i\in [K+1]}\vartheta_i}=\frac{\mathbb E[|\partial (A_{j+1}\setminus A_j)|]}{\mathbb E[|\partial A_j|]}\gtrsim 1,\label{parking}
\end{align}
which implies 
\begin{align}
\max_{i\in [K+1]}\frac{\tau_i}{\vartheta_i + 1}\gtrsim 1.\label{tyre}
\end{align}
The addition of constant $1$ ensures each denominator is positive. In the following, we discuss the implications of $\tau_i/(\vartheta_i +1)\gtrsim 1$ for each $i\in [K+1]$. 

We first consider the scenario where $\tau_i/(\vartheta_i +1)\gtrsim 1$ for some $i\in [K]$. 
In this setting, if $\vartheta_i>0$, then by the inequalities in \eqref{lu1} and \eqref{uselot}, 
\begin{align*}
\frac{\Delta a_{j, i}}{\min\{a_{j,i}, n_i-a_{j,i}\}}\stackrel{\eqref{lu1}}{\gtrsim}\frac{\tau_i}{\vartheta_i}\geq\frac{\tau_i}{\vartheta_i + 1}\gtrsim 1.
\end{align*}
Consequently, $\{a_{j,i}\}_j$ grows exponentially if $a_{j,i}\leq n_i/2$ with some rate greater than $1$ or $\{n_i-a_{j,i}\}_j$ decreases exponentially if $a_{j,i}> n_i/2$ with some rate less than $1$. Both rates are absolute and independent of the choice of the admissible sequence and $n_i$.  

If $\vartheta_i=0$, then $a_{j,i} = n_{i}$ or $0$, implying either $\Delta a_{j,i} = 0$ (which further implies $\tau_i=0$ so that both $\tau_i$ and $\vartheta_i$ can be removed from the summations in \eqref{parking} before proceeding to \eqref{tyre}) or $a_{j+1, i}>a_{j, i}=0$. The latter can occur at most once for each $i$ (i.e., changing from zero to nonzero) and thus can be safely ignored.   

The case $\tau_{K+1}/(\vartheta_{K+1}+1)\gtrsim 1$ is more complicated. We first observe that as long as $A_j\subset [n]$, $\vartheta_{K+1}>0$. 
We introduce the following notation for the index set. For any $\bm\alpha = (\alpha_1, \ldots, \alpha_{2K})^\top\in\mathbb N^{2K}$, let 
\begin{align*}
\mGamma(\bm\alpha) :=\bigg\{&\bm \mm = (\mm_1, \ldots, \mm_{2K})^\top\in\mathbb N^{2K}: \sum_{i\in [2K]}\mm_i = \mm; \ \mm_i\leq \alpha_i, \ i\in [2K]; \\
&\ \mm_{2i-1}+\mm_{2i}<\mm, \ i\in [K]; \ \sum_{i\in [K]}\mm_{2i-1}<\mm; \sum_{i\in [K]}\mm_{2i}<\mm\bigg\}.
\end{align*}
Recall that $M$ is the edge size in an HSBM. The elements in $\mGamma(\bm\alpha)$ partition the $M$-boundary edges of a subset associated with $\bm\alpha$ into equivalent classes based on their intersected sizes with different communities, counting from both within and outside. Specifically, denote by $\bm a_j = (a_{j, 1}, n_1-a_{j, 1}, \ldots, a_{j,i}, n_K-a_{j, K})^\top\in\mathbb N^{2K}$ and $\Delta\bm a_j = (\Delta a_{j, 1}, n_1-\Delta a_{j, 1}, \ldots, \Delta a_{j,K}, n_K-\Delta a_{j, K})^\top\in\mathbb N^{2K}$. When  $\bm \alpha = \bm a_j$, $\mm_{2i-1}$ and $\mm_{2i}$ count the size of the intersection of an $\mm$-boundary edge of $A_j$ with $V_i\cap A_j$ and  $V_i\cap A^\complement_{j}$, respectively. When $\bm \alpha = \Delta\bm a_j$, $\mm_{2i-1}$ and $\mm_{2i}$  count the size of intersection of an $\mm$-boundary edge of $A_j$ with $V_i\cap (A_{j+1}\setminus A_j)$ and with $V_i\cap (A_{j+1}\setminus A_j)^\complement$.
 
Similar to Vandermonde's identity, we can rewrite $\vartheta_{K+1}$ and $\tau_{K+1}$ as the following sums:
{\scriptsize	\begin{align*}
\vartheta_{K+1} &= \omega_{n,0}\sum_{\bm \mm\in\mGamma(\bm a_j)}\prod_{i\in [K]}{a_{j,i}\choose \mm_{2i-1}}{n_i-a_{j,i}\choose \mm_{2i}}\gtrsim \omega_{n,0}\sum_{\bm \mm\in\mGamma(\bm a_j)}\prod_{i\in [K]}a_{j,i}^{\mm_{2i-1}}(n_i-a_{j,i})^{\mm_{2i}}\\ 
\tau_{K+1} &= \omega_{n,0}\sum_{\bm \mm\in\mGamma(\Delta\bm a_j)}\prod_{i\in [K]}{\Delta a_{j,i}\choose \mm_{2i-1}}{n_i-\Delta a_{j,i}\choose \mm_{2i}}\stackrel{\eqref{uselot}}{\lesssim} \omega_{n,0}\sum_{\bm \mm\in\mGamma(\Delta\bm a_j)}\prod_{i\in [K]}(\Delta a_{j,i})^{\mm_{2i-1}}n_i^{\mm_{2i}},  
\end{align*}}
where the inequalities hold for any valid choices of $a_{j,i}, \Delta a_{j,i}, n_i$ and do not require divergence of any of these numbers. 

To utilize these bounds in $\tau_{K+1}/\vartheta_{K+1}\geq\tau_{K+1}/(\vartheta_{K+1}+1)\gtrsim 1$, we need a more explicit estimate on $\tau_{K+1}$. Since $K$ is finite, it suffices to consider a potential $\bm \mm\in\mGamma(\Delta\bm a_j)$ whose associated summand is maximized in asymptotic order when $n\to\infty$. For a fixed summand $\prod_{i\in [K]}(\Delta a_{j,i})^{\mm_{2i-1}}n_i^{\mm_{2i}}$, since $\Delta a_{j,i}\leq n_i$, one may wish to put the support of $\bm \mm$ on even indices to maximize this quantity. However, this is not allowed since $\sum_{i\in [K]}\mm_{2i} \leq \mm-1$, that is, a boundary edge of $(A_{j+1}\setminus A_{j})$ necessarily intersects both $(A_{j+1}\setminus A_{j})$ and $(A_{j+1}\setminus A_{j})^\complement$. This upper bound is attainable since for any $\bm \mm$ with $\sum_{i\in [K]}\mm_{2i-1} > 1$, one can lift it to some $\bm \mm'$ satisfying $\sum_{i\in [K]}\mm'_{2i} = \mm-1$ without decreasing the asymptotic order of the summand. For instance, one can tentatively reduce any $\mm_{2i-1}\geq 1$ by one while increasing $\mm_{2i}$ by one at a time; if increasing $\mm_{2i}$ is not allowed, that is, $\mm_{2i} = n_i - \Delta a_{j,i}>0$, then increase $\mm_{2i_0^*}$ by one for $i^*_0 = \argmax_{i\in [K]}n_i$ (which is always possible since $n_{i^*_0}-\Delta a_{j, i^*_0}\geq n_{i^*_0}/2\gtrsim n\geq M$ due to \eqref{uselot}). Implementing this until the support condition is satisfied would yield a modified $\bm \mm'\in\mGamma(\Delta a_j)$ with nondecreasing asymptotic order than $\bm\mm$.

Therefore, we can restrict to the subset of $\bm \mm\in\mGamma(\Delta a_j)$ with $\sum_{i\in [K]}\mm_{2i}=\mm-1$ in order to maximize the summand. To further maximize the summand, one may wish to set $\mm_{2i_0^*}=\mm-1$. However, this is not always allowed. For example, if $(A_{j+1}\setminus A_j)\subseteq V_{i^*_0}$, then setting $\mm_{2i_0^*}=\mm-1$ yields a boundary edge contained in $V_{i_0^*}$ rather than across $V_{i_0^*}$. Nevertheless, one can always initially set $\mm_{2i_0^*}=\mm-2$ and consider the following optimization problem: 
 \begin{align*}
(i_1^*, i_2^*) \in \argmax_{\substack{i_1\in [K]: \Delta a_{j, i_1}>0\\ i_2\in [K]\setminus\{i_1\}}} \Delta a_{j, i_1}n_{i_2}.
\end{align*}
Following the reasoning above, we have
\begin{align*}
\sum_{\bm \mm\in\mGamma(\Delta\bm a_j)}\prod_{i\in [K]}(\Delta a_{j,i})^{\mm_{2i-1}}n_i^{\mm_{2i}}\lesssim n^{\mm-2}_{i_0^*}\Delta a_{j, i^*_1}n_{i_2^*}.
\end{align*}
Therefore, $\tau_{K+1}/\vartheta_{K+1}\geq \tau_{K+1}/(\vartheta_{K+1}+1)\gtrsim 1$ implies that 
\begin{align}
n^{\mm-2}_{i_0^*}\Delta a_{j, i^*_1}n_{i_2^*}\gtrsim\sum_{\bm \mm\in\mGamma(\bm a_j)}\prod_{i\in [K]}{a_{j,i}\choose \mm_{2i-1}}{n_i-a_{j,i}\choose \mm_{2i}}\label{consq}.
\end{align}

To see the consequences of \eqref{consq}, note that $n_{i_1^*}-a_{j,i_1^*}\geq\Delta a_{j,i_1^*}> 0$. 
Consider two choices of $\bm \mm\in\mathbb N^{2K}$ depending on whether $i_0^*=i_2^*$ or not:
\begin{enumerate} 
\item [(i)] If $i_0^*=i_2^*$ (which holds if $i_0^*\neq i_1^*$), take 
\begin{align*}
\mm_t = \begin{cases}
 (\mm-1)\cdot \mathbf 1_{\{a_{j,i_0^*}\leq n_{i_0^*}/2\}}& t = 2i_0^*\\
 (\mm-1)\cdot \mathbf 1_{\{a_{j,i_0^*}> n_{i_0^*}/2\}}& t = 2i_0^*-1\\
\mathbf 1_{\{a_{j,i_0^*}\leq n_{i_0^*}/2\}}& t = 2i_1^*-1\\
\mathbf 1_{\{a_{j,i_0^*}> n_{i_0^*}/2\}} & t = 2i^*_1
\end{cases} 
\end{align*}
Since $n_{i_1^*}-a_{j,i_1^*}>0$, setting $M_{2i^*_1} = \mathbf 1_{\{a_{j,i_0^*}> n_{i_0^*}/2\}}$ is always valid. 
For $M_{2i^*_1-1} = \mathbf 1_{\{a_{j,i_0^*}\leq n_{i_0^*}/2\}}$ to be valid, one must ensure $a_{j, i_1^*}>0$, which can fail at most once for each $i_1^*$ since $a_{j+1,i_1^*}-a_{j,i_1^*}=\Delta a_{j,i_1^*}> 0$. Consequently, this choice of $\bm M$ is valid ($\bm \mm\in\mGamma(\bm a_j)$) except for at most $K-1$ times.   
\item [(ii)] If $i_0^*\neq i_2^*$ (which implies $i_0^*= i_1^*$), take 
\begin{align*}
\mm_t = \begin{cases}
 (\mm-2)\cdot \mathbf 1_{\{a_{j,i_0^*}\leq n_{i_0^*}/2\}} + \mathbf 1_{\{a_{j,i_0^*}> n_{i_0^*}/2\}}& t = 2i_0^*\\
 (\mm-2)\cdot \mathbf 1_{\{a_{j,i_0^*}> n_{i_0^*}/2\}}+ \mathbf 1_{\{a_{j,i_0^*}\leq n_{i_0^*}/2\}}& t = 2i_0^*-1\\
\mathbf 1_{\{a_{j,i_2^*}\leq n_{i_2^*}/2\}}& t = 2i_2^*\\
\mathbf 1_{\{a_{j,i_2^*}> n_{i_2^*}/2\}}& t = 2i_2^*-1
\end{cases} 
\end{align*}
Similar to case (i), $M_{2i_0^*-1}$ may not take the value one if $0=a_{j,i_0^*}\leq n_{i_0^*}/2$, but this can occur at most once since $a_{j+1,i_0^*}-a_{j,i_0^*}=\Delta a_{j,i_0^*}> 0$.   
\end{enumerate}

Whenever the above choice of $\bm M$ is valid, we can lower bound the sum in the right-hand side of \eqref{consq} using the single summand associated with $\bm \mm$ to obtain 
\begin{align*}
\frac{\Delta a_{j, i_1^*}}{n_{i_1^*}-a_{i_1^*}}\gtrsim 1\quad\text{or}\quad \frac{\Delta a_{j, i_1^*}}{a_{i_1^*}}\gtrsim 1.
\end{align*}
Consequently, $\{a_{j,i_1^*}\}_j$ grows exponentially with some rate greater than $1$ or $\{n_{i_1^*}-a_{j,i_1^*}\}_j$ decreases exponentially with some rate that is less than $1$, where both rates are absolute and depend neither on the choice of the admissible sequence nor $n_{i_1^*}$.  

Putting all cases implied by \eqref{tyre} together, we conclude that, except for at most a fixed number of times (which depends on $K$ only), there exists at least one community index $i\in [K]$ such that either $a_{j,i}$ grows exponentially with absolute constant rate larger than one from $j$ to $j+1$, or $(n_i-a_{j,i})$ shrinks exponentially with absolute constant rate less than one from $j$ to $j+1$. Hence, $J\lesssim \log n$. Substituting this and \eqref{kaizi} into \eqref{kpl} and applying the Borel--Cantelli lemma over $n$ yields the desired result. We finish the proof of Lemma \plainref{main:REs} under the HSBM.

\subsubsection{Proof of Lemma \ref{lemma:ln}}
Without loss of generality, we assume $n-\st>\ms$; the other case can be discussed similarly. 
It follows from the direct computation that
\begin{align*}
	\Lu(n, \st, \ms) &= {n\choose \ms}\underbrace{\left[1-\frac{\st\cdots (\st-\ms+1)}{n\cdots (n-\ms+1)}-\frac{(n-\st)\cdots (n-\st-\ms+1)}{n\cdots (n-\ms+1)}\right]}_{(\#)}.
\end{align*}
It is easy to see that 
\begin{align*}
	1-\left(\frac{\st}{n}\right)^\ms-\left(1-\frac{\st}{n}\right)^\ms\leq (\#)\leq 1-\left(\frac{\st-\ms+1}{n-\ms+1}\right)^\ms-\left(1-\frac{\st}{n-\ms+1}\right)^\ms.
\end{align*}
The lower bound in \eqref{lu1} follows by noting $[1-x^\ms-(1-x)^\ms]/[\ms\min\{x, 1-x\}]$ is uniformly away from $0$ for $x\in [0,1]$ (when $x=0$ or $x=1$ the ratio is understood by taking the limit), and the upper bound follows by noting $(1-x)^\ms\geq 1-\ms x$ for $x\in [0,1]$. We finish the proof of Lemma \ref{lemma:ln}.

\section{Proofs of Lemmas Related to Asymptotic Normality}\label{sec:s5}

We provide the technical details omitted while proving the asymptotic normality result for the deterministic comparison graph sequence in Section~\ref{tn-t}, as well as the remaining lemmas in Section~\plainref{sec:6+}. To proceed, we first complete the proof of Lemma \ref{CLT:2} assuming Lemma \ref{CLT:3} holds in Section \ref{laile}. Then, we prove Lemma \ref{CLT:3} and Lemma \ref{CLT:1} in Section \ref{laile2} and Section \ref{laile3} respectively.  
 In Section \ref{laile4},
we prove Lemma~\ref{newtech} which provides explicit bound on the parameters used in Assumptions \plainref{ap:deterministic}-\plainref{ap:deterministic++}
in the context of NURHM and HSBM. In Section \ref{new19}, we establish Lemma~\plainref{lastone}.

\subsection{Proof of Lemma \ref{CLT:2}}\label{laile}
We complete the proof of Lemma \ref{CLT:2} assuming Lemma \ref{CLT:3} holds. Recall that
\begin{align}
	\Delta \cH \in\left\{\cJ(\bw, \bu^*) - \cH({\bu^*}), \ \cH({\bu^*}) - \cH^*(\bu^*)\right\}.\label{Delta_H_d}
\end{align}
Following our discussion in Section~\ref{tn-t}, it remains to provide an upper bound on $\lVert ({\rm iv}) \lVert_\infty$ as needed in the second step in \eqref{tobejustified}. Specifically, we aim to show for any $t \geq 0$,
\begin{align*}
	\lVert	(\A - \cP_1)^t\D^{-1/2} \Delta \cH(\bm w - \bm u^*) \lVert_\infty \lesssim \sqrt{\frac{\N_{n,+}}{\N_{n,-}}} \lVert	\D^{-1/2} \Delta \cH(\bm w - \bm u^*) \lVert_\infty,
\end{align*}
where $\Delta \cH$ is defined in \eqref{Delta_H_d}. 

Recall that $  \A = \D^{-1/2}\W\D^{-1/2}$ and $-\cH^*(\bm u^*) = \D - \W$. For convenience, we let $\cH_{ij} := [\cH^*(\bm u^*)]_{ij}$ for $i \neq j \in [n]$ and $\cH_{ii} = 0$ for $i \in [n]$.  The explicit form of $\A$ and $\cP_1$ can be written as 
\begin{eqnarray}\label{xiuwuwang_d}
	[\A]_{ij} = \frac{\h_{ij}}{\sqrt{\D_{ii}\D_{jj}}}, \ \ \ \ \ [\cP_1]_{ij} = \frac{\sqrt{\D_{ii}\D_{jj}}}{\sum_{i\in [n]}\D_{ii}} \ \ \qquad i,j \in[n].
\end{eqnarray}

To prove Lemma \ref{CLT:2}, recall from \eqref{tobeu} that $\cP_1\D^{-1/2} \Delta \cH(\bm w - \bm u^*) = \bm 0$ and $\A\cP_1 = \cP_1\A$. As a result,
\begin{eqnarray*}
	(\A - \cP_1)^t\D^{-1/2} \Delta \cH(\bm w - \bm u^*) = \A^t\D^{-1/2} \Delta \cH(\bm w - \bm u^*).
\end{eqnarray*}
Then,
\begin{align*}
	\lVert	\A^t\D^{-1/2} \Delta \cH(\bm w - \bm u^*) \lVert_\infty \leq \lVert	\A^t \lVert_{\infty\to\infty} \lVert	\D^{-1/2} \Delta \cH(\bm w - \bm u^*) \lVert_\infty,
\end{align*}
where $\|B\|_{\infty\to\infty} = \max_{i\in[n]} \sum_{j\in[n]} |B_{ij}| $ for a $n \times n$ matrix $B$. It remains to bound $\lVert\A^t \lVert_{\infty\to\infty}$ to obtain the desired results. 

Note that under Assumptions~\plainref{ap:2}-\plainref{ap:1} and \plaineqref{wearediff}-\plaineqref{binyan}, 
\begin{equation}\label{degree_concentration_d}
	\N_{n,-} \lesssim \min_{j \in[n]} \D_{jj} \leq \max_{j \in[n]} \D_{jj} \lesssim \N_{n,+}
\end{equation}
for both the marginal MLE and QMLE, where $\N_{n, \pm}$ are defined in \plaineqref{mydn}. We now can compute $\lVert \A^t\lVert_\infty$, $t > 0$ as follows:
\begin{eqnarray*}
	\lVert \A^t\lVert_{\infty\to\infty} &= & \max_{i \in [n]} \sum_{j_1 \in [n]} \cdots \sum_{j_t \in [n]} \frac{\h_{ij_1}\h_{j_1j_2} \cdots \h_{j_{t-1}j_t} }{\sqrt{\D_{ii}} \D_{j_1j_1}\D_{j_2j_2} \cdots \D_{j_{t-1}j_{t-1}} \sqrt{\D_{j_tj_t}}}\\
	&\stackrel{\eqref{degree_concentration_d}}{\lesssim}& \frac{1}{\sqrt{\N_{n,-}}}\max_{i \in [n]} \sum_{j_1 \in [n]} \cdots \sum_{j_t \in [n]} \frac{\h_{ij_1}\h_{j_1j_2} \cdots \h_{j_{t-1}j_t} }{\sqrt{\D_{ii}} \D_{j_1j_1} \D_{j_2j_2} \cdots \D_{j_{t-1}j_{t-1}}}\\
	& = & \frac{1}{\sqrt{\N_{n,-}}}\max_{i \in [n]}\sum_{j_1 \in [n]} \cdots \sum_{j_{t-1} \in [n]} \frac{\h_{ij_1}\h_{j_1j_2} \cdots \h_{j_{t-2}j_{t-1}} }{\sqrt{\D_{ii}} \D_{j_1j_1} \D_{j_2j_2} \cdots \D_{j_{t-2}j_{t-2}}}\\
	& = &  \frac{1}{\sqrt{\N_{n,-}}}\max_{i \in [n]} \sqrt{\D_{ii}} \lesssim \sqrt{\frac{\N_{n,+}}{\N_{n,-}}}.
\end{eqnarray*}
When $t=0$, $\lVert \A^t\lVert_{\infty\to\infty} = 1.$ Therefore, for any $t \geq 0$, 
\begin{align*}
	\lVert	\A^t\D^{-1/2} \Delta \cH(\bm w - \bm u^*) \lVert_\infty \lesssim 	\sqrt{\frac{\N_{n,+}}{\N_{n,-}}}  \lVert	\D^{-1/2} \Delta \cH(\bm w - \bm u^*) \lVert_\infty.
\end{align*}
This justifies the inequality \eqref{tobejustified}. We finish the proof of Lemma \ref{CLT:2}.

\subsection{Proof of Lemma \ref{CLT:3}}\label{laile2}

\subsubsection{Proof of \eqref{residual_21} in Lemma \ref{CLT:3}}
Recall the Hessian matrices of the marginal MLE and QMLE computed in \plaineqref{wearediff} and \plaineqref{binyan}. 
We begin by summarizing a few common properties shared by Hessian matrices of both the marginal MLE and QMLE:
\begin{itemize}
	\item [(I)] (Graph Laplacian) For $k\in [n]$, $\{\cH(\bm u)\}_{kk} + \sum_{k'\neq k}\{\cH(\bm u)\}_{kk'} = 0$;
	\item [(II)] (Sparsity) For $k \neq k'$, $\{\cH(\bm u)\}_{kk'} = 0$ unless there exists some $T_i\in \E$ such that $\{k, k'\}\subseteq T_i$;
	\item [(III)] (Lipschitz property of the summand) Consider $\bu, \bv \in \mathbb{R}^n$ satisfying $\|\bu - \bv \|_\infty < 1$. 
\end{itemize}
In the marginal MLE, we have
\begin{eqnarray}\label{panglip_d}
	&&\left|\sum_{j\in [r_i(k)\wedge r_i(k')\wedge y_i]} \left\{\frac{\e{u_k}\e{u_{k'}}}{(\sum_{t \geq j}\e{u_{\pi_i(t)}})^2} - \frac{\e{v_k}\e{v_{k'}}}{(\sum_{t \geq j}\e{v_{\pi_i(t)}})^2} \right\} \right|\nonumber\\
	&\lesssim&\sum_{j\in [r_i(k)\wedge r_i(k')\wedge y_i]}  \max_{t \geq j} | u_t - v_t| \lesssim \lVert \bu - \bv \lVert_\infty.
\end{eqnarray}
In the QMLE, we have
\begin{align}\label{panglip}
\left|\frac{\e{u_k}\e{u_{k'}}}{(\e{u_k} + \e{u_k'})^2}-\frac{\e{v_k}\e{v_{k'}}}{(\e{v_k} + \e{v_k'})^2}\right|\lesssim \lVert \bu - \bv \lVert_\infty.
\end{align}
Using the above properties, we can bound $[\{\cJ(\bm w, \bu^*) - \cH({\bu^*})\}(\bm w - \bm u^*)]_k$ as follows:
\begin{align*}
	& |[\{\cJ(\bm w, \bu^*) - \cH({\bu^*})\}(\bm w - \bm u^*)]_k | \\
	=&\ \left| \sum_{k' \neq k} [\cJ(\bm w, \bu^*) - \cH({\bu^*})]_{kk'} \{(w_k - u^*_k) - (w_{k'} - u^*_{k'})\}\right|\\ 
		\stackrel{\eqref{myJ}}{=}&\ \left| \int_0^1\sum_{k' \neq k} [\cH(t\bm w + (1-t)\bu^*) - \cH({\bu^*})]_{kk'} \{(w_k - u^*_k) - (\widehat{u}_{k'} - u^*_{k'})\}\d t\right|\\
	\stackrel{\eqref{panglip_d}, \eqref{panglip}}{\lesssim}&\ \sum_{k' \neq k} \sum_{i: \{k, k'\}\subseteq T_i} \lVert \bm w - \bu^* \lVert_\infty \left| \{(w_k - u^*_k) - (w_{k'} - u^*_{k'})\}\right|\\
	\leq\ & 2 \sum_{k' \neq k} \sum_{i: \{k, k'\}\subseteq T_i} 
	\lVert {\bm w} - \bu^* \lVert_\infty^2 =   2  \lVert {\bm w} - \bu^* \lVert_\infty^2 \sum_{i: k \in T_i} |T_i|.
\end{align*}

Meanwhile, note that $\D$ is the degree matrix of $\cH^*(\bm u^*).$ Under Assumptions \plainref{ap:2}-\plainref{ap:1}, 
\begin{eqnarray}\label{forgot}
	\D_{kk} \gtrsim \sum_{i: k \in T_i} |T_i|.
\end{eqnarray}
Consequently,
\begin{align}\label{Lemma_D3_1}
&	|[\D^{-1/2} \{\cJ(\bm w, \bu^*) - \cH({\bu^*})\}(\bm w - \bm u^*)]_k| \nonumber \\
&	\lesssim   \lVert {\bm w} - \bu^* \lVert_\infty^2 \sqrt{\sum_{i: k \in T_i} |T_i|}\lesssim \lVert {\bm w} - \bu^* \lVert_\infty^2\sqrt{\N_{n,+}}.
\end{align}
According to Theorem \plainref{main:general}, for all sufficiently large $n$, with probability at least $1-n^{-3}$, 
\begin{equation}\label{Lemma_D3_2}
	\lVert \bm w - \bu^* \lVert_\infty^2 \lesssim (\Gamma^{\RE}_n)^2.
\end{equation}
Substituting this into \eqref{Lemma_D3_2} into \eqref{Lemma_D3_1} and taking a union bound over $k \in [n]$ yields
 \begin{align*}
	\|\D^{-1/2} \{\cJ(\bw, \bu^*) - \cH({\bu^*})\}(\bm w - \bm u^*)\|_\infty = O_p\left( (\Gamma^{\RE}_n)^2  \sqrt{\N_{n,+}}   \right).
\end{align*}

\subsubsection{Proof of \eqref{residual_31} in Lemma \ref{CLT:3}}
The results for the choice-one MLE and the QMLE are trivial since $\cH({\bu^*}) = \cH^*(\bu^*)$. 
One can check this by noting that $\cH(\bu^*)$ in \plaineqref{binyan} is independent of the comparison outcomes $\{\pi_i\}_{i\in [N]}$, and $\cH(\bu^*)$ in \plaineqref{wearediff} is independent of the comparison outcomes $\{\pi_i\}_{i\in [N]}$ if $y_i=1$ for all $i\in [N]$. As a result, $\|\D^{-1/2} \{\cH({\bu^*}) - \cH^*(\bu^*)\}(\bm w - \bm u^*)\|_\infty = 0$. Under such circumstances, \eqref{residual_31} automatically holds. 

Things become more complicated for marginal MLE in general as $\cH({\bu^*}) \neq \cH^*(\bu^*)$.  
In this case, we write $\bw = \buu$. Similar to the previous section, for $k\in [n]$, separating the summation of diagonal and off-diagonal entries, 
\begin{eqnarray*}
	&&	[\{\cH({\bu^*}) -\cH^*(\bu^*) \}(\buu - \bm u^*)]_k = \sum_{k' \neq k} [\cH({\bu^*}) - \cH^*(\bu^*)]_{kk'} \{(\widehat{u}_k - u^*_k) - (\widehat{u}_{k'} - u^*_{k'})\}\\
	&=&\underbrace{ \sum_{k' \neq k} [\cH({\bu^*}) - \cH^*(\bu^*)]_{kk'} (\widehat{u}_k - u^*_k)}_{({\text{vi}})} - \underbrace{\sum_{k' \neq k} [\cH({\bu^*}) - \cH^*(\bu^*)]_{kk'} (\widehat{u}_{k'} - u^*_{k'})}_{({\text{vii}})},
\end{eqnarray*}
where $\cH^*({\bu^*})$ is the expectation of $\cH({\bu^*})$ over the comparison outcomes. In the following, we will bound (vi) and (vii) separately. 

\smallskip

\textbf{\underline{Step I}: Bound on (vi)}

\smallskip

We begin by noting 

\begin{align}
	&\sum_{k' \neq k} [\cH({\bu^*}) - \cH^*(\bu^*)]_{kk'}\label{cuties}\\
	=&\ \sum_{k' \neq k}\sum_{i: \{k, k'\}\subseteq T_i}\Bigg\{\sum_{j\in [r_i(k)\wedge r_i(k')\wedge y_i]}\frac{\e{u_k}\e{u_{k'}}}{(\sum_{t \geq j}\e{u_{\pi_i(t)}})^2} \nonumber \\
	& \quad\quad\quad\quad\quad\quad\quad\quad\quad\quad - \mathbb E\bigg[\sum_{j\in [r_i(k)\wedge r_i(k')\wedge y_i]}\frac{\e{u_k}\e{u_{k'}}}{(\sum_{t \geq j}\e{u_{\pi_i(t)}})^2}\bigg]\Bigg\}\nonumber\\
	=&\ \sum_{i: k\in T_i}\Bigg\{\sum_{k'\in T_i, k'\neq k}\sum_{j\in [r_i(k)\wedge r_i(k')\wedge y_i]}\frac{\e{u_k}\e{u_{k'}}}{(\sum_{t \geq j}\e{u_{\pi_i(t)}})^2}  \nonumber\\
&	\quad\quad\quad\quad\quad\quad\quad\quad\quad\quad - \mathbb E\bigg[\sum_{k'\in T_i, k'\neq k}\sum_{j\in [r_i(k)\wedge r_i(k')\wedge y_i]}\frac{\e{u_k}\e{u_{k'}}}{(\sum_{t \geq j}\e{u_{\pi_i(t)}})^2}\bigg]\Bigg\}\nonumber
\end{align}
is a sum of $\ndeg_k$ uniformly bounded independent random variables with mean zero under Assumptions~\plainref{ap:2}-\plainref{ap:1}.  
By Hoeffding's inequality, with probability at least $1 - n^{-3}$,
\begin{align}\label{0722}
|({\text{vi}})| = \left|(\widehat{u}_k - u^*_k) \sum_{k' \neq k} [\cH({\bu^*}) - \cH^*(\bu^*)]_{kk'}\right| \lesssim  \lVert  \buu  - \bu^*\lVert_\infty\sqrt{\ndeg_k\log n}.
\end{align}
\smallskip

\textbf{\underline{Step II}: Bound on (vii)}

\smallskip

One may wish to apply a similar argument to bound (vii). Unfortunately, $({\text{vii}})$ is not a sum of independent random variables because the summation index $k'$ also appears in $(\widehat{u}_{k'} - u^*_{k'})$. To address this, we resort to a perturbation argument similar to the leave-one-out analysis in \cite{gao2023uncertainty}. Recall the following leave-one-out log-likelihood function defined in \plaineqref{leave_one_out_likelihood}:
\begin{align*}
	&l_1^{(-k)}(\bm u) = \sum_{i: k \notin T_i} \sum_{j\in [y_i]}\left[u_{\pi_i(j)} - \log\left(\sum_{l=j}^{m_i}\e{u_{\pi_i(l)}}\right)\right]& k \in [n].
\end{align*}
To find the (constrained) marginal MLE of the leave-one-out log-likelihood, for each $k\in [n]$, we introduce the following constrained set:
\begin{align*}
\mathcal{S}_k = \left\{{\bu \in \mathbb{R}^{n-1}}: \sum_{j=1}^{n-1} u_j = \sum_{j \neq k } u_j^* = -u^*_k, \ \lVert \bu - \bu^*_{-k}  \lVert_\infty \leq 1  \right \},
\end{align*}
where $\bu^*_{-k}$ is obtained by removing the $k$th component of $\bu^*$, and consider the constrained leave-one-out marginal MLE:
\begin{align*}
\widehat{\bu}^{(-k)}\in \argmax_{\bu\in\mathcal S_k}l_1^{(-k)}(\bm u). 
\end{align*}

Since $\mathcal S_k$ is compact and $l_1^{(-k)}(\cdot)$ is continuous, $\widehat{\bu}^{(-k)}$ exists (uniqueness is not needed). Serving as a bridge, the following estimate of $\|\widehat{\bu}^{(-k)} -\widehat{\bu}_{-k}\|_2$ will be needed in our proof. 
\begin{Lemma}\label{lm:tired}             
	Under Assumptions \plainref{ap:2}, \plainref{ap:1} and \plainref{ap:deterministic}, with probability at least $1-2n^{-3}$, for all $k\in [n]$, 
		\begin{align*}
\left\|\widehat{\bu}^{(-k)} - \widehat{\bu}_{-k}  \right \|_2 \lesssim  \frac{\Gamma_n^{\RE} \N_{n, +} + \sqrt{\N_{n, +}\log n}}{\lambda_2^\leave}. 
	\end{align*}                                                                                                                       
\end{Lemma}
The proof of Lemma~\ref{lm:tired} is deferred to the end of this section. We are now ready to bound $(\text{vii})$. Inserting $\widehat{u}^{(-k)}_{k'}$ in the summand, 
\begin{eqnarray*}
	(\text{vii}) = \underbrace{\sum_{k' \neq k} [\cH({\bu^*}) - \cH^*(\bu^*)]_{kk'} (\widehat{u}^{(-k)}_{k'} - u^*_{k'})}_{(\text{vii.1})} + \underbrace{\sum_{k' \neq k} [\cH({\bu^*}) - \cH^*(\bu^*)]_{kk'} (\widehat{u}_{k'} - \widehat{u}^{(-k)}_{k'} )}_{(\text{vii.2})}. 
\end{eqnarray*}
For (vii.1), note that $[\cH({\bu^*}) - \cH^*(\bu^*)]_{kk'}$ depend only on the comparison outcomes $\{\pi_i\}_{i: k\in T_i}$ while $(\widehat{u}^{(-k)}_{k'} - u^*_{k'})$ depend only on $\{\pi_i\}_{i: k\notin T_i}$, so the two terms are independent. Thus, conditioning on $\{\pi_i\}_{i: k \notin T_i}$, $(\widehat{u}^{(-k)}_{k'} - u^*_{k'})$ is nonrandom. Therefore, ${(\text{vii.1})}$ could be further written as a sum of bounded independent random variables. Applying Hoeffding's inequality and a union bound over $k\in [n]$, with probability at least $1 - n^{-3}$,
\begin{align*}
	&|{(\text{vii.1})}| \lesssim  \|\widehat{\bu}^{(-k)} -\bu^*_{-k}\|_\infty \sqrt{\ndeg_k\log n}& k \in [n].
\end{align*}
Meanwhile, by the Cauchy--Schwarz inequality and another application of Hoeffding's inequality plus union bound, with probability at least $1-n^{-3}$,  
\begin{align*}
&|{(\text{vii.2})}| \leq \| [\cH({\bu^*}) - \cH^*(\bu^*)]_{k\cdot} \|_2 \|\widehat{\bu}^{(-k)} -\widehat{\bu}_{-k}\|_2\lesssim \sqrt{\ndeg_k\log n}\cdot \|\widehat{\bu}^{(-k)} -\widehat{\bu}_{-k}\|_2 & k \in [n],
\end{align*}
where Hoeffding's inequality is applied to control $|[\cH({\bu^*}) - \cH^*(\bu^*)]_{kk'}|\lesssim\sqrt{\N_{kk'}\log n}$ for all $k'\in [n]\setminus \{k\}$; see \eqref{cuties} for the expression of $[\cH({\bu^*}) - \cH^*(\bu^*)]_{kk'}$. 
As a result, with probability at least $1-2n^{-3}$, for all $k\in [n]$, 
\begin{align}\label{0723}
|{(\text{vii})}|  \leq |{(\text{vii.1})}|  + |{(\text{vii.2})}|\lesssim (\|\widehat{\bu}^{(-k)} -\bu^*_{-k}\|_\infty + \|\widehat{\bu}^{(-k)} -\widehat{\bu}_{-k}\|_2) \sqrt{\ndeg_k\log n}. 
\end{align}
Putting the estimates in \eqref{0722}, \eqref{0723}, and Lemma~\ref{lm:tired} together via a union bound yields that, with probability at least $1-n^{-2}$, for all $k\in [n]$, 
\begin{align}
	&|[\{\cH({\bu^*}) -\cH^*(\bu^*) \}(\buu - \bm u^*)]_k|\nonumber\\
	 \leq&\  |{(\text{vi})}|  + |{(\text{vii})}|\nonumber\\
	 \lesssim&\ \left(\lVert  \buu  - \bu^*\lVert_\infty+ \|\widehat{\bu}^{(-k)} -\bu^*_{-k}\|_\infty + \|\widehat{\bu}^{(-k)} -\widehat{\bu}_{-k}\|_2 \right)\sqrt{\ndeg_k\log n}\nonumber\\
	\leq&\ \left[\lVert  \buu  - \bu^*\lVert_\infty + (\|\widehat{\bu}^{(-k)} -\buu_{-k}\|_\infty + \|\buu_{-k} -\bu^*_{-k}\|_\infty) + \|\widehat{\bu}^{(-k)} -\widehat{\bu}_{-k}\|_2\right] \sqrt{\ndeg_k\log n}\nonumber\\
	\lesssim\ & \left(\lVert  \buu  - \bu^*\lVert_\infty+ \|\widehat{\bu}^{(-k)} -\widehat{\bu}_{-k}\|_2\right)\sqrt{\ndeg_k\log n} \nonumber\\
	\lesssim\ & \left(\Gamma^\RE_n +  \frac{\Gamma_n^{\RE} \N_{n, +} + \sqrt{\N_{n, +}\log n}}{\lambda_2^\leave}\right)\sqrt{\ndeg_k\log n}\nonumber\\
	\lesssim\ & \frac{\Gamma_n^{\RE} \N_{n, +} + \sqrt{\N_{n, +}\log n}}{\lambda_2^\leave}\sqrt{\ndeg_k\log n},\label{dinglu}
\end{align}
To see why the last step holds, fixing $k\in [n]$ and denoting by $\D^{(-k)}$ the degree matrix of the weighted undirected graph associated with $-{\cH^*}^{(-k)}(\bu^*)$,  
\begin{align}
\frac{\N_{n, +}}{\lambda_2^\leave} \nonumber&\geq \frac{\N_{n, +}}{\lambda_2(-{\cH^*}^{(-k)}(\bu^*))} \\
&\geq \frac{\N_{n, +}}{\lambda_2(\{\mathcal D^{(-k)}\}^{1/2}\{\mathcal D^{(-k)}\}^{-1/2}\{-{\cH^*}^{(-k)}(\bu^*)\}\{\mathcal D^{(-k)}\}^{-1/2}\{\mathcal D^{(-k)}\}^{1/2})}\nonumber\\
&\geq \frac{\N_{n, +}}{\max_{j\in [n-1]}[\D^{(-k)}]_{jj}\lambda_2(\{\mathcal D^{(-k)}\}^{-1/2}\{-{\cH^*}^{(-k)}(\bu^*)\}\{\mathcal D^{(-k)}\}^{-1/2})}\nonumber\\
&\geq \frac{\N_{n, +}}{\max_{j\in [n-1]}[\D^{(-k)}]_{jj}}\nonumber\\
&\gtrsim 1,\label{zhaoran}
\end{align}
where the penultimate step used $\lambda_2(\{\mathcal D^{(-k)}\}^{-1/2}\{-{\cH^*}^{(-k)}(\bu^*)\}\{\mathcal D^{(-k)}\}^{-1/2})\leq 1$, and the last step holds under Assumption~\plainref{ap:deterministic} and follows as a consequence of \eqref{louis} in Section~\ref{new19}. 
Hence, 
\begin{align*}
|[\D^{-1/2}\{\cH({\bu^*}) -\cH^*(\bu^*) \}(\buu - \bm u^*)]_k|\lesssim\frac{\Gamma_n^{\RE} \N_{n, +}\sqrt{\log n} + \sqrt{\N_{n, +}}\log n}{\lambda_2^\leave}.
\end{align*}
\subsubsection{Proof of Lemma~\ref{lm:tired}}
To finish the proof, it remains to prove Lemma~\ref{lm:tired}.
\begin{proof}[Proof of Lemma~\ref{lm:tired}]
Note that $l_1^{(-k)}(\bm u)$ is a concave function and the domain $\mathcal S_k$ is convex. By the optimality condition for convex optimization \citep{boyd2004convex}, the solution $\widehat{\bu}^{(-k)}$ satisfies  
\begin{equation}\label{pangeru+}
(\widehat{\bu}^{(-k)}-\bu)^\top \{\nabla l_1^{(-k)}(\widehat{\bu}^{(-k)})\} \geq 0\quad\quad\text{for all $\bu \in \mathcal{S}_k$}.
\end{equation}
Take $\bu = \widehat{\bu}_{-k}+ (n-1)^{-1}(\widehat{u}_k -  u^*_{k})\langle 1\rangle$. 
Note that this choice is valid since $\langle 1\rangle^\top \bu = \langle 1\rangle^\top\buu - u^*_k = -u_k^*$. In addition, according to Theorem~\plainref{main:general}, with probability at least $1-n^{-3}$, 
\begin{align*}
\|\bu - \bu^*_{-k}\|_\infty\leq \|\buu_{-k} - \bu^*_{-k}\|_\infty + \frac{1}{n-1}|\widehat{u}_k-u^*_k|\leq 2\|\buu - \bu^*\|_\infty\lesssim\Gamma_n^\RE = o(1),
\end{align*}
where the last step is implied by Assumption~\plainref{ap:deterministic}. Hence, $\bu\in \mathcal S_k$. 

Therefore, it follows from the mean-value theorem that
\begin{align*}
&(\widehat{\bu}^{(-k)} - {\bu})^\top  \left\{\nabla l_1^{(-k)}({\bu})\right\}\\
&\stackrel{\eqref{pangeru+}}{\geq} (\widehat{\bu}^{(-k)} - {\bu})^\top  \left\{\nabla l_1^{(-k)}({\bu})\right\} -  (\widehat{\bu}^{(-k)}-\bu)^\top \{\nabla l_1^{(-k)}(\widehat{\bu}^{(-k)})\} \\
	&= 	-(\widehat{\bu}^{(-k)} - {\bu})^\top       \left\{\nabla l_1^{(-k)}(\widehat{\bu}^{(-k)}) -   \nabla l_1^{(-k)}({\bu})\right\}  \\
	&= 	(\widehat{\bu}^{(-k)} - {\bu})^\top       \left\{-\cH^{(k)}(\bar{\bu}^{(-k)})\right\}   (\widehat{\bu}^{(-k)} - {\bu}) \geq 0,
\end{align*}
where $\cH^{(k)}(\bar{\bu}^{(-k)})=\nabla^2 l_1^{(-k)}(\bar{\bu}^{(-k)})$ and $\bar{\bu}^{(-k)}$ lies on the line segment between $\widehat{\bu}^{(-k)}$ and ${\bu}$. By the Cauchy--Schwarz inequality,  
\begin{align}
\|\widehat{\bu}^{(-k)} - {\bu}\|_2 \leq \frac{\|  \nabla l_1^{(-k)}({\bu})\|_2}{\lambda_2(-\cH^{(-k)}(\bar{\bu}^{(-k)}))}\lesssim \frac{\|  \nabla l_1^{(-k)}({\bu})\|_2}{\lambda_2(-\cH^{(-k)}(\bu^*_{-k}))}\stackrel{\plaineqref{myleave}}{\leq} \frac{\|  \nabla l_1^{(-k)}({\bu})\|_2}{\lambda_2^\leave}.\label{jjuhy}
\end{align}
For the second step, thanks to Theorem~\plainref{main:general}, with probability at least $1-n^{-3}$,   
\begin{align*}
\|\bar{\bu}^{(-k)}\|_\infty\leq\max\{\|\bm u\|_\infty, \|\buu^{(-k)}\|_\infty\}\leq 2(\|\bu^*\|_\infty + 1)<\infty. 
\end{align*}
Under Assumption~\plainref{ap:2}, both $\bar{\bu}^{(-k)}$ and $\bu^*$ are uniformly bounded in $\|\cdot\|_\infty$. It follows from \plaineqref{wearediff} (that is, the summand for each edge $T_i$ is uniformly bounded for all possible comparison outcomes $r_i$) that for $s\neq s'\in [n]\setminus\{k\}$, 
\begin{align}
[\cH^{(-k)}(\bar{\bu}^{(-k)})]_{ss'}\asymp [\cH^{(-k)}(\bu^*_{-k})]_{ss'}\asymp [{\cH^*}^{(-k)}(\bu^*_{-k})]_{ss'}\asymp |\{i: \{s, s'\}\subseteq T_i, k\notin T_i\}|.\label{mazda1}
\end{align}
Therefore, for any $\bm x\in\mathbb R^{n-1}$, 
\begin{align}
\bm x^\top [-\cH^{(-k)}(\bar{\bu}^{(-k)})]\bm x &= \frac{1}{2}\sum_{s, s'\in [n]\setminus\{k\}}\cH^{(-k)}(\bar{\bu}^{(-k)})(x_s-x_{s'})^2\label{mazda2}\\
&\stackrel{\eqref{mazda1}}{\asymp} \frac{1}{2}\sum_{s, s'\in [n]\setminus\{k\}}{\cH^*}^{(-k)}(\bu^*_{-k})(x_s-x_{s'})^2\nonumber\\
& = \bm x^\top [-{\cH^*}^{(-k)}(\bu^*_{-k})]\bm x. \nonumber
\end{align}
By the Courant--Fischer theorem, $-\cH^{(-k)}(\bar{\bu}^{(-k)})$ and $-{\cH^*}^{(-k)}(\bu^*_{-k})$ have the same zero-eigenspace; moreover, $\lambda_2(-\cH^{(-k)}(\bar{\bu}^{(-k)}))\asymp\lambda_2(-{\cH^*}^{(-k)}(\bu^*_{-k}))\geq\lambda_2^\leave$. This justifies the second step in \eqref{jjuhy}. 

To finish the proof, it remains to bound the numerator in the upper bound in \eqref{jjuhy}.
Since $\nabla l_1(\widehat{\bu}) = 0$, separating the edges involving $k$ and the rest apart and noting that log-likelihood summand from edges not containing $k$ does not depend the $k$th component,  
\begin{align}
\|\nabla l_1^{(-k)}({\widehat{\bu}_{-k}})\|_2^2 = \sum_{k' \in [n], k' \neq k} \left\{\sum_{i: \{k, k'\}\subseteq T_i} \psi(k'; T_i, \pi_i, \widehat{\bu})\right\}^2,\label{localt}
\end{align}
where $\psi$ is defined in \eqref{yourpsi} and we write it down for the reader's convenience:
\begin{align*}
\psi(k'; T_i, \pi_i, \widehat{\bu}) =  \mathbf 1_{\{r_i(k')\leq y_i\}} - \sum_{j\in [r_i(k')\wedge y_i]}\frac{\e{\widehat{u}_{k'} }}{\sum_{t=j}^{m_i} \e{\widehat{u}_{\pi_i(t)}}}.
\end{align*} 

The right-hand side of \eqref{localt} is a local term which can be estimated as follows:
\begin{align*}
&\sum_{k' \in [n], k' \neq k} \bigg\{\sum_{i: \{k, k'\}\subseteq T_i} \psi(k'; T_i, \pi_i, \widehat{\bu})\bigg\}^2\\
& \leq 2  \sum_{k' \in [n], k' \neq k} \Bigg[\bigg\{ \sum_{i: \{k, k'\}\subseteq T_i} \Big(\psi(k'; T_i, \pi_i, \widehat{\bu}) -  \psi(k'; T_i, \pi_i, {\bu^*})\Big)\bigg\}^2 \\
& \quad\quad\quad\quad\quad\quad\quad\quad\quad\quad\quad\quad\quad\quad\quad\quad+ \bigg\{  \sum_{i: \{k, k'\}\subseteq T_i} \psi(k'; T_i, \pi_i, {\bu^*})\bigg\}^2 \Bigg].
\end{align*}
For the first term, note that $\psi(k'; T_i, \pi_i, {\bu})$ is bounded by $M$ under Assumption \plainref{ap:1} and is Lipschitz continuous with respect to $\bm u$. As a result, 
\begin{align*}
	|\psi(k'; T_i, \pi_i, \widehat{\bu}) - \psi(k'; T_i, \pi_i, {\bu^*})| \lesssim \lVert \widehat{\bu} - {\bu^*} \lVert_\infty \lesssim \Gamma_n^{\RE}.
\end{align*}
Meanwhile, since $\mathbb E[\psi(k'; T_i, \pi_i, {\bu^*})]=0$, denoting $\N_{kk'} = |\{ i: \{k, k'\}\subseteq T_i\}|$ for $k' \neq k$ and applying Hoeffding's inequality, 
\begin{align*}
\mathbb{P}	\Bigg(  \Big|\sum_{i: \{k, k'\}\subseteq T_i} \psi(k'; T_i, \pi_i, {\bu^*})\Big| \leq \sqrt{24M\N_{kk'}\log n}  \Bigg) \geq 1 - n^{-5}.
\end{align*}
Taking a union bound over $k'\in [n]\setminus\{k\}$ yields that, with probability at least $1-n^{-4}$, 
\begin{align}
\max_{k'\in [n]\setminus\{k\}}\Big|\sum_{i: \{k, k'\}\subseteq T_i} \psi(k'; T_i, \pi_i, {\bu^*})\Big| \leq \sqrt{24M\N_{kk'}\log n}.\label{aftern}
\end{align}
Consequently,
\begin{align*}
\|\nabla l_1^{(-k)}({\widehat{\bu}_{-k}})\|_2^2 \lesssim \sum_{k' \in [n], k' \neq k} \{(\N_{kk'}\Gamma_n^{\RE})^2 + \N_{kk'} \log n\}\lesssim (\Gamma_n^{\RE} \N_{n, +})^2 + \N_{n, +} \log n.
\end{align*}
Recall that $\bu = \widehat{\bu}_{-k}+ (n-1)^{-1}(\widehat{u}_k -  u^*_{k})\langle 1\rangle$. As a result, 
\begin{equation}\label{l2_error_leave_one_out}
\left\|\widehat{\bu}^{(-k)} - \widehat{\bu}_{-k} - \frac{1}{n-1}(\widehat{u}_k -  u^*_{k})\langle 1\rangle\right \|_2 \lesssim  \frac{\Gamma_n^{\RE} \N_{n, +} + \sqrt{\N_{n, +}\log n}}{\lambda_2^\leave}. \end{equation}
with probability at least $1-n^{-3}$. Since $\|(n-1)^{-1}(\widehat{u}_k -  u^*_{k})\langle 1\rangle \|_2 \lesssim n^{-1/2}\Gamma_n^\RE = o(\N_{n,+}\Gamma_n^\RE/\lambda_2^\leave)$ as a result of \eqref{zhaoran}, we conclude that 
 \begin{equation}\label{l2_error_leave_one_out}
\left\|\widehat{\bu}^{(-k)} - \widehat{\bu}_{-k} \right \|_2 \lesssim \frac{\Gamma_n^{\RE} \N_{n, +} + \sqrt{\N_{n, +}\log n}}{\lambda_2^\leave} . \end{equation}
with probability at least $1-n^{-3}-n^{-4}$. We would like to point out this probability comes from Theorem \plainref{main:general} and \eqref{aftern}. Taking a union bound over $k\in [n]$ in \eqref{aftern}, we finish the proof of Lemma \ref{lm:tired}.
\end{proof}

\subsection{Proof of Lemma \ref{CLT:1}}\label{laile3}

We first prove the result for the QMLE since it is more involved. For any fixed $k\in[n]$,
\begin{equation}\label{CTL: decomposition_1}
	({\mathcal{L}^\dagger_{\sym}} \D^{-1/2} \nabla l_2(\bm u^*))_k = (\D^{-1/2} \nabla l_2(\bm u^*))_k
	+ (({\mathcal{L}^\dagger_{\sym}- \I} )\D^{-1/2} \nabla l_2(\bm u^*))_k.
\end{equation}
For the first term, note that each component of $\nabla l_2(\bm u^*)$ is a sum of uniformly bounded and mean-zero independent random variables indexed by edges (see \eqref{eq:231}):
\begin{align}
(\nabla l_2(\bu^*))_k = \sum_{i: k\in T_i}\varphi(k, T_i, \pi_i, \bu^*),\label{letgo}
\end{align}
where $\varphi(k, T_i, \pi_i, \bu^*)$ are random variables defined in \eqref{yourphi} with comparable second moments: $\mathbb E[\varphi^2(k, T_i, \pi_i, \bu^*)]\asymp 1$ uniformly for all $i\in [N]$ under Assumptions~\plainref{ap:2}-\plainref{ap:1}. 
Since $\D^{-1/2} $ is a diagonal matrix with $\min_{j\in [n]}\D_{jj}\gtrsim \N_{n,-}\to\infty$ under Assumption~\plainref{ap:deterministic}, by the Lindeberg--Feller central limit theorem,  
\begin{align*}
	& (\D^{-1/2} \nabla l_2(\bm u^*))_k \to N(0, \Sigma_{kk} ) & \Sigma := \D^{-1/2} \mathbb{E}[\nabla l_2(\bm u^*) \nabla l_2(\bm u^*)^\top] \D^{-1/2}
\end{align*}
in distribution as $n \to \infty$. Thus, it remains to show that $(({\mathcal{L}^\dagger_{\sym}-\I} )\D^{-1/2} \nabla l_2(\bm u^*))_k = o_p(1)$, for which we use Chebyshev's inequality.
Since $\mathbb{E}[(({\mathcal{L}^\dagger_{\sym}-\I} )\D^{-1/2} \nabla l_2(\bm u^*))_k] =0$, it suffices to compute the second moment and verify that any fixed component of it converges to zero.  

Rewriting the second term on the right-hand side in \eqref{CTL: decomposition_1} using the Neumann series expansion of $(\mathcal{L}^\dagger_{\sym}-\I)$ as a truncated sum, 
\begin{align}
	&(({\mathcal{L}^\dagger_{\sym}-\I} )\D^{-1/2} \nabla l_2(\bm u^*))_k\nonumber\\
	=& \left(\sum_{t=1}^{t_n-1 } (\A - \cP_1)^t \D^{-1/2} \nabla l_2(\bm u^*)\right)_k + \left(	\sum_{t=t_n}^{\infty}(\A - \cP_1)^t\D^{-1/2} \nabla l_2(\bm u^*)\right)_k\label{k890}
\end{align}
where $t_n$ is the same constant in \eqref{mytn} such that $\left\|\sum_{t = t_n}^{\infty} (\A- \cP_1)^t\right\|_2\lesssim n^{-2M}$. As a result,
\begin{align}
	\left(	\sum_{t=t_n}^{\infty}(\A - \cP_1)^t\D^{-1/2} \nabla l_2(\bm u^*)\right)_k &\leq \left(\sum_{t=t_n}^\infty\|\A - \cP_1\|^t_2\right)\|\D^{-1/2} \nabla l_2(\bm u^*)\|_2\label{56324}\\
	&\lesssim n^{-2M}\|\D^{-1}\|_2\|\nabla l_2(\bm u^*)\|_2 = o_p(1),\nonumber
\end{align} 
because $\|\D^{-1}\|_2\lesssim \N_{n,-}^{-1}\lesssim 1$ and each component of $\nabla l_2(\bm u^*)$ consists of a summation of at most $O(n^{M})$ uniformly bounded random variables so that $\|\nabla l_2(\bm u^*)\|_2\lesssim n^{M+1/2}$. 
For the first term in \eqref{k890}, by a similar argument to obtain \eqref{tobeu}, $\mathcal P_1\D^{-1/2}\nabla l_2(\bu^*) = \bm 0$. This combined with the commutativity between $\A$ and $\mathcal P_1$ yields 
\begin{align*}
	\sum_{t=1}^{t_n-1 } (\A - \cP_1)^t \D^{-1/2} \nabla l_2(\bm u^*) = \sum_{t=1}^{t_n-1 } \A^t \D^{-1/2} \nabla l_2(\bm u^*),
\end{align*}
where each term has mean zero and variance
\begin{align*}
	&\mathbb{E}\left[\left\{(\A^t \D^{-1/2} \nabla l_2(\bm u^*))_k\right\}^2\right]
	=\ e_k^\top\A^t\D^{-1/2}  \mathbb{E}[\{\nabla l_2(\bm u^*)\}\{\nabla l_2(\bm u^*)\}^\top] \D^{-1/2}\A^te_{k}.
\end{align*}
Recall that $\{\nabla l_2(\bu^*)\}_k = \sum_{i: k\in T_i}\varphi(k, T_i, \pi_i, \bu^*)$, where $\varphi(k, T_i, \pi_i, \bu^*)$ is defined in \eqref{yourphi}. We compute the $k'k''$th entry of $\mathbb{E}[\{\nabla l_2(\bm u^*)\}\{\nabla l_2(\bm u^*)\}^\top]$ for $k', k'' \in [n]$ as follows. When $k' = k''$, 
\begin{align*}
(	\mathbb{E}[\{\nabla l_2(\bm u^*)\}\{\nabla l_2(\bm u^*)\}^\top])_{k'k'} = \sum_{i: k'\in T_i} \text{Var}(\varphi(k', T_i, \pi_i, \bu^*)) \lesssim \D_{k'k'}.
\end{align*}
When $k' \neq k''$, 
\begin{align*}
	(	\mathbb{E}[\{\nabla l_2(\bm u^*)\}\{\nabla l_2(\bm u^*)\}^\top])_{k'k''} = \sum_{{i: \{k', k''\}\subseteq T_i}} \mathbb{E}\{\varphi(k', T_i, \pi_i, \bu^*) \times \varphi(k'', T_i, \pi_i, \bu^*) \} \lesssim \cH_{k'k''}.
\end{align*}
As a result, 
\begin{align}\label{QMLE_special}
	(\D^{-1/2}  \mathbb{E}[\{\nabla l_2(\bm u^*)\}\{\nabla l_2(\bm u^*)\}^\top] \D^{-1/2})_{k'k''}&\lesssim (\I+\A)_{k'k''}.
\end{align}
Since $\mathcal A^t$ has nonnegative entries for every $t> 0$, 
\begin{align*}
	\mathbb{E}\left[\left\{(\A^t \D^{-1/2} \nabla l_2(\bm u^*))_k\right\}^2\right] 
	&\lesssim e_k^\top\A^t(\I+\A)\A^te_{k} \lesssim \cc,
\end{align*}
where the last step follows from a similar computation as the proof of Lemma \ref{CLT:2}. Specifically, for any $t\geq 0$, we have
\begin{eqnarray*}
		 [\A^{t+1}]_{kk}  &= &  \sum_{j_1 \in [n]} \cdots \sum_{j_t \in [n]} \frac{\h_{kj_1}\h_{j_1j_2} \cdots \h_{j_{t-1}j_t}\h_{j_tk} }{{\D_{kk}} \D_{j_1j_1}\D_{j_2j_2} \cdots \D_{j_{t-1}j_{t-1}} {\D_{j_tj_t}}}\\
		&{\leq}&   \left\{\max_{j_t \in [n]} \frac{\h_{j_tk}}{\D_{j_tj_t}}\right\} \sum_{j_1 \in [n]} \cdots \sum_{j_t \in [n]} \frac{\h_{kj_1}\h_{j_1j_2} \cdots \h_{j_{t-1}j_t} }{{\D_{kk}} \D_{j_1j_1} \D_{j_2j_2} \cdots \D_{j_{t-1}j_{t-1}}}\\
		& = &  \left\{\max_{j_t \in [n]} \frac{\h_{j_tk}}{\D_{j_tj_t}}\right \} \sum_{j_1 \in [n]} \cdots \sum_{j_{t-1} \in [n]}\frac{\h_{kj_1}\h_{j_1j_2} \cdots \h_{j_{t-2}j_{t-1}} }{{\D_{kk}} \D_{j_1j_1} \D_{j_2j_2} \cdots \D_{j_{t-2}j_{t-2}}}\\
		& = &  \max_{j_t \in [n]} \frac{\h_{j_tk}}{\D_{j_tj_t}} \lesssim \cc.
\end{eqnarray*}
Thus, the first term in \eqref{k890} can be bounded using the Cauchy--Schwarz inequality as
\begin{align*}
	\text{Var}\left[\sum_{t=1}^{t_n-1 } \left((\A - \cP_1)^t \D^{-1/2} \nabla l_2(\bm u^*)\right)_k\right]&\leq t^2_n\max_{1\leq t\leq t_n-1}\text{Var}\left[((\A - \cP_1)^t \D^{-1/2} \nabla l_2(\bm u^*))_k\right]\\
	&\lesssim \frac{\cc (\log n)^2  }{\mix^2}\\
	& = o(1).\quad\quad\quad(\text{Assumption~\plainref{ap:deterministic}})
\end{align*}
Applying Chebyshev's inequality, we have
\begin{align*}
	\sum_{t=1}^{t_n-1 } (\A - \cP_1)^t \D^{-1/2} \nabla l_2(\bm u^*) = o_p(1). 
\end{align*}
This combined with the estimate in \eqref{56324} yields the desired result. 

The proof of the marginal MLE is almost identical except that \eqref{QMLE_special} can be explicitly computed as
\begin{align*}
	&(\D^{-1/2}  \mathbb{E}[\{\nabla l_1(\bm u^*)\}\{\nabla l_1(\bm u^*)\}^\top] \D^{-1/2})_{k'k''} =  (\I-\mathcal A)_{k'k''}\leq (\I+\A)_{k'k''} &k', k''\in [n]
\end{align*}
due to $\mathbb{E}[\{\nabla l_1(\bm u^*)\}\{\nabla l_1(\bm u^*)\}^\top] = \mathbb E[-\nabla^2 l_1(\bu^*)] = -\cH^*(\bm u^*)$. We finish the proof of Lemma \ref{CLT:1}.

\subsection{Proof of Lemma~\plainref{newtech}}\label{laile4}
It suffices to prove \plaineqref{2024222} and \plaineqref{2024333} only since \plaineqref{2024111} is an immediate consequence of degree concentration. 
We first recall the notation $\Omega^{(\ms)} = \{ T\subseteq [n]:   |T| = \ms \}$. Moreover, we define $\Omega^{(\ms)}_i = \{ T\subseteq [n]:  i\in T, |T| = \ms \}$ and  $\Omega^{(\ms)}_{ij} = \{ T\subseteq [n]:  \{i, j\}\in T, |T| = \ms \}$.

Recall the eigenvalues of ${\mathcal{L}_{\sym} = \I - \A  = \I - \D^{-1/2}\W\D^{-1/2}}$ as $0 = \lambda_1(\mathcal L_\sym)\leq\cdots\leq\lambda_n(\mathcal L_\sym)\leq 2$. Let $\bar{\D}$ and $\bar{\W}$ denote the expectation of $\D$ and $\W$, respectively, and $\bar{\mathcal{L}}_{\sym} = \I - \bar{\D}^{-1/2}\bar{\W}\bar{\D}^{-1/2}$. The only randomness here comes from the comparison graph sampling. We denote the eigenvalues of $\bar{\mathcal{L}}_{\sym}$ by $0=\lambda_1(\bar{\mathcal{L}}_{\sym})\leq\cdots\leq \lambda_n(\bar{\mathcal{L}}_{\sym})\leq 2$. Since both $\mathcal{L}_{\sym}$ and $\bar{\mathcal{L}}_{\sym}$ are symmetric, by Weyl's inequality, $\max_{i\in [n]}|\lambda_i({\mathcal{L}}_{\sym})-\lambda_i(\bar{\mathcal{L}}_{\sym})|\leq \|\mathcal{L}_{\sym}-\bar{\mathcal{L}}_{\sym}\|_2$. Consequently,
\begin{align}
\mix\geq \bar{\mix} - \|\mathcal{L}_{\sym}-\bar{\mathcal{L}}_{\sym}\|_2, \quad\quad \bar{\mix}:=\min\{\lambda_2(\bar{\mathcal{L}}_{\sym}), 2-\lambda_n(\bar{\mathcal{L}}_{\sym})\}.\label{lexing}
\end{align}

To prove \plaineqref{2024222}, we bound $\|\mathcal{L}_{\sym}-\bar{\mathcal{L}}_{\sym}\|_2$ from above and $\bar{\mix}$ from below, respectively. We begin by introducing the definitions of the Cheeger constant and its dual version in the spectral graph literature; readers familiar with the relevant materials may skip Section \ref{guihuazi}. 

\subsubsection{Review of Cheeger Constants}\label{guihuazi}
For a weighted undirected graph $\GG$ (not a hypergraph) with vertex set $[n]$ and nonnegative symmetric edge weights $\{\chi_{ij}\}_{i, j\in [n]}$ ($\chi_{ij}=\chi_{ji}$), the Cheeger constant of $\GG$ is defined as 	
\begin{align}
&g_{\GG}:= \min_{U\subset [n]}\frac{\sum_{i\in U, j\in U^\complement}\chi_{ij}}{\min\{\text{vol}(U), \text{vol}(U^\complement)\}}&\vol(U) = \sum_{i\in U}\sum_{j\in [n]}\chi_{ij}.\label{fc:cheeger}
\end{align}

This definition is more often used in the literature to characterize the spectral gap of a graph \cite[Lemmas 2.1-2.2]{MR1421568} and is different from the modified Cheeger constant $h_{\GG}$ where normalization in the denominator counts vertices rather than edges. 
\begin{Lemma}[Cheeger's inequality]\label{chg1}
Let $\lambda_2(\GG)$ be the second smallest eigenvalue of the normalized graph Laplacian matrix of the weighted undirected graph $\GG$, and $g_{\GG}$ be the Cheeger constant defined in \eqref{fc:cheeger}. Then,
\begin{align}
\frac{g_\GG^2}{2}<\lambda_2(\GG)\leq 2g_{\GG}.\label{yonghuol}
\end{align}
\end{Lemma}

We also need a dual version of the Cheeger constant defined in \eqref{fc:cheeger}:
\begin{align}
g^{\ddd}_{\mathsf G} = \max_{U_1, U_2\subset [n], U_1\cap U_2 = \emptyset}\frac{2\sum_{i\in U_1, j\in U_2}\chi_{ij}}{\text{vol}(U_1)+\text{vol}(U_2)},\label{dualcheeger}
\end{align}
This definition was introduced in \cite{bauer_jost_2013} to study the largest eigenvalue of the weighted graph Laplacian of $\GG$. 
A similar inequality for the dual Cheeger's constant is proved in \cite[Theorem 3.2]{bauer_jost_2013}.
\begin{Lemma}[Dual Cheeger's inequality]\label{chg2}
Let $\lambda_n(\GG)$ be the largest eigenvalue of the normalized graph Laplacian matrix of the weighted undirected graph $\GG$, and $g^{\ddd}_{\GG}$ be the dual Cheeger constant defined in \eqref{dualcheeger}. Then,
\begin{align}
\lambda_n(\GG)\leq 1 + \sqrt{1-(1-g^{\ddd}_{\mathsf G})^2}.\label{kouyi}
\end{align}
\end{Lemma}
Both Lemmas~\ref{chg1} and \ref{chg2} will be used in Section~\ref{ksjsjsjsk} to obtain lower bound on $\min\{\lambda_2(\bar{\mathcal{L}}_{\sym}), 2-\lambda_n(\bar{\mathcal{L}}_{\sym})\}$. 

\subsubsection{Proof of \plaineqref{2024222}}

The proof consists of three steps:

\smallskip

\textbf{\underline{Step I}: Bound on $\|\mathcal{L}_{\sym}-\bar{\mathcal{L}}_{\sym}\|_2$}

\smallskip
The idea to bound $\|\mathcal{L}_{\sym}-\bar{\mathcal{L}}_{\sym}\|_2$ is as follows:
{\begin{align*}
&\|\mathcal{L}_{\sym} - \bar{\mathcal{L}}_{\sym}\|_2 = \|\D^{-1/2}\W\D^{-1/2} - \bar{\D}^{-1/2}\bar{\W}\bar{\D}^{-1/2}\|_2\\
\leq&\  \|\D^{-1/2}\W\D^{-1/2} - \bar{\D}^{-1/2}\W\D^{-1/2}\|_2 \\
&+ \|\bar{\D}^{-1/2}\W\D^{-1/2}-\bar{\D}^{-1/2}\bar{\W}\D^{-1/2}\|_2 + \|\bar{\D}^{-1/2}\bar{\W}\D^{-1/2}-\bar{\D}^{-1/2}\bar{\W}\bar{\D}^{-1/2}\|_2\nonumber\\
\leq&\ \|\D^{-1/2}-\bar{\D}^{-1/2}\|_2\|\W\|_2\|\D^{-1/2}\|_2 \\
&+ \|\bar{\D}^{-1/2}\|_2\|\W-\bar\W\|_2\|{\D}^{-1/2}\|_2 + \|\bar\D^{-1/2}\|_2\|\bar\W\|_2\|\D^{-1/2}-\bar\D^{-1/2}\|_2\\
\stackrel{\eqref{random_graph_1}, \eqref{random_graph_2}}{\lesssim}&\sqrt{\frac{\xi_{n, +}^2\log n}{\xi_{n, -}^3}} + \sqrt{\frac{\xi_{n, +}\log n}{\xi_{n, -}^2}}\nonumber\\
\lesssim & \sqrt{\frac{\xi_{n, +}^2\log n}{\xi_{n, -}^3}}\nonumber.
\end{align*}}
In the rest of this section, we shall justify the penultimate step using matrix concentration. 

{For both the marginal MLE and QMLE, the $i$th diagonal entry of $\D$ can be written as a sum of independent uniformly bounded random variables:
\begin{align*}
	\sum_{\ms=2}^M\sum_{{\ \ T\in \Omega^{(\ms)}_{i}}}\mathbf 1_{\{T\in \E\}}\underbrace{\sum_{j\in T: j\neq i}z_{ij}(T)}_{=: z_i(T)\stackrel{\eqref{wenhai}}{\asymp} 1}
\end{align*}
where 
\begin{align}\label{wenhai}
z_{ij}(T) &= 
\begin{cases}
\frac{1}{|T|!}\sum_{r^{-1} = \pi\in\mathcal S(T)}\sum_{s\in [r(i)\wedge r(j)\wedge y_T]}\frac{\e{u^*_i}\e{u^*_{j}}}{(\sum_{t\geq s}\e{u^*_{\pi(t)}})^2} & \text{(marginal MLE)}\\
\\
\frac{\e{u_i}\e{u_{j}}}{(\e{u_i} + \e{u_j})^2}&\text{(QMLE)}
\end{cases}\\
& \asymp 1\quad\quad\quad\quad\quad\quad\quad\quad\quad\quad\text{(Assumptions \plainref{ap:2}-\plainref{ap:1})}\nonumber
\end{align}
and $1\leq y_T\leq |T|$ refers to the choice-$y_T$ observations on $T$ in the marginal MLE\footnote{For notational convenience, instead of defining $y_i$ only for the observed edges in $\E$, we assume each edge $T$ corresponds to a $y_T$. In particular, for the $i$th edge $T_i\in \E$, $y_i = y_{T_i}$.}. 
The expectation of $\D_{ii}$ (taken with respect to $\mathbf 1_{\{T\in \E\}}$) satisfies
\begin{align}
	(\log n)^3\leq\xi_{n, -}\lesssim \bar{\D}_{ii} = \sum_{\ms=2}^{M}\sum_{\ \ T \in \Omega^{(\ms)}_i}p_{T, n}^{(\ms)}z_{i}(T)\lesssim \xi_{n, +} \ \ \ \ \ i \in [n],\label{kuner}
\end{align}
where $p^{(\ms)}_{T,n}$ are defined in \plaineqref{pnqn}. 
}

By the Chernoff bound, with probability at least $1-n^{-3}$, all $D_{ii}$ are concentrated around their means 
\begin{align}
&|\D_{ii} - \bar{\D}_{ii}|\lesssim\sqrt{\bar{\D}_{ii}\log n}\stackrel{\eqref{kuner}}{=}o(\bar{\D}_{ii}) & i\in [n].\label{liuliu}
\end{align}
As a result,  
\begin{align}
\|\D^{-1/2} - \bar{\D}^{-1/2}\|_2 &= \max_{i\in [n]}|\D^{-1/2}_{ii} - \bar{\D}^{-1/2}_{ii}|\stackrel{\eqref{liuliu}}{\lesssim} \max_{i\in [n]}\left(\bar{\D}^{-3/2}_{ii}|\D_{ii} - \bar{\D}_{ii}|\right)\nonumber\\
&\stackrel{\eqref{kuner}, \eqref{liuliu}}{\lesssim}\sqrt{\frac{\log n}{\xi_{n ,-}^2}}\label{google01}\\
\|\D^{-1/2}\|_2&\stackrel{\eqref{liuliu}}{\lesssim} \|\bar{\D}^{-1/2}\|_2 \stackrel{\eqref{kuner}}{\lesssim} \sqrt{\frac{1}{\xi_{n, -}}}.\label{google02}
\end{align}
On the other hand, 
\begin{align}
\|\bar{\W}\|_2 &= \|\bar{\D}^{1/2}\bar{\D}^{-1/2}\bar{\W}\bar{\D}^{-1/2}\bar{\D}^{1/2}\|_2\leq \|\bar{\D}^{1/2}\|^2_2\|\bar{\D}^{-1/2}\bar{\W}\bar{\D}^{-1/2}\|_2\leq \|\bar{\D}\|_2\nonumber\\
&=\max_{i\in [n]}\bar{\D}_{ii}\stackrel{\eqref{kuner}}{\lesssim}\xi_{n, +},\label{google4}
\end{align}
where we used the fact that $\bar{\D}^{-1/2}\bar{\W}\bar{\D}^{-1/2}$ is symmetric and all its eigenvalues are contained in $[-1, 1]$. Combining \eqref{google01} and \eqref{google4}, we obtain
\begin{align}\label{random_graph_1}
	\|\D^{-1/2}- \bar{\D}^{-1/2}\|_2(\|\W\|_2\|\D^{-1/2}\|_2 + \|\bar\W\|_2\|\bar\D^{-1/2}\|_2) \lesssim \sqrt{\frac{\xi_{n, +}^2\log n}{\xi_{n, -}^3}}.
	\end{align} 
\color{black}

For $\W-\bar{\W}$, note that it can be written as a sum of independent symmetric matrices
\begin{align*}
	\W-\bar{\W} &= \sum_{\ms=2}^M\sum_{\ \ T \in \Omega^{(m)}}(\mathbf 1_{\{T\in \E\}} - p_{T, n}^{(\ms)}) \Lambda_T,
	\end{align*}
 where $[ \Lambda_{T}]_{ij} = \mathbf 1_{\{\{i, j\}\subseteq T\}}z_{ij}(T)$ when $i \neq j$ and $[\Lambda_{T}]_{ij} = 0$ when $i = j$.
Note that $\Lambda_{T}$ is symmetric and has at most $|T|^2$ nonzero entries. 
For all $T$ with $|T|=\ms\leq M$, there exists an absolute constant $C>0$ such that 
\begin{align*}
\|\mathbf 1_{\{T\in \E\}}\Lambda_{T}- p_{T, n}^{(\ms)}\Lambda_{T}\|_2 \leq \|\Lambda_{T}\|_2\leq \|\Lambda_{T}\|_F\lesssim M,
\end{align*}
and
\begin{align}
	\left\|\sum_{\ms=2}^M\sum_{\ \ T \in \Omega^{(m)}}\mathbb E\left[(\mathbf 1_{\{T\in \E\}} - p_{T, n}^{(\ms)})^2\Lambda^2_T\right]\right\|_2&\lesssim \sum_{\ms=2}^M\left\|\sum_{\ \ T \in \Omega^{(m)}}p_{T, n}^{(\ms)}(1-p_{T, n}^{(\ms)})\Lambda^2_{T}\right\|_2\nonumber\\
	&\lesssim \sum_{\ms=2}^Mq_{n}^{(\ms)}\left\|\sum_{\ \ T \in \Omega^{(m)}}\Lambda^2_{T}\right\|_2\label{hopea}.
\end{align} 
	To further bound the right-hand side, note for $i\neq j$, 
	\begin{align}
		\left|\left\{\sum_{\ \ T \in \Omega^{(m)}}\Lambda^2_{T}\right\}_{ij}\right| &= \left|\sum_{\ \ T \in \Omega_{ij}^{(m)}}[\Lambda^2_{T}]_{ij}\right| \leq \sum_{\ \ T \in \Omega_{ij}^{(m)}}\left|(\Lambda_{T}[:, i])^\top\Lambda_{T}[:, j]\right|\label{gshj1}\\
		&\lesssim \left|\Omega_{ij}^{(m)}\right|\lesssim n^{\ms-2}\nonumber, 
	\end{align} 
and
\begin{align*}
	\left\{\sum_{\ \ T \in \Omega^{(m)}}\Lambda^2_{T}\right\}_{ii} &= \sum_{\ \ T \in \Omega_i^{(m)}}[\Lambda^2_{T}]_{ii} = \sum_{\ \ T \in \Omega_i^{(m)}}\|\Lambda_T[:, i]\|_2^2\gtrsim \left|\Omega_i^{(m)}\right|\gtrsim n^{\ms-1},
\end{align*}
where $\Lambda_T[:, i]$ is the $i$th column of $\Lambda_T$. 
By the Gershgorin circle theorem, every eigenvalue of $(\sum_{T \in \Omega^{(m)}}\Lambda^2_{T})$ lies in some circle centered at one of its diagonal entries with radius equal to the $\ell_1$-norm of the remaining entries of that row. 
Since $(\sum_{T \in \Omega^{(m)}}\Lambda^2_{T})$ is also symmetric and nonnegative definite, 
proceeding with \eqref{hopea},  
\begin{align*}
	\left\|\sum_{\ms=2}^M\sum_{\ \ T \in \Omega^{(m)}}\mathbb E\left[(\mathbf 1_{\{T\in \E\}} - p_{T, n}^{(\ms)})^2\Lambda^2_T\right]\right\|_2&\lesssim \sum_{\ms=2}^Mq_{n}^{(\ms)}\left\|\sum_{\ \ T \in \Omega^{(m)}}\Lambda^2_{T}\right\|_2\\
	&\lesssim\sum_{\ms=2}^M q_{n}^{(\ms)}\left(n^{\ms-1} + (n-1)\cdot n^{\ms-2}\right)\\
	&\lesssim \xi_{n, +}.
\end{align*} 
By the matrix Bernstein inequality \citep{tropp2012user}, with probability at least $1-n^{-3}$, 
\begin{align}
\|\W-\bar{\W}\|_2 \lesssim\sqrt{\xi_{n, +}\log n}. \label{google03}
\end{align}
Combining \eqref{google02} and \eqref{google03} using a union bound, we obtain that, with probability at least $1-2n^{-3}$, 
	\begin{align}\label{random_graph_2}
	 \|\D^{-1}\|_2\|\W-\bar{\W}\|_2 \lesssim \sqrt{\frac{\xi_{n, +}\log n}{\xi_{n, -}^2}}.
\end{align}

\smallskip

\textbf{\underline{Step II}: Bound on $\bar{\mix}$}\label{ksjsjsjsk}

\smallskip

It suffices to lower bound $\lambda_2(\bar{\mathcal{L}}_{\sym})$ and $2-\lambda_n(\bar{\mathcal{L}}_{\sym})$, respectively, for which we appeal to Lemmas~\ref{chg1} and \ref{chg2}. 
Note that $\bar{\mathcal{L}}_{\sym}$ can be viewed as the normalized graph Laplacian of a weighted graph $\bar{\GG}$, where the weight between $i, j\in [n]$ is given by $\bar{\W}$, where 
\begin{align}
	\bar{\W}_{ij} = \sum_{\ms=2}^M \sum_{\ \ T \in \Omega_{ij}^{(m)}}p_{T, n}^{(\ms)}z_{ij}(T).\label{WWW}
\end{align}
For $\lambda_2(\bar{\mathcal{L}}_{\sym})$, \eqref{yonghuol} in Lemma~\ref{chg1} implies $\lambda_2(\bar{\mathcal{L}}_{\sym})>g^2_{\bar{\mathsf G}}/2$, and the lower bounded can be estimated using \eqref{WWW} as follows:
\begin{align}
	g_{\bar{\mathsf G}} &= \min_{U\subset [n]}\frac{\sum_{i\in U}\sum_{j\in U^\complement}\sum_{\ms=2}^M\sum_{T \in \Omega_{ij}^{(m)}}p_{T, n}^{(\ms)}z_{ij}(T)}{\min_{S\in \{U, U^\complement\}}\left\{\sum_{i\in S}\sum_{j\in [n]}\sum_{\ms=2}^M\sum_{T \in \Omega_{ij}^{(m)}}p_{T, n}^{(\ms)}z_{ij}(T)\right\}}\label{kouyiniao}\\
	&\stackrel{\eqref{wenhai}}{\gtrsim} \min_{U\subset [n]}\frac{|U|(n-|U|)\sum_{\ms=2}^M{n-2\choose \ms-2}p_n^{(\ms)}}{\min\{|U|, n-|U|) n\sum_{\ms=2}^M{n-2\choose \ms-2}q_n^{(\ms)}}\gtrsim\frac{\xi_{n,-}}{\xi_{n, +}}\nonumber,
\end{align}
where $\xi_{n,-}$ is defined in \plaineqref{myxis}. Therefore,
\begin{align}
\lambda_2(\bar{\mathcal{L}}_{\sym})\gtrsim\left(\frac{\xi_{n,-}}{\xi_{n,+}}\right)^2.\label{egen2}
\end{align}

To bound $2-\lambda_n(\bar{\mathcal{L}}_{\sym})$, we resort to Lemma~\ref{chg2} and obtain
\begin{align}
 2-\lambda_n(\bar{\mathcal{L}}_{\sym}) \geq 1-\sqrt{1-(1-g^{\ddd}_{\bar{\mathsf G}})^2}\geq \frac{(1-g^{\ddd}_{\bar{\mathsf G}})^2}{2},\label{kouyi}
\end{align}
where for the last inequality follows from the elementary inequality $1-\sqrt{1-x}\geq x/2$ for $0\leq x\leq 1$. 

To estimate $1-g^{\ddd}_{\bar{\mathsf G}}$, we assume $g^{\ddd}_{\bar{\mathsf G}}$ is attained at some disjoint sets $U_1, U_2\subset [n]$ with $|U_1|\geq |U_2|$. Then, 
\begin{align*}
g^{\ddd}_{\bar{\mathsf G}}&=\frac{2\sum_{i\in U_1, j\in U_2}\bar{\W}_{ij}}{2\sum_{i\in U_1, j\in U_2}\bar{\W}_{ij} + \sum_{i\in U_1, j\in U^\complement_2}\bar{\W}_{ij} + \sum_{i\in U^\complement_1, j\in U_2}\bar{\W}_{ij}}\\
&=\frac{1}{1 + \frac{\sum_{i\in U_1, j\in U^\complement_2}\bar{\W}_{ij} + \sum_{i\in U^\complement_1, j\in U_2}\bar{\W}_{ij}}{2\sum_{i\in U_1, j\in U_2}\bar{\W}_{ij}}}.
\end{align*}
By a similar computation as \eqref{kouyiniao}, 
\begin{align*}
\frac{\sum_{i\in U_1, j\in U^\complement_2}\bar{\W}_{ij} + \sum_{i\in U^\complement_1, j\in U_2}\bar{\W}_{ij}}{2\sum_{i\in U_1, j\in U_2}\bar{\W}_{ij}}&\gtrsim \frac{(|U_1||U_2^\complement| + |U_1^\complement||U_2|)\xi_{n,-}}{|U_1||U_2|\xi_{n,+}}
\geq \frac{|U_1||U_2^\complement|\xi_{n,-}}{|U_1||U_2|\xi_{n,+}} \geq \frac{\xi_{n,-}}{\xi_{n,+}}.
\end{align*}
Consequently, there exists an absolute constant $C>1$ such that   
\begin{align*}
g^\ddd_{\bar{\mathsf G}}\leq \frac{1}{1+\frac{\xi_{n, -}}{C\xi_{n,+}}}\Longrightarrow 1-g^\ddd_{\bar{\mathsf G}}\geq \frac{\xi_{n, -}}{2C\xi_{n,+}}.
\end{align*}
This combined with \eqref{kouyi} implies 
\begin{align}
2-\lambda_n(\bar{\mathcal{L}}_{\sym})\gtrsim \left(\frac{\xi_{n,-}}{\xi_{n,+}}\right)^2.\label{egenn}
\end{align}
Putting \eqref{egen2} and \eqref{egenn} together, 
\begin{align*}
\bar{\mix}\gtrsim \left(\frac{\xi_{n,-}}{\xi_{n,+}}\right)^2.
\end{align*}

\smallskip

\textbf{\underline{Step III}: Wrap-up}

\smallskip

Under the assumption $\xi_{n,-}^7/(\xi_{n,+}^6\log n)\to\infty$, the expected spectral gap dominates the fluctuation errors:
\begin{align*}
\sqrt{\frac{\xi_{n, +}^2\log n}{\xi_{n, -}^3}} = o\left(\frac{\xi^2_{n,-}}{\xi^2_{n, +}}\right). 
\end{align*}
Substituting this into \eqref{lexing} finishes the proof for \plaineqref{2024222}. 

To prove the lower bound on $\lambda_2^\leave$ in \plaineqref{2024222}, note 
\begin{align}
\lambda_2(-H^*(\bu^*)) = \lambda_2(\D^{1/2}\mathcal{L}_{\sym}\D^{1/2})\geq \lambda_2(\mathcal{L}_{\sym})\cdot\min_{i\in [n]}\D_{ii}\gtrsim\frac{\xi^3_{n,-}}{\xi^2_{n, +}}. \label{jhuhuj}
\end{align}
To bound $\lambda_2(-{H^*}^{(-k)}(\bu^*))$ in the leave-one-out analysis, notice that after removing $k$ and the corresponding edges, the remaining comparison remains a NURHM model with $n-1$ vertices. Consequently, \eqref{jhuhuj} holds for each $-{H^*}^{(-k)}(\bu^*)$ with probability at least $1-2n^{-3}$. Taking a union bound over $k\in [n]$ yields the desired result. 

\subsubsection{Proof of \plaineqref{2024333}}\label{jianlo}
We first consider the setting of NURHM and then discuss a refined estimate for HSBM. To begin with, note that under the assumptions $\xi_{n,-}\gtrsim\log n$ and $\zeta_{n,-}\gtrsim \log n$, the estimate \plaineqref{2024111} holds with probability at least $1-n^{-5}$. In NURHM, for $j\neq k\in [n]$, 
\begin{align*}
\N_{jk} = \sum_{m=2}^M\sum_{\ \ T \in \Omega^{(\ms)}_{jk}}  \underbrace{\mathbf 1_{\{T\in \E\}}}_{=:Y_T}, \quad\quad \mathbb E[Y_T] = p_{T, n}^{(\ms)}.
\end{align*}
is a sum of independent Bernoulli random variables, where $p^{(\ms)}_{T,n}$ are the edge probabilities defined in \plaineqref{pnqn}. Since we wish to obtain an upper bound on $\N_{jk}$, without loss of generality, we  assume $p^{(\ms)}_{T, n} = q_{n}^{(\ms)}$ as defined in \plaineqref{pnqn}. Under such circumstances, we can compute $\mathbb E[\N_{jk}]$ as 
\begin{align*}
\mathbb E[\N_{jk}] = \sum_{\ms=2}^M{n-2\choose \ms-2}q^{(\ms)}_{n}\leq \sum_{\ms=2}^M n^{\ms-2}q^{(\ms)}_n \stackrel{\plaineqref{myxis}}{=} \frac{\xi_{n, +}}{n}.
\end{align*}
We first consider the case where $\xi_{n,+}/n\geq 15 \log n$. It follows from the Chernoff bound that 
\begin{align}
\mathbb P\left(\N_{jk}\geq 2\mathbb E[\N_{jk}] \right)\leq \e{-\mathbb E[\N_{jk}]/3}\leq n^{-5}. \label{1843}
\end{align}
On the other hand, if $\xi_{n,+}/n<15\log n$, then we can construct a coupling $Y'_T$ of $Y_T$ such that $Y_T\leq Y'_T$ and $\sum_{\ms=2}^M \sum_{T \in \Omega^{(\ms)}_{jk}}  \mathbb E[Y'_{T}]= 15\log n$. Applying the estimate \eqref{1843} to $N'_{jk}:=\sum_{\ms=2}^M\sum_{T \in \Omega^{(\ms)}_{jk}} Y'_{T}$ and utilizing $Y_T\leq Y'_T$ shows that 
\begin{align}
\mathbb P\left(\N_{jk}\geq 30\log n \right)\leq \mathbb P\left(N'_{jk}\geq 30\log n \right)\leq n^{-5}.\label{1844}
\end{align}
Combining \eqref{1843}-\eqref{1844} and applying a union bound over all distinct pairs $(j, k)$ yields that, with probability at least $1-n^{-3}$, 
\begin{align*}
\max_{j\neq k\in [n]}\N_{jk}\lesssim \max\left\{\frac{\xi_{n,+}}{n}, \log n\right\}. 
\end{align*}
The proof is finished by combining this with \plaineqref{2024111} and the Borel--Cantelli lemma:
\begin{align}
\cc = \max_{j\neq k\in [n]}\frac{\N_{jk}}{\ndeg_j}\lesssim\max_{j\neq k\in [n]}\frac{\N_{jk}}{\xi_{n,-}}\lesssim\frac{\max\left\{n^{-1}{\xi_{n,+}}, \log n\right\}}{\xi_{n,-}}. \label{7463}
\end{align}
We finish the proof of Lemma \plainref{newtech}.

For HSBM, more refined estimates can be obtained for $\ndeg_j$ and $\N_{jk}$. Consider an HSBM with edge size $\mm\geq 2$ and $K$ communities $V_1, \ldots, V_K$. The edge probabilities within $V_1, \ldots, V_K$ and across them are $\omega_{n,1}, \ldots, \omega_{n,K}$, and $\omega_{n,0}$, respectively. Fix $i\in [K]$ and any object $j\in V_i$. By degree concentration, with probability at least $1-n^{-5}$, 
\begin{align}
\ndeg_j \gtrsim \mathbb E[\ndeg_j]&=\omega_{n,i}{|V_i|-1\choose \mm-1} + \omega_{n,0}\sum_{s=1}^{\mm-1}{n-|V_i|\choose s}{|V_i|-1\choose \mm-1-s}\nonumber\\
&\gtrsim  \max\left\{|V_i|^{\mm-1}\omega_{n,i}, (n-|V_i|)n^{\mm-2}\omega_{n,0}\right\}.\label{198234}
\end{align}
On the other hand, for each $k\in [n]\setminus\{j\}$, it follows from a similar concentration argument as in the NURHM setting that, with probability at least $1-n^{-5}$, 
\begin{align}
\N_{jk}&\lesssim \mathbb \max\left\{\mathbb E[\N_{jk}], \log n\right\}\lesssim\max\left\{|V_i|^{\mm-2}\omega_{n, i}, n^{\mm-2}\omega_{n, 0},\log n\right\},\label{198235}
\end{align}
where the last step follows from the observation
\begin{align*}
\mathbb E[\N_{jk}] = \begin{cases}
\omega_{n,0}{n-2\choose \mm-2}\lesssim \omega_{n,0} n^{\mm-2}& k\notin V_i\\
\\
\omega_{n,i}{|V_i|-2\choose \mm-2} + \omega_{n,0}\sum_{s=1}^{\mm-2}{n-|V_i|\choose s}{|V_i|-2\choose \mm-2-s}\lesssim |V_i|^{\mm-2}\omega_{n, i}+ n^{\mm-2}\omega_{n, 0}& k\in V_i
\end{cases}.
\end{align*}
Putting \eqref{198234} and \eqref{198235} together, we have 
\begin{align}
\frac{\N_{jk}}{\ndeg_j}&\lesssim\frac{\max\left\{|V_i|^{\mm-2}\omega_{n, i}, n^{\mm-2}\omega_{n, 0}\right\}}{\max\left\{|V_i|^{\mm-1}\omega_{n,i}, (n-|V_i|)n^{\mm-2}\omega_{n,0}\right\}}\nonumber\\
&\lesssim\max\left\{\frac{1}{|V_i|}, \frac{1}{n-|V_i|}\right\}\nonumber\\
&\lesssim\frac{1}{\min_{i\in [K]}|V_i|}.\label{kouqi}
\end{align}
Taking a union bound over $i, j, k$ and combining with \eqref{7463} (HSBM identified as a special case of NURHM) yields that, with probability at least $1-n^{-2}$,  
\begin{align*}
\cc \lesssim \min\left\{\frac{1}{\min_{i\in [K]}|V_i|}, \frac{\max\left\{n^{-1}{\zeta_{n,+}}, \log n\right\}}{\zeta_{n,-}}\right\}.
\end{align*}
When $\min_{i\in [K]}|V_i|\gtrsim n$, this becomes $\cc\lesssim \max\{n^{-1}, (\log n)/\zeta_{n,-}\}$.  

\subsection{Proof of Lemma~\plainref{lastone}}\label{new19}
We show how to lower bound the leave-one-out spectral gap $\lambda_2^{\leave}$ using $\lambda_2(\mathcal L_{\sym})$ and $\cc$. 
Recall that the weighted graph associated with $-\cH^*(\bu^*)$ is $\GG$. For each $k\in [n]$, denote the weighted graph associated with $-{\cH^*}^{(-k)}(\bu^*)$ as $\GG^{(-k)}$. By definition, the edge weights in $\GG$ and $\GG^{(-k)}$ are given by the off-diagonal entries of ${\cH^*}(\bu^*)$ and ${\cH^*}^{(-k)}(\bu^*)$, respectively.  

Under Assumptions~\plainref{ap:2}-\plainref{ap:1}, a direct computation using \plaineqref{wearediff} shows that 
\begin{align}
[\cH^*(\bu^*)]_{ij} \asymp \N_{ij},\label{giveup1}
\end{align}
where $\N_{ij}$ is defined in \plaineqref{mydij}. 
Moreover, for $i\neq j\in [n]\setminus\{k\}$, 
\begin{align}
[\cH^*(\bu^*)]_{ij}-[{\cH^*}^{(-k)}(\bu^*)]_{ij}\asymp \N_{ijk}:=\left|\{t: \{i, j, k\}\subseteq T_t\}\right|.\label{giveup2}
\end{align}

We now use Cheeger's inequality to obtain a lower bound on $\lambda_2(-{\cH^*}^{(-k)}(\bu^*))$. 
To this end, we first conduct a perturbation analysis to estimate $g_{\GG^{(-k)}}$. Suppose the Cheeger constant $g_{\GG^{(-k)}}$ is attained at some partition $U_1$ and $U_2$ of $[n]\setminus\{k\}$ ($U_2$ is the complement of $U_1$ in $[n]\setminus\{k\}$) with $\vol^{(-k)}(U_1)\leq \vol^{(-k)}(U_2)$, where $\vol^{(-k)}(\cdot)$ is the volume defined on $\GG^{(-k)}$; see \eqref{fc:cheeger}. Define $U'_2 = U_2\cup\{k\}$. According to the definition in \eqref{fc:cheeger},  
{\small\begin{align}
&g_{\GG} \leq \frac{\sum_{i\in U_1}\sum_{j\in U_2'}[\cH^*(\bu^*)]_{ij}}{\min\{\vol(U_1), \vol(U'_2)\}}\nonumber\\
& = \frac{\sum_{i\in U_1}\sum_{j\in U_2}[{\cH^*}^{(-k)}(\bu^*)]_{ij} + \sum_{i\in U_1}\sum_{j\in U_2}([\cH^*(\bu^*)]_{ij}-[{\cH^*}^{(-k)}(\bu^*)]_{ij})+\sum_{i\in U_1}[\cH^*(\bu^*)]_{ik}}{\min\{\vol(U_1), \vol(U'_2)\}}\nonumber\\
&\leq \frac{\sum_{i\in U_1}\sum_{j\in U_2}[{\cH^*}^{(-k)}(\bu^*)]_{ij}}{\min\{\vol^{(-k)}(U_1), \vol^{(-k)}(U_2)\}} + \frac{\sum_{i\in U_1}\sum_{j\in U_2}([\cH^*(\bu^*)]_{ij}-[{\cH^*}^{(-k)}(\bu^*)]_{ij})+\sum_{i\in U_1}[\cH^*(\bu^*)]_{ik}}{\min\{\vol(U_1), \vol(U'_2)\}}\nonumber\\
&= g_{\GG^{(-k)}} + \frac{\sum_{i\in U_1}\sum_{j\in U_2}([\cH^*(\bu^*)]_{ij}-[{\cH^*}^{(-k)}(\bu^*)]_{ij})+\sum_{i\in U_1}[\cH^*(\bu^*)]_{ik}}{\min\{\vol(U_1), \vol(U'_2)\}}.\label{uwe}
\end{align}}
To bound the second term in \eqref{uwe}, note
\begin{align}
\vol(U_1) - \vol^{(-k)}(U_1) &= \sum_{i\in U_1}\sum_{j\in [n]\setminus\{k\}, j\neq i}([\cH^*(\bu^*)]_{ij}-[{\cH^*}^{(-k)}(\bu^*)]_{ij}) + \sum_{i\in U_1}[\cH^*(\bu^*)]_{ik}\nonumber\\
&\stackrel{\eqref{giveup1}, \eqref{giveup2}}{\lesssim}\sum_{i\in U_1}\sum_{j\in [n]\setminus\{k\}, j\neq i} \N_{ijk} + \sum_{i\in U_1}\N_{ik}\nonumber\\
& \lesssim \sum_{i\in U_1}\N_{ik}\label{onefortwo}\\
&= \sum_{T\in \E: k\in T}|T\cap U_1|\lesssim \ndeg_k;\nonumber
\end{align}
and
\begin{align}
\vol(U'_2) - \vol^{(-k)}(U_2) &\geq \sum_{i\in U_1}[\cH^*(\bu^*)]_{ik}\gtrsim\ndeg_k.\nonumber
\end{align}
Combining these with $\vol^{(-k)}(U_1)\leq \vol^{(-k)}(U_2)$, we have 
\begin{align}
\vol(U_1)\lesssim \vol(U_2'). \label{u1u2}
\end{align}
Thus,
\begin{align}
&\frac{\sum_{i\in U_1}\sum_{j\in U_2}[\cH^*(\bu^*)-{\cH^*}^{(-k)}(\bu^*)]_{ij}+\sum_{i\in U_1}[\cH^*(\bu^*)]_{ik}}{\min\{\vol(U_1), \vol(U'_2)\}}\nonumber\\
\stackrel{\eqref{onefortwo}, \eqref{u1u2}}{\lesssim}&\ \frac{\sum_{i\in U_1}\N_{ik}}{\vol(U_1)} \lesssim \frac{\sum_{i\in U_1}\N_{ik}}{\sum_{i\in U_1}\ndeg_i}\leq\max_{i\in U_1}\frac{\N_{ik}}{\ndeg_i}\leq\cc. \label{bert}
\end{align}
Under the assumption $\cc/\lambda_2(\mathcal L_\sym)\to 0$ and applying Lemma~\ref{chg1}, 
\begin{align}
\frac{\cc}{g_{\GG}}\lesssim\frac{\cc}{\lambda_2(\mathcal L_\sym)}\to 0. \label{zero1}
\end{align}
Putting \eqref{uwe}, \eqref{bert}, and \eqref{zero1} together yields $g_{\GG^{(-k)}}\gtrsim g_{\GG}$. 

To finish the proof, denote by $\D^{(-k)}$ the degree matrix in $\GG^{(-k)}$. By another application of Lemma~\ref{chg1}, 
\begin{align*}
\lambda_2(-{\cH^*}^{(-k)}(\bu^*)) &= \lambda_2(\{\D^{(-k)}\}^{1/2}[\{\D^{(-k)}\}^{-1/2}\{-{\cH^*}^{(-k)}(\bu^*)\}\{\D^{(-k)}\}^{-1/2}]\{\D^{(-k)}\}^{1/2})\nonumber\\
&\gtrsim \min_{i\in [n]\setminus\{k\}}[\D^{(-k)}]_{ii}\lambda_2(\{\D^{(-k)}\}^{-1/2}\{-{\cH^*}^{(-k)}(\bu^*)\}\{\D^{(-k)}\}^{-1/2})\nonumber\\
&\gtrsim \min_{i\in [n]\setminus\{k\}}[\D^{(-k)}]_{ii}g^2_{\GG^{(-k)}}\nonumber\\
&\gtrsim \min_{i\in [n]\setminus\{k\}}[\D^{(-k)}]_{ii}g^2_{\GG}\nonumber\\
&\gtrsim \min_{i\in [n]\setminus\{k\}}[\D^{(-k)}]_{ii}\lambda^2_2(\mathcal L_\sym)\nonumber\\
&\gtrsim \N_{n, -}\lambda^2_2(\mathcal L_\sym),
\end{align*}
where the last step follows because 
\begin{align}
\max_{i\in [n]\setminus \{k\}}\frac{\D_{ii} - [{\D}^{(-k)}]_{ii}}{\D_{ii}} &= \max_{i\in [n]\setminus \{k\}}\frac{\sum_{j\in [n]\setminus \{k\}, j\neq i}([\cH^*(\bu^*)]_{ij}-[{\cH^*}^{(-k)}(\bu^*)]_{ij}) + [\cH^*(\bu^*)]_{ik}}{\D_{ii}}\nonumber\\
& \lesssim \max_{i\in [n]\setminus \{k\}}\frac{\N_{ik}}{\ndeg_i}\leq \cc\to 0. \label{louis}
\end{align}
The desired result follows by taking the infimum over $k\in [n]$.  We finish the proof of Lemma~\plainref{lastone}.

\bibliographystyle{imsart-nameyear}
\bibliography{ref} 
\end{document}